\documentclass[reqno, a4paper, 10pt]{amsart}
\usepackage{
  amsmath,
  amssymb,
  amsthm,
  geometry,
  paralist,
  mathtools,
  thmtools,
  tikz,
  tikz-cd,
  verbatim,
  relsize, %mathlarger
  bbm, %blackboard numbers
  dsfont,
  mathrsfs,
  comment      %\mathscr
}

\usepackage{adjustbox}

\usepackage[unicode, psdextra]{hyperref}
\hypersetup{
    colorlinks=true,
    linkcolor=blue,
    filecolor=magenta,      
    urlcolor=cyan,
    citecolor = blue,
	breaklinks=true
} 

\usepackage{cleveref} %cref

\usepackage[maxbibnames=99,
			maxalphanames=10, 
			style=alphabetic]{biblatex}
\usepackage[T1]{fontenc}
\usepackage{lmodern}
\usepackage[final]{microtype}

\usepackage{xurl}

\usepackage[textsize=scriptsize]{todonotes}
\presetkeys{todonotes}{inline}{}

\makeatletter
\providecommand\@dotsep{5}
\makeatother

\newcommand{\B}{\mathbb B}
\newcommand{\bbC}{\mathbb C}
\newcommand{\bbH}{\mathbb H}

\newcommand{\bbR}{\mathbb R}

\newcommand{\bbW}{\mathbb W}
\newcommand{\bbZ}{\mathbb Z}
\newcommand{\bbN}{\mathbb N}

\newcommand{\cA}{\mathcal A}

\newcommand{\cC}{\mathcal C}
\newcommand{\cD}{\mathcal D}
\newcommand{\cF}{\mathcal F}
\newcommand{\cG}{\mathcal G}
\newcommand{\cH}{\mathcal H}

\newcommand{\cP}{\mathcal P}

\newcommand{\cU}{\mathcal U}

\newcommand{\cZ}{\mathcal Z}

\newcommand{\catname}[1]{{\mathsf{#1}}}

\newcommand{\TLJ}{\catname{TLJ}}

\renewcommand{\mod}{\catname{mod}}

\renewcommand{\vec}{\catname{vec}}
\newcommand{\rep}{\catname{rep}}
\newcommand{\Fib}{\catname{Fib}}

\mathchardef\mhyphen="2D 
\newcommand{\lmod}{\mhyphen \mathsf{mod}}
\newcommand{\rmod}{\mathsf{mod} \mhyphen}
\newcommand{\bimod}{\mhyphen \mathsf{mod} \mhyphen}
\newcommand{\lprmod}{\mhyphen \mathsf{prmod}}
\newcommand{\rprmod}{\mathsf{prmod} \mhyphen}

\newcommand{\<}{\langle}
\renewcommand{\>}{\rangle}

\DeclareMathOperator{\id}{id}
\DeclareMathOperator{\Z}{\mathbb Z}
\DeclareMathOperator{\R}{\mathbb R}

\newcommand{\1}{\mathds{1}}

\DeclareMathOperator{\ra}{\rightarrow}
\newcommand{\xra}{\xrightarrow}

\DeclareMathOperator{\Aut}{Aut}

\DeclareMathOperator{\cone}{cone}
\DeclareMathOperator{\Hom}{Hom}
\DeclareMathOperator{\End}{End}
\DeclareMathOperator{\cEnd}{\mathcal{E}nd}

\DeclareMathOperator{\im}{\mathfrak {Im}}
\DeclareMathOperator{\img}{im}

\DeclareMathOperator{\Stab}{Stab}

\DeclareMathOperator{\Kom}{Kom}
\DeclareMathOperator{\Br}{Br}
\DeclareMathOperator{\PBr}{PBr}
\DeclareMathOperator{\Irr}{Irr}
\DeclareMathOperator{\reg}{reg}
\DeclareMathOperator{\ST}{ST}
\DeclareMathOperator{\FPdim}{FPdim}

\DeclareMathOperator{\re}{\mathfrak{Re}}

\declaretheorem[numberwithin=section]{theorem}
\declaretheorem[sibling=theorem]{lemma}
\declaretheorem[sibling=theorem]{corollary}
\declaretheorem[sibling=theorem]{conjecture}

\declaretheorem[sibling=theorem]{proposition}
\declaretheorem[sibling=theorem, style=remark]{remark}
\declaretheorem[sibling=theorem, style=definition]{definition}
\declaretheorem[sibling=theorem, style=definition]{example}

\newcommand{\nPi}[1][n]{{ \overset{\scriptscriptstyle #1}{\Pi}}}
\newcommand{\mPi}{{ \overset{\scriptscriptstyle m}{\Pi}}}

\newcommand{\bigboxtimes} {
    {\mathrel{\raisebox{-.6ex}
    	{$\mathlarger{\mathlarger{\mathlarger{\boxtimes}}} $}
    	}
    }
}

\newcommand{\bigboxtimesp}[2] {
    \underset{#2}{\overset{#1}{
    	{\mathrel{
    		\raisebox{-.6ex} 
    		{$\mathlarger{\mathlarger{\mathlarger{\mathlarger{\boxtimes}}}} $}
    		}
    	}
    }
    }
}

\newcommand{\zig}{\mathfrak{z}\mathfrak{i}\mathfrak{g}}

\renewcommand{\comment}[1]{}

\title[Stability conditions and Artin--Tits groups]{Stability conditions and Artin--Tits groups}

\author[]{Edmund Heng}
\address{Institut des Hautes \'{E}tudes Scientifiques (IH\'{E}S). Le Bois-Marie, 35, route de Chartres, 91440 Bures-sur-Yvette (France)}
\email{heng@ihes.fr}

\author[]{Anthony M. Licata}
\address{Mathematical Sciences Institute, Australian National University (ANU), Canberra (Australia)}
\email{anthony.licata@anu.edu.au}

\calclayout

\begin{document}

\begin{abstract}
We describe spaces of Bridgeland stability conditions on certain triangulated categories associated to Coxeter systems.  These categories are defined algebraically using the category of modules for zigzag algebras associated to Coxeter systems, which we construct as distinguished (quadratic, graded) algebra objects in fusion categories.  The resulting stability spaces are closely related to conjectural $K(\pi,1)$ spaces for Artin--Tits groups.
\end{abstract}

\maketitle

\setcounter{tocdepth}{1}

\tableofcontents

\section{Introduction}\label{sec:intro}
Group actions on triangulated categories appear prominently in higher representation theory.  A fundamental example is the action of a Kac--Moody braid group on the 2-Calabi--Yau category of a quiver, a category which can be constructed using representations of preprojective or zigzag algebras.  This example is important in part because the braid generators act by Seidel--Thomas spherical twists, which are categorical analogs of reflections.  As a result, there is a sense in which the 2-Calabi--Yau category of a quiver plays a similar role in the higher representation theory of a Kac--Moody braid group as the role played by the root lattice in the representation theory of a Kac--Moody Weyl group.  %Moreover, the complex manifold of Bridgeland stability conditions associated to the 2-Calabi--Yau category is closely related to a famous open problem known as the $K(\pi,1)$ conjecture.
The primary goal of this paper is to define a triangulated category which extends this example from Kac-Moody type to that of an arbitrary Coxeter system, and then to study the resulting space of Bridgeland stability conditions.

Let $\Gamma$ be a Coxeter graph with Coxeter group $\bbW := \bbW(\Gamma)$ and Artin--Tits group $\B:= \B(\Gamma)$. 
We first construct a complex manifold $\Upsilon_{\reg} := \Upsilon_{\reg}(\Gamma)$ with a free and proper action of $\bbW$ such that the fundamental group $\pi_1(\Upsilon_{\reg}/\bbW)$ is given by
\[
\pi_1(\Upsilon_{\reg}/\bbW) \cong 
\begin{cases}
\B, &\text{ if } \bbW \text{ is finite;}  \\
\bbZ \times \B, &\text{ if } \bbW \text{ is infinite.}
\end{cases}
\]
The manifold $\Upsilon_{\reg}$ is a mild variant of the complexified hyperplane complement $\Omega_{\reg}$ considered by Van der Lek \cite{VdL_thesis} and studied classically in group theory.  Roughly speaking, the space $\Upsilon_{\reg}$ contains the $\bbW$-orbits of both positive and negative fundamental chambers, whereas $\Omega_{\reg}$ contains only the orbit of the positive chamber.
In particular, when $\bbW$ is finite,  the negative chamber lies in the orbit of the positive chamber, and $\Upsilon_{\reg} = \Omega_{\reg}$ .
On the other hand, when $\bbW$ is infinite, there is a $\bbW$-equivariant homotopy equivalence $\Upsilon_{\reg} \simeq S^1 \times \Omega_{\reg}$, where the $S^1$ component represents a loop around hyperplanes associated to the imaginary cone of $\Gamma$ (see Section \ref{sec:titsimaginarycone} for the definition of the imaginary cone in the generality of arbitrary Coxeter systems).  

When the off-diagonal entries of the Coxeter matrix associated to $\Gamma$ satisfy $m_{s,t}\in \{2,3,\infty\}$, the Coxeter group $\bbW$ is the Weyl group of a symmetric Kac--Moody algebra.  In these cases, foundational work (in increasing generality) of Thomas \cite{thomas_2006}, Bridgeland \cite{bridgeland_2009}, and Ikeda \cite{ikeda2014stability} explains how the space $\Upsilon_{\reg}$ arises in representation theory via the study of stability conditions on an associated 2-Calabi--Yau category $\cD_\Gamma$.  
%The action of the Kac--Moody braid gorup on $\cD_\Gamma$ categorifies the action of the corresponding Wely group on the root lattice.
%These triangulated categories $\cD_\Gamma$, on which the corresponding Kac--Moody braid groups $\B$ act, categorify Tits representations of the corresponding Weyl groups $\bbW$.
%Algebraically, they are usually constructed from the category of (dg)modules for zigzag algebras, or using the quadratic dual preprojective algebras.
%The category $\cD_\Gamma$ is 2 Calabi-Yau, and comes equipped with an action of the corresponding Artin--Tits group $\B$, where the standard generators act by spherical twists.

% Ginzburg (dg-)algebras, contraction algebras etc.\ (along with their geometric counterparts) \cite{HueKho, bridgeland_2009, ikeda2014stability, Keller_deformedCY, ginzburg_CYalg, QW_18, AW_22}. \ed{I'm not quite sure how to best cite all the relevant papers; in particular the tilting functor point of view for preprojective algebras aren't mentioned here}
%In these Kac-$\Upsilon_{\reg}$ appears as the covering image of the complex manifold $\Stab(\cD_\Gamma)$ of Bridgeland stability conditions on $\cD_\Gamma$.
%More precisely, the local homeomorphism $\cZ: \Stab(\cD_\Gamma) \ra \Hom_\bbZ(\Lambda, \bbC)$ which sends a stability condition to its central charge restricts to a covering space map on a distinguished component of $\Stab(\cD_\Gamma)$, whose image is the hyperplane complement $\Upsilon_{\reg}$.

\subsection*{The main results}
We generalise the construction of $\cD_\Gamma$ to arbitrary Coxeter systems.
An important ingredient in this construction, which should be of independent interest, is the definition of zigzag algebras and preprojective algebras associated to such systems.
These algebras are constructed as objects in a certain fusion category $\cC(\Gamma)$ attached to $\Gamma$; %the construction of which is partly motivated by the earlier works of the first author with Elias \cite{heng2023coxeter, EH_fusionquiver} that studied represenations of quivers in fusion categories
the construction is motivated by the earlier works of the first author with Elias \cite{heng2023coxeter, EH_fusionquiver}, where representations of quivers in fusion categories were used to give a Coxeter-theoretic generalisation of Gabriel's theorem.
%We note that the fusion category $\cC(\Gamma)$ for $\Gamma$ associated to symmetric Kac--Moody systems is equivalent to the category of vector spaces $\vec$, and our algebras agree with the ordinary zigzag (resp.\ preprojective) algebras associated to the Kac--Moody diagrams.  
%The triangulated category $\cD_\Gamma$ is then constructed as the bounded homotopy category of (projective graded) modules over the zigzag algebra in $\cC(\Gamma)$.
The zigzag algebra we construct is non-negatively graded and Frobenius, with the trace of degree 2.  Just as in symmetric Kac--Moody type, dg modules over our zigzag algebra form a triangulated, 2-Calabi--Yau category $\cD(\Gamma)$, which we show carries an action of the associated Artin--Tits group $\B$ (whose generators act by a fusion-enhanced version of spherical twists).  
Notably, this category also carries a commuting action of the fusion category $\cC(\Gamma)$, where the action is non-trivial outside of symmetric Kac--Moody type.

%The upshot of our construction of $\cD_\Gamma$ is that the space $\Upsilon_{\reg}$ makes its appearance again via Bridgeland stability conditions -- now in the setting of arbitrary Coxeter systems.
%To explain ourselves, we note that an important consequence from our construction of the zigzag algebra -- being an algebra in $\cC(\Gamma)$ -- is that $\cD_\Gamma$ carries a natural action of the (underlying) fusion category $\cC(\Gamma)$.
The dimension of the entire space of stability conditions $\Stab(\cD(\Gamma))$ is typically larger than the rank of the underlying Coxeter system, which is the dimension of the space $\Upsilon_{\reg}$.  However, there is a closed, complex submanifold $\Stab_{\cC(\Gamma)}(\cD(\Gamma))$ of $\cC(\Gamma)$-equivariant stability conditions (see \cref{defn: C equivariant stab}), whose dimension is equal to the rank of the underlying Coxeter system.
%In turn, $\Stab(\cD_\Gamma)$ contains a subspace of ``$\cC(\Gamma)$-fixed points'', known as the subspace $\Stab_{\cC(\Gamma)}(\cD_\Gamma)$ of \emph{$\cC$-equivariant stability conditions}.
%This notion was first introduced in the first author's thesis \cite{Heng_PhDthesis} and was later shown, in joint work of both authors with Dell, that $\Stab_\cC(\cD)$ is a complex submanifold of $\Stab(\cD)$ for general triangulated categories $\cD$ equipped with actions of general fusion categories $\cC$ \cite{DHL_fusionstab}.
Echoing the situation in symmetric Kac--Moody type, we prove the following.
%(We will also describe the corresponding distinguished component of $\Stab(\cD(\Gamma))$ towards the end of the paper).  
\begin{theorem}[= \protect{\Cref{thm:maintheorem}}]\label{thm:intromain}
The moduli space of $\cC(\Gamma)$-equivariant stability conditions on $\cD(\Gamma)$ has a distinguished connected component $\Stab^\dagger_{\cC(\Gamma)}(\cD(\Gamma))$ which is a covering space of $\Upsilon_{\reg}$.
Moreover, the deck transformation action of $\pi_1(\Upsilon_{\reg}/\bbW)$ on the composite covering 
\[
\underline{\pi}: \Stab^\dagger_{\cC(\Gamma)}(\cD(\Gamma)) \twoheadrightarrow \Upsilon_{\reg} \twoheadrightarrow \Upsilon_{\reg}/\bbW
\] 
agrees with the actions of $\B$ and $[2\mathbb{Z}]$ by autoequivalences of $\cD(\Gamma)$.
\end{theorem}
In the above theorem, $[2\mathbb{Z}]$ is the subgroup of autoequivalences generated by the Serre functor $[2]$.
When $\Gamma$ corresponds to a symmetric Kac--Moody algebra, all stability conditions are automatically $\cC(\Gamma)$-equivariant, in which case our theorem recovers the main theorems of \cite{thomas_2006, bridgeland_2009, ikeda2014stability} (see \cref{cor:symKMthm}).
The action of $\cC(\Gamma)$ is crucial in all other cases, where $\Stab^\dagger_{\cC(\Gamma)}(\cD(\Gamma))$ is a positive codimension submanifold of the non-equivariant stability manifold.
%For symmetric Kac--Moody types, the theorem above was proven, in increasing generality, in the work of Thomas \cite{thomas_2006}, Bridgeland \cite{bridgeland_2009}, and Ikeda \cite{ikeda2014stability}.  
%The introduction of the action of $\cC(\Gamma)$ is crucial in the extension of their work to arbitrary Coxeter types.%, 
%as the correct submanifold $\Stab^\dagger_{\cC(\Gamma)}(\cD(\Gamma))$ corresponding to $\Upsilon_{\reg}$ is in general a proper submanifold of the entire space of stability conditions on $\cD(\Gamma)$.

One consequence of our theorem is that the faithfulness of the action of $\B$ on $\cD(\Gamma)$ is equivalent to the simply connectivity of the component $\Stab^\dagger_{\cC(\Gamma)}(\cD(\Gamma))$.  Faithfulness is known when $\bbW$ is a finite Coxeter group (see Section \ref{sec:conjandconsequence}).
It then follows from the $K(\pi,1)$ conjecture -- which is a theorem for finite $\bbW$ -- that the manifold $\Stab^\dagger_{\cC(\Gamma)}(\cD(\Gamma))$ is contractible when $\bbW$ is finite.  
In fact, we conjecture that $\Stab^\dagger_{\cC(\Gamma)}(\cD(\Gamma))$ is always contractible, though the current paper does not attempt to make any progress on that conjecture.  
Outside of finite Coxeter groups, the natural appearance of $\mathbb{Z} \times \B$ is also potentially of interest in light of the relationships between such groups and Garside groups developed in recent work of Haettel--Huang \cite{HH_garside}.

%one can show that $\B$ indeed acts faithfully
%one can show that the covering space $\Stab^\dagger_{\cC(\Gamma)}(\cD(\Gamma))$ is moreover universal 
%- a fact which we expect to hold generally (as was expected for all symmetric Kac--Moody types). 
%Moreover, it follows from the proven $K(\pi,1)$ theorem \cite{Deligne_72} for finite Coxeter groups that $\Stab^\dagger_{\cC(\Gamma)}(\cD(\Gamma))$ is indeed contractible in these cases (see also \cite{paolini_masterthesis} for a proof via discrete Morse theory).
%In fact, we further conjecture that $\Stab^\dagger_{\cC(\Gamma)}(\cD(\Gamma))$ is in general contractible; this is in line with the general expectation of the topology of $\Stab(\cD)$ -- and in part, the $K(\pi,1)$ conjecture.
%Nonetheless, the current paper does not attempt to make any progress on that question.

In fact the distinguished connected component $\Stab^\dagger(\cD(\Gamma)) \supseteq \Stab^\dagger_{\cC(\Gamma)}(\cD(\Gamma))$ of the entire (non-equivariant)
stability manifold also admits a Coxeter-theoretic description in terms of an unfolding of the Coxeter diagram $\Gamma$ (see \cref{defn:unfolding}).  We prove the following:

%We associate to every Coxeter graph $\Gamma$ a symmetric Kac--Moody type Coxeter graph $\check{\Gamma}$, which we call the \emph{unfolding} of $\Gamma$. \todo{Where has this notion appeared previously, other than Ed's paper?}
%\ed{Lusztig considered some of it in his paper ``Some examples of square integrable representations...''; my Coxeter quiver paper also cited a few other papers if we want more citations. Our (un)foldings are special cases of what's known as LCM homomorphisms, which played important roles in both the proofs of the Tits conjecture and the fact that Artin monoids inject into their groups; at least these are the reasons why Luis Paris care about these type of morphisms.}
%(If $\Gamma$ was symmetric Kac--Moody type to begin with, then $\Gamma = \check{\Gamma}$.)
%While $\cD(\Gamma)$ and $\cD_{\check{\Gamma}}$ carry the actions of different fusion categories, we prove that $\cD(\Gamma)$ is equivalent to $\cD_{\check{\Gamma}}$ as triangulated categories (i.e.\ ignoring the fusion categorical action).
%As a result, we obtain $\Stab^\dagger(\cD_{\Gamma}) \cong \Stab^\dagger(\cD_{\check{\Gamma}})$, whereby our previous theorem tells us that the latter is a covering space of $\Upsilon_{\reg}(\check{\Gamma})$.
%In addition, the inclusion $\Stab^\dagger_{\cC(\Gamma)}(\cD(\Gamma)) \subseteq \Stab^\dagger(\cD(\Gamma))$ translates to the local picture into an embedding of hyperplane complements $\Upsilon_{\reg}(\Gamma) \hookrightarrow \Upsilon_{\reg}(\check{\Gamma})$.

\begin{theorem}[= \protect{\Cref{cor:unfolding}}] \label{thm:introunfolding}
Let $\Gamma$ be a Coxeter graph and $\check{\Gamma}$ be its unfolding.
%Then there is a commutative diagram, with horizontal maps covering maps and vertical maps closed embeddings:
%\[
%\begin{tikzcd}
%\Stab_{\cC(\Gamma)}^\dagger(\cD(\Gamma)) \ar[d, hookrightarrow] \ar[r, twoheadrightarrow]
%	& 	\Upsilon_{\reg}(\Gamma) \ar[d, hookrightarrow] \ar[r, twoheadrightarrow]
%	& \Upsilon_{\reg}(\Gamma)/\bbW(\Gamma) \ar[d, hookrightarrow]
%	\\
%\Stab^\dagger(\cD(\Gamma)) \ar[r, twoheadrightarrow]
%	& \Upsilon_{\reg}(\check{\Gamma}) \ar[r, twoheadrightarrow]
%	& \Upsilon_{\reg}(\check{\Gamma})/\bbW(\check{\Gamma}).
%\end{tikzcd}
%\]
Then there is a commutative diagram, with vertical maps covering maps and horizontal maps closed embeddings:
\[
\begin{tikzcd}
\Stab_{\cC(\Gamma)}^\dagger(\cD(\Gamma)) \ar[r, hook, "\subseteq"] \ar[d, two heads]
	& \Stab^\dagger(\cD(\Gamma)) \ar[d, two heads]
	\\
\Upsilon_{\reg}(\Gamma)/\bbW(\Gamma) \ar[r, hook, "\underline{\cU}"]
	& \Upsilon_{\reg}(\check{\Gamma})/\bbW(\check{\Gamma}).
\end{tikzcd}
\]
\end{theorem}
The embeddings $\underline{\cU}$ of hyperplane complements obtained above are interesting in their own right (see \cref{fig:pentagonunfold} for an example), as hyperplane complements do not generally behave well under subspace considerations.
Furthermore, these embeddings turn out to be closely related to embeddings of Salvetti complexes associated to LCM-homomorphisms \cite[\S 3]{Crisp99}, and our theorem has interesting group theoretical consequences for such homomorphisms; see Section \ref{sec:LCM} for discussion.
An in-depth study of such embeddings will appear in upcoming work with Paris.
%More precisely, the embedding $\Upsilon_{\reg}(\Gamma) \hookrightarrow \Upsilon_{\reg}(\check{\Gamma})$ induces a homomorphism between the fundamental groups
% $$\iota_* : \pi_1(\Upsilon_{\reg}(\Gamma)/\bbW(\Gamma)) \ra \pi_1(\Upsilon_{\reg}(\check{\Gamma})/\bbW(\check{\Gamma})).$$%, where $\iota_*$ restricted to the Artin--Tits group component agrees with the LCM-homomorphism  induced by the folding of $\check{\Gamma}$ onto $\Gamma$.
%This group homomorphism is conjecturally injective, and is known to be injective when restricted to the positive Artin-Tits monoid.
%In particular, once one establishes the faithfulness of the Artin-Tits group actions on $\cD(\Gamma)$ and $\cD_{\check{\Gamma}}$, this conjectural injectivity will follow as a consequence.

\subsection*{Related work}
%Mapping class groups of surfaces have a $K(\pi,1)$ -- the Teichmuller space -- which is a moduli space of metrics on a surface.  
A stability condition on a triangulated category gives rise, via its mass functional, to a metric on the category (see, e.g.\ \cite{neeman_metric} for a discussion of metrics on triangulated categories).
Mapping class groups of surfaces have $K(\pi,1)$s which are constructed as moduli space of metrics on the surface, and our theorems above may be thought of as theorems in a somewhat analogous ``Teichmuller theory of Artin--Tits groups''.

Categorical actions of arbitrary Artin--Tits groups are ubiquitous in the mathematics associated to Soergel bimodules and the Hecke category, where the Artin--Tits group actions are given by Rouquier complexes \cite{rouquier_2006}, and also in the theory of categorified quantum groups initiated by Khovanov--Lauda and Rouquier \cite{KL_klralgI, KL_klralgII, rouquier_2kacmoody}.  The current work suggests that much of the mathematics of categorified quantum groups, which includes geometric structures such as Lusztig--Nakajima quiver varieties, might admit a fusion-enhanced generalisation to arbitrary Coxeter systems.

Our current approach to constructing the category $\cD(\Gamma)$ begins with a fusion category and constructs $\cD(\Gamma)$ from it.  Since the fusion category is present from the outset, this approach is particularly convenient for the study of fusion-equivariant structures.  However, our approach has the disadvantage of not making direct contact with the Hecke category, where stability conditions are somewhat less understood.  In forthcoming work with Libedinsky, we will construct our category $\cD(\Gamma)$ -- and the space of fusion-equivariant stability conditions on it -- directly from the Hecke category.  In that approach an important role is played by a categorification $\mathcal{J}_\bbW$ of Lusztig's asymptotic Hecke Algebra, whose monoidal structure is shown to be rigid and pivotal in \cite{EW_leftschetz}. Algebra objects in fusion categories  in relation to preprojective algebras also appeared earlier in \cite{MOV_quiver, cooper_phdthesis}, and more recently in \cite{GP_frobalgTL}.

Fusion-equivariant stability conditions have also appeared recently in related settings, see \cite{Heng_PhDthesis, QZ_fusion-stable, DHL_fusionstab}.  
Of note, \cite{QZ_fusion-stable} also study fusion-equivariant structures, both in relation to stability conditions and in relation to cluster-tilting theory.  In particular, they also 
prove a version of \cref{thm:intromain} in the case of finite type Coxeter systems.   In their approach, they study a space of fusion-equivariant stability conditions on a 2-Calabi--Yau category of finite type $ADE$, and show that considering fusion-equivariant structures has the effect of ``folding'' at the level of Coxeter diagrams and root systems.  As a result, our stability space constructions and theirs end up being essentially the same in finite types.  One benefit of the approach we take here is that our statements about spherical twist groups of arbitrary Coxeter type do not rely on the injectivity of LCM-homomorphisms between Artin--Tits groups (which is not known in general).
%One benefit of working  is intrinsic to $\Gamma$, e.g.\ our $\B$-action is defined via fusion-enhanced versions of spherical twists, which commutes with the $\cC(\Gamma)$-action from the outset; we also do not depend on the injectivity of LCM-homomorphisms, which is not generally known outside of finite-types.
%In fact, we prove results about LCM.
%I believe they use injectivity of LCM homomorphism to establish injectivity of ST groups.
%The result about faithfulness iff simply-connected in their setting will also depend on injectivity of LCM-homomorphisms.

\subsection*{Outline}
A more detailed outline of the paper is as follows.  
\begin{itemize}
\item The construction of $\Upsilon_{\reg}$ for arbitrary Coxeter system is given in Section \ref{sec:coximagconehyperplane}, and it generalises Ikeda's construction for symmetric Kac--Moody types \cite{ikeda2014stability}. This space is constructed using the notion of imaginary cones for Coxeter systems introduced in \cite{Hee1993} and studied in \cite{Dyer_imaginaryreflect, DHR_imaginarycone}.  The main result in Section \ref{sec:coximagconehyperplane} is stated in \cref{cor:homotopyequivalence}.  
\item Section \ref{sec:fusion1} begins with the background on fusion categories and fusion quivers, which we then use to define zigzag and preprojective algebras.
\item The construction of fusion categories associated to Coxeter systems and the corresponding zigzag algebras are given in Section \ref{zigzagCoxeter}, which also contains the definitions of various triangulated categories of interest.
\item In Section \ref{sec: action and categorification}, we introduce the Artin--Tits group action on $\cD(\Gamma)$, and we study its decategorification in Section \ref{sec:Grothendieckandlattice}.
\item Finally, in Section \ref{sec:stability}, we introduce the space of fusion-equivariant stability conditions and prove our main \cref{thm:intromain} (= \cref{thm:maintheorem}).  Our proof strategy largely follows that of Ikeda in \cite{ikeda2014stability}, though along the way we correct some errors and fill a significant gap in loc.\ cit.\  (see e.g.\ \cref{rem:Ikeda}).  
\item Section \ref{sec:unfolding} introduces unfolding and proves \cref{thm:introunfolding} (= \cref{cor:unfolding}). Relations to LCM-homomorphisms are discussed at the end of this section.
\end{itemize}

\subsection*{Notation and conventions}
Throughout all categories considered are essentially small and linear over an algebraically closed field.  
%Our constructions involving hyperplane complements, on the other hand,  work$\bbR$.
All subcategories we consider are taken to be full.  An additive subcategory will be assumed closed under taking direct summands; in particular any additive subcategory of an abelian category is automatically idempotent complete.  By a \emph{locally finite} abelian category we mean a finite-length abelian category with finite-dimensional hom spaces; it is moreover \emph{finite} if it contains only finitely many isomorphism classes of simple objects and contains enough projectives (so that it is equivalent to the category of finitely-generated modules over a finite-dimensional algebra).
A set of representatives of isomorphism classes of simple objects in an abelian category $\cA$ will be denoted by $\Irr(\cA)$.

Given a finite abelian category $\cA$, a ($\bbZ$-)\emph{graded object} $M$ of $\cA$ is an object $M \in \cA$ together with a decomposition $M = \bigoplus_{j \in \bbZ} M_i \in \cA$, with all but finitely many $M_i = 0$.  The \emph{category of graded objects of $\cA$}, denoted by $g\cA$, is the category whose objects are graded objects of $\cA$ and whose morphisms are the grading preserving maps.
Thus $g\cA$ is a locally finite abelian category.
%In the main text, two similarly named but different notions will appear, that of a module category over a fusion category, and that of the category of modules over an algebra in a fusion category.
%Whilst related, they should not be confused with one another.

\subsection*{Acknowledgements}
We would like to thank Arend Bayer, Asilata Bapat, Rachael Boyd, Tom Bridgeland, Hannah Dell, Anand Deopurkar, Ben Elias, Joe Grant, Ailsa Keating, Bernhard Keller, Nico Libedinsky, Giovanni Paolini, Luis Paris, Daniel Tubbenhauer,  Michael Weymss, Geordie Williamson, Oded Yacobi, and Yu Qiu for useful conversations.  We also thank the Oberwolfach Research Institute for Mathematics and the Institute for Computational and Experimental Research in Mathematics (ICERM), where portions of this work were carried out.  A.L. acknowledges support from France Australia Mathematical Sciences and Interactions and the Australian Research Council.

\section{Root systems, imaginary cones and hyperplane complements} \label{sec:coximagconehyperplane}
This section contains basic information about root systems and imaginary cones associated to Coxeter groups.  

\begin{definition}
Let $\Gamma_0$ be a finite set.
A \emph{Coxeter matrix} is a symmetric square matrix $(m_{s,t})_{s,t \in \Gamma_0}$ with diagonal $m_{s,s} = 1$ for all $s\in \Gamma_0$ and $m_{s,t} =m_{t,s} \in \{2,3,..., \infty\}$ for $s \neq t$. 
The corresponding \emph{Coxeter graph} $\Gamma$ of a Coxeter matrix $(m_{s,t})_{s,t \in \Gamma_0}$ is the simplicial graph with set of vertices $\Gamma_0$ and each pair of vertices $s$ and $t$ with $m_{s,t} \geq 3$ is connected by an edge labelled by $m_{s,t}$.  When $m_{s,t} = 2$, the vertices $s$ and $t$ are not connected by an edge.
The set of edges of $\Gamma$ is denoted by $\Gamma_1$.
The \emph{Coxeter group} of a Coxeter graph $\Gamma$ is the group with the group presentation:
\[
\bbW(\Gamma) = \< s \in \Gamma_0  \mid s^2 = 1, (st)^{m_{s,t}} = 1\>.
\]
By convention, when $m_{s,t} = \infty$ there are no relation between $s$ and $t$.
The pair $(\bbW(\Gamma),\Gamma_0)$ is known as a \emph{Coxeter system}.
%
%\W( \Gamma ) = \< s_i \text{ for each } i \in I  | s_i^2 = 1, \underbrace{s_i s_j s_i ...}_{m_{i,j} \text{ times}} = \underbrace{s_j s_i s_j ...}_{m_{i,j} \text{ times}} \>.

\noindent
The \emph{Artin--Tits group} $\B(\Gamma)W$ is the group with the group presentation:
\[
\B(\Gamma) = \< \sigma_s \text{ for each } s \in \Gamma_0  \mid \underbrace{\sigma_s \sigma_t \sigma_s ...}_{m_{s,t} \text{ times}} = \underbrace{\sigma_t \sigma_s \sigma_t ...}_{m_{s,t} \text{ times}} \>.
\]
As before, there are no relation between $\sigma_s$ and $\sigma_t$ when $m_{s,t} = \infty$.
\end{definition}
Recall that when $\bbW(\Gamma)$ is a finite group, we say that $\Gamma$ is \emph{finite-type}.
The irreducible finite Coxeter groups are classified by the Coxeter--Dynkin diagrams (A,B=C,D,E,F,G,H,I).
We will say that a Coxeter graph is \emph{symmetric Kac--Moody type} if $m_{s,t}\in \{2,3,\infty\}$ for all pairs $s\neq t\in \Gamma_0$.

\subsection{The Tits cone and imaginary cone} \label{sec:titsimaginarycone}
%To each Coxeter graph $\Gamma$, consider the lattice
%\begin{equation} \label{eq:fusionlattice}
%\Lambda := \bigoplus_{s \in \Gamma_0} K_0(\cC) \{\alpha_s\} \cong \bigoplus_{s \in \Gamma_0, A \in \Irr(\cC)} \Z \{A \cdot \alpha_s\}
%\end{equation}
%Recall that each fusion ring as an associated ring morphism into $\bbR$ given by taking the Frobenius--Perron dimension $\FPdim: K_0(\cC) \ra \bbR$.
%We shall use this morphism to endow $\bbR$ with the structure of a $K_0(\cC)$-module.
%As such we have a real vector space obtained from $\Lambda$ via scalar extension:

To each Coxeter graph $\Gamma$ associated the real vector space
\begin{equation} \label{eq:realvec}
\bbR\Lambda := \bigoplus_{s \in \Gamma_0} \bbR\{\alpha_s\}.
\end{equation}
We denote by $\bbR_{\geq 0} \Lambda := \bigoplus_{s \in \Gamma_0} \bbR_{\geq 0}\{\alpha_s\}$  the convex cone non-negatively spanned by all the $\alpha_s$.
We equip $\bbR\Lambda$ with a symmetric bilinear form $B(-,-): \bbR\Lambda \times \bbR\Lambda \ra \bbR$
\[
(\alpha_s, \alpha_t) := -2\cos(\pi/m_{s,t}),
\]
where by convention $-2\cos(\pi/\infty) := -2$.

The Coxeter group $\bbW := \bbW(\Gamma)$ acts linearly on $\bbR\Lambda$, with the generator $s \in \Gamma_0$ acting via 
\[
s\cdot v := v - B(\alpha_s, v)\alpha_s.
\]
This representation is known as the Tits representation (or the symmetric geometric representation) of $\bbW$, and it is a fundamental result in Coxeter theory that it is a faithful representation of $\bbW$. 
\begin{definition}
The set of \emph{simple roots} is given by $\{ \alpha_s \}_{s \in \Gamma_0}$, and the set of \emph{real roots} $\Phi$ is defined as the $\bbW$-orbit of the set of simple roots.
The set of \emph{positive real roots} is defined as $\Phi^+ := \Phi \cap \bbR_{\geq 0}\Lambda$ and the set of negative roots is defined as $\Phi^- := -\Phi^+$.
Note that the set of real roots $\Phi$ decomposes into $\Phi^+ \sqcup \Phi^-$.
\end{definition}

\begin{remark}
What we have called ``real roots" are sometimes simply called  ``roots" elsewhere in the Coxeter theory literature;  we use the adjective ``real'' to distinguish them from ``imaginary roots," which play an important role in Kac--Moody root systems and which are not in the $\bbW$ orbit of the simple roots.
\end{remark}

Following Dyer \cite{Dyer_imaginaryreflect}, we define the closed imaginary cone associated to $\Gamma$:
\begin{definition}
Let $K:= \{ v \in \bbR_{\geq 0}\Lambda \mid B(v,\alpha_s) \leq 0 \text{ for all } s \in \Gamma_0\}$.
The \emph{closed imaginary cone} $I$ is defined to be the closure of the $\bbW$-orbit on $K$
\[
I := \overline{\bbW\cdot K} \subseteq \bbR\Lambda.
\]
\end{definition}

\begin{proposition}[\protect{\cite{Dyer_imaginaryreflect}}] \label{prop:imagcone}
The closed imaginary cone $I$ satisfies the following properties:
\begin{enumerate}
\item $I = \{0\}$ if and only if the Coxeter group $\bbW$ is finite. \label{item:imagzeroifffinite}
\item $I \subseteq \bbR_{\geq 0}\Lambda$. \label{item:imagpositivecone}
\item When $I\neq \{0\}$, $I$ agrees with the convex hull of the limit rays of the set of rays $\{\bbR_{\geq 0}\alpha \mid \alpha \in \Phi^+ \}$. \label{item:imagconvexlimit}
\end{enumerate}
In particular, $I$ is a closed convex cone in $\bbR_{\geq 0}\Lambda$.
\end{proposition}
\begin{proof}
\eqref{item:imagzeroifffinite} follows directly from the fact that the bilinear form $B(-,-)$ is positive definite if and only $\bbW$ is finite.
\eqref{item:imagpositivecone} follows from the fact that
\[
s\cdot v - v = -B(\alpha_s, v)\alpha_s
\]
for all Coxeter generators $s$ (see also proof of \cite[Proposition 3.2]{Dyer_imaginaryreflect}).
\eqref{item:imagconvexlimit} is \cite[Theorem 5.4]{Dyer_imaginaryreflect}; see also \cite[Section 12]{Dyer_imaginaryreflect}.
\end{proof}

Let $J \subset \Gamma_0$ and consider the standard parabolic subgroup $\bbW_J \subseteq \bbW$ associated to $J$.
We view $\bbR\Lambda_J := \bigoplus_{s' \in J} \bbR\alpha_{s'}$ as a subspace of $\bbR\Lambda$ with the restricted bilinear form, where the positive real roots $\Phi^+_J$ and imaginary cone $I_J$ associated to $\bbW_J$ are then subsets of $\Phi^+$ and $I$ respectively.
The following lemma is proven by Dyer (the final statement is contained in the proof of the lemma):
\begin{lemma}[{\protect{\cite[Lemma 3.3]{Dyer_imaginaryreflect}}}] \label{lem:imagcontain}
Let $J \subseteq \Gamma_0$. 
We have $I_J \subseteq I$.
Moreover, if $v \in I$ such that $v = \sum_{s' \in J \subseteq \Gamma_0} k_{s'} \alpha_{s'}$ with all $k_{s'} \neq 0$, then $v \in I_J$.
\end{lemma}

Let $\Bbbk$ be either $\bbR$ or $\bbC$, where we regard $\bbC = \bbR \oplus i\bbR$ as a real vector space.
Denote
\begin{equation}\label{eqn:dualnotation}
(\bbR\Lambda)^*_\Bbbk := \Hom_\bbR(\bbR\Lambda, \Bbbk).
\end{equation}
Note that the action of $\bbW$ on $\bbR\Lambda$ induces the contragradient action on $(\bbR\Lambda)^*_\Bbbk$ with
\[
(w\cdot Z)(v) = Z(w^{-1}\cdot v), \qquad \text{for each } w\in \bbW.
\]
\begin{definition}\label{defn:titscone}
Let $C_\bbR := \{ Z \in (\bbR\Lambda)^*_\bbR \mid Z(\alpha_s) >0 \text{ for all } s \in \Gamma_0 \}$ with closure $\overline{C}_\bbR = \{ Z \in (\bbR\Lambda)^*_\bbR \mid Z(\alpha_s) \geq 0 \text{ for all } s \in \Gamma_0 \}$.
%Consider the $\bbW$-orbit space $\overline{T}_\bbR := \bbW \cdot \overline{C}_\bbR$.
The (real) \emph{Tits cone} is defined as the $\bbW$-orbit of $\overline{C}_\bbR$
\[
T_\bbR := \bbW \cdot \overline{C}_\bbR.
\]
The \emph{interior} of the Tits cone is denoted by $T_\bbR^{\circ}$.
\end{definition}
\begin{remark}
Note that sometimes the interior $T_\bbR^\circ$ of $T_\bbR$ is also called the Tits cone. 
Here we choose to work with the convention that the orbit space is itself the Tits cone.
\end{remark}

The following statements are classical and can be found in \cite{Vinberg_discretereflection, VdL_thesis, Bourbaki_456}; see also the survey \cite{Paris_Kpi1survey}.
\begin{proposition} \label{prop:titscone}
The following statements hold:
\begin{enumerate}
\item $T_\bbR$ is a convex cone.
\item $T_\bbR^\circ$ is non-empty and $\bbW$ acts on $T_\bbR^\circ$ properly discontinuously; namely for all $Z \in T_\bbR$, there exists a open neighbourhood $U_Z$ of $Z$ such that $w \cdot U_Z \cap U_Z \neq \emptyset$ if and only if $w \cdot Z = Z$.
\item For each $Z \in T_\bbR$, its stabilisers $\bbW_Z := \{ w \in \bbW \mid w\cdot Z = Z\}$ is not $\{\id\}$ if and only if $Z \in H_{\alpha, \bbR} := \{ Z \in T_\bbR \mid Z(\alpha) = 0\}$ for some $\alpha \in \Phi^+$.
\item $Z \in T_\bbR^\circ$ if and only if its stabilisers $\bbW_Z$ is finite. In particular, $Z$ is in $\overline{C}_\bbR \cap T_\bbR^\circ$ if and only if the set $\{ s \in \Gamma_0 \mid Z(\alpha_s) = 0\}$ generates a finite parabolic subgroup of $\bbW$.
\item $T_\bbR^\circ = T_\bbR = (\bbR\Lambda)^*_\bbR$ if and only if $\bbW$ is a finite Coxeter group.
\end{enumerate}
As such, the set of hyperplanes $H_{\alpha,\bbR}$ is locally-finite in $T_\bbR$.
Moreover, the action of $\bbW$ on the (real) hyperplane complement $T_\bbR \setminus \cup_{\alpha \in \Phi^+} H_{\alpha,\bbR}$ is free and properly discontinuous, with fundamental domain given by $C_\bbR$.
\end{proposition}

The following relation between the Tits cone and closed imaginary cone is well-known in the setting of root systems of Kac--Moody algebras.
We provide here the corresponding variant in the general setting of Coxeter systems. 
\begin{proposition}\label{prop:titscone=positiveimagcone}
Let $Z \in (\bbR\Lambda)^*_\bbR$. 
The following conditions on $Z$ are equivalent:
\begin{enumerate}
\item $Z \in T_\bbR = \bbW\cdot \overline{C}_\bbR$; and \label{item:intits}
\item $Z(\alpha) < 0$ only for finitely many positive roots $\alpha \in \Phi^+$. \label{item:finitenegatives}
\end{enumerate}
Moreover, we have that
\begin{equation} \label{eqn:interiortitsimag}
T_\bbR^\circ = \{Z \in (\bbR\Lambda)^*_\bbR \mid Z(v) > 0 \text{ for all } v \in I \setminus \{0\} \}.
\end{equation}
\end{proposition}
\begin{proof}
The proof of the equivalence between conditions \eqref{item:intits} and \eqref{item:finitenegatives} is essentially the same as in the setting of Kac--Moody algebras, which we reproduce for the convenience of the reader; see for example \cite[Proposition 3.4.1]{Kleshchev_lectureLie}.
Throughout the proof, for any $Z \in (\bbR\Lambda)^*_\bbR$ set 
\[
R_Z := \{ \alpha \in \Phi^+ \mid Z(\alpha) < 0 \}.
\]

We shall first show that condition \eqref{item:finitenegatives} is invariant under the $\bbW$-action.
Suppose $Z$ satisfies condition \eqref{item:finitenegatives}, so that $R_Z$ is finite.
Given $w \in \bbW$, we have that $R_{w\cdot Z} = \{ \alpha \in \Phi^+ \mid Z( w^{-1} \cdot \alpha) < 0 \}$.
Since any $w \in \bbW$ only sends finitely many positive roots to negative (this number is equal to the length of $w$), we have that $R_{w\cdot Z}$ is finite if and only if $R_Z$ is finite, as required.

We now prove the equivalence.
Let $Z \in \overline{C}_\bbR \subseteq T_\bbR$. 
Then $R_Z$ is in fact empty and so $Z$ satisfies \eqref{item:finitenegatives}.
Since condition \eqref{item:finitenegatives} is invariant under the $\bbW$-action, this proves that \eqref{item:intits} $\implies$ \eqref{item:finitenegatives}.
Now let that $Z$ be such that $R_Z$ is finite.
If $R_Z$ is empty, then in particular, $Z(\alpha_s) \geq 0$ for all simple roots $\alpha_s$, so $Z \in \overline{C}_\bbR$.
Otherwise, $R_Z$ is non-empty and it must at least contain one of the simple root $\alpha_s$.
Since $s \cdot \Phi^+ \setminus \{\alpha_s\} = \Phi^+ \setminus \{\alpha_s\}$, we see that $|R_{s\cdot Z}| = |R_Z| - 1 < |R_Z|$.
By inducting on the size of $R_Z$, $|R_{w\cdot Z}|$ will eventually be empty for some $w \in \bbW$, and hence $w \cdot Z \in \overline{C}_\bbR$.
This shows that $Z \in \bbW\cdot \overline{C}_\bbR$ as required.

Now we proceed to prove the equality in \eqref{eqn:interiortitsimag}.
To simplify notation let us denote 
\[
T' := \{Z \in (\bbR\Lambda)^*_\bbR \mid Z(v) > 0 \text{ for all } v \in I \setminus \{0\} \}
\]
(the RHS of \eqref{eqn:interiortitsimag}).
The case where $I = \{0\}$, equivalently $\bbW$ is finite by \cref{prop:imagcone}, is trivial since the condition defining $T'$ is superfluous and \cref{prop:titscone} says that $T_\bbR^\circ$ is just the whole space $(\bbR\Lambda)^*_\bbR$.
Thus from here on we will assume that $I \neq \{0\}$.

To show the containment $\subseteq$, first note that $T'$ is closed under the action of $\bbW$, since $\bbW(I) = I$ by construction and $\bbW\cdot \{0\} = \{0\}$.
It is therefore sufficient to show that any $Z \in \overline{C}_\bbR \cap T_\bbR^\circ$ satisfies $Z(v) > 0$ for all $v \in I \setminus \{0\}$, as $\overline{C}_\bbR$ is a fundamental domain of $T_\bbR$.
So consider $Z \in \overline{C}_\bbR \cap T_\bbR^\circ$ arbitrary. 
Being in $\overline{C}_\bbR$ gives us $Z(\alpha_s) \geq 0$ for all simple roots $\alpha_s$, and so $Z(v)\geq 0$ for all $v \in I$ ($v \in \bbR_{\geq 0}\Lambda$ by \cref{prop:imagcone}).
Suppose to the contrary that there exists $v \in I \setminus \{0\}$ such that $Z(v) = 0$.
Consider $J \subseteq \Gamma_0$ so that
\[
v = \sum_{s \in J \subseteq \Gamma_0} k_s \alpha_s \in \bbR_{\geq 0}\Lambda_J \subseteq \bbR_{\geq 0}\Lambda, \qquad k_s > 0 \text{ for all } s \in J.
\]
By \cref{lem:imagcontain}, we get that $v \neq 0$ is also in the closed imaginary cone $I_J$ associated to $\bbW_J$.
In particular, $\bbW_J$ can not be finite (\cref{prop:imagcone}).
However, $Z(v) = 0$ (and $Z(\alpha_s) \geq 0$ for all $s$) implies that $Z(\alpha_s)=0$ for all $s \in J$.
By \cref{prop:titscone}, this contradicts the fact that $Z$ is also in $T_\bbR^\circ$, as all the infinitely many elements in $\bbW_J$ would be stabilisers of $Z$.
This proves that $Z \in T'$ as required.

Conversely, let $Z \in T'$.
We first show that $Z$ satisfies the condition \eqref{item:finitenegatives} of the equivalence above, which implies that $Z$ is at least in the Tits cone $T_\bbR$.
Consider the compact subset $E := \bbR_{\geq 0}\Lambda \cap \{v \in \bbR\Lambda \mid ||v||=1\}$ (here $||\cdot ||$ is the standard norm on real vector spaces).
Note that every ray $\bbR_{\geq 0}v$ of a non-zero vector $v$ in $\bbR_{\geq 0}\Lambda$ intersects $E$ at exactly one point; we shall denote this (normalised) element by $\hat{v}$.
Suppose to the contrary that we have infinitely many $\beta \in \Phi^+$ such that $Z(\beta)<0$.
Then by compactness of $E$, we have a sequence of positive roots $(\beta_j)_{j=1}^\infty$, all satisfying $Z(\beta_j)<0$, such that the sequence $(\hat{\beta}_j)_{j=1}^\infty$ in $E$ converges to some point $\hat{b} \in E$.
By \cref{prop:imagcone}, the accumulation points of the set $\{\hat{\beta} \in E \mid \beta \in \Phi^+ \}$ are all contained in $I \cap E$, so in particular $\hat{b} \in I$.
However, $Z(\beta_j) <0 \implies Z(\hat{\beta}_j) <0$, so the point $\hat{b}$ that the sequence converges to must satisfy $Z(\hat{b}) \leq 0$, which contradicts the fact that $Z \in T'$.
It follows that there are only finitely many $\alpha \in \Phi^+$ such that $Z(\alpha)<0$, so $Z$ is in the Tits cone $T_\bbR$ by the equivalence above.

It remains to show that $Z$ is moreover in the interior $T_\bbR^\circ$.
Recall that $T'$ is closed under the $\bbW$-action.
Since we showed that $Z \in T_\bbR = \bbW\cdot \overline{C}_\bbR$, we can assume without lost of generality that $Z \in \overline{C}_\bbR$.
By \cref{prop:titscone}, it is sufficient to show that the set 
\[
J = \{ s \in \Gamma_0 \mid Z(\alpha_s) = 0 \}
\]
generates a finite parabolic subgroup $\bbW_J$.
Using \cref{prop:imagcone}, this is equivalent to showing that the closed imaginary cone $I_J$ of $\bbW_J$ contains only the zero element.
Let $v$ be an arbitrary element in $I_J$.
By \cref{lem:imagcontain}, it is also an element in the closed imaginary cone $I$.
By the construction of $J$, we have that $Z(v) = \sum_{s' \in J} k_{s'} Z(\alpha_{s'}) = 0$, so it must be that $v$ is also zero; otherwise we would have found a non-zero $v \in I$ with $Z(v)=0$, contradicting the fact that $Z \in T'$.
This completes the proof of the of equality of \eqref{eqn:interiortitsimag}.
\end{proof}

\begin{remark}\label{Ikeda fix}
Let $U_{I\setminus \{0\}, \bbR} := \{Z \mid Z(v) = 0 \text{ for some } v \in I \setminus \{0\} \}$ (which is empty if $\bbW$ is finite).
Equation \eqref{eqn:interiortitsimag} also implies that the space $(\bbR\Lambda)^*_\bbR$ decomposes as follows:
\[
(\bbR\Lambda)^*_\bbR = 
	\begin{cases}
	T_\bbR^\circ, &\text{ if } \bbW \text{ is finite, equivalently } I = \{0\}; \\
	T_\bbR^\circ \sqcup 
		 U_{I\setminus \{0\}, \bbR} \sqcup 
			-T_\bbR^\circ, &\text{ if } \bbW \text{ is infinite, equivalently } I \neq \{0\}.
	\end{cases}
\]
\end{remark}

\begin{remark}
In Ikeda's work on Kac--Moody root systems \cite{ikeda2014stability}, his version of the equality in the above remark \ref{Ikeda fix} is stated incorrectly.
Namely, one should take the interior of the Tits cone and not include $\{0\}$ on the RHS of the equation in Lemma 2.12 of loc.\ cit.
\end{remark}

\subsection{A hyperplane complement from the imaginary cone}
Recall from \eqref{eqn:dualnotation} that $(\bbR\Lambda)^*_\bbC := \Hom_{\bbR}(\bbR\Lambda, \bbC)$.
We have a direct sum decomposition of real vector spaces 
$$(\bbR\Lambda)^*_\bbC = (\bbR\Lambda)^*_\bbR \oplus i(\bbR\Lambda)^*_\bbR,$$
and given $Z \in (\bbR\Lambda)^*_\bbC$, we write
\[
Z=Z_{\re} + iZ_{\im},
\]
with $Z_{\re}, Z_{\im} \in (\bbR\Lambda)^*_\bbR$.

\begin{definition}[Hyperplane complement from imaginary cone] \label{defn:hyperplanecomplement}
For each $v \in \bbR\Lambda$, let $H_v := \{ Z \in (\bbR\Lambda)^*_\bbC \mid Z(v) = 0\}$.
We define the hyperplane complements
\begin{eqnarray}
 \Upsilon := (\bbR\Lambda)^*_\bbC \setminus \cup_{v \in I \setminus \{0\}} H_v \\
 \Upsilon_{\reg} := \Upsilon \setminus \cup_{\alpha \in \Phi^+} H_\alpha.
\end{eqnarray}
Let $\bbH := \{ x \in \bbC \mid \im(x) > 0\}$ denote the strict upper half plane. 
We define the \emph{complexified chamber}
\begin{equation}\label{eqn:complexiefiedchamber}
C := \{Z \in (\bbR\Lambda)^*_\bbC \mid Z(\alpha_s) \in \bbH \cup \bbR_{<0} \text{ for all } s\in \Gamma_0 \}.
\end{equation}
\end{definition}
Note that $C \subset \Upsilon_{\reg} \subset \Upsilon$ by construction, where the first containment follows from the fact that $I \subseteq \bbR_{\geq 0}\Lambda$, as noted in \cref{prop:imagcone}.

The space $\Upsilon_{\reg}$ is closely related to, though in general distinct from, the usual the usual complexified hyperplane complement $\Omega_{\reg}$ considered in the study of Artin--Tits groups \cite{VdL_thesis}; we recall that definition now.
We remind the reader that $T_\bbR^\circ$ denotes the interior of the Tits cone from \cref{defn:titscone}.
\begin{definition}[``Usual'' hyperplane complement] \label{defn:usualhyperplanecomplement}
Denote $\Omega := \{ Z \in (\bbR\Lambda)^*_\bbC \mid Z_{\im} \in T^\circ_\bbR \}$.
For each $\alpha \in \Phi^+$, consider the complex hyperplane $H_\alpha = \{ Z \in (\bbR\Lambda)^*_\bbC \mid Z(\alpha) = 0\}$.
The hyperplane complement $\Omega_{\reg}$ associated to $\Gamma$ is defined as
\[
\Omega_{\reg} := \Omega \setminus \cup_{\alpha \in \Phi^+} H_\alpha.
\]
By the result of Van der Lek \cite{VdL_thesis}, we have that $\bbW$ acts on $\Omega_{\reg}$ freely and properly discontinuously, and moreover
\[
\pi_1(\Omega_{\reg}/\bbW) \cong \B,
\]
where $\B := \B(\Gamma)$ is the Artin--Tits group associated to $\Gamma$.
\end{definition}
The precise relationship between $\Upsilon_{\reg}$ and $\Omega_{\reg}$ depends on whether $\bbW$ is finite or infinite.
On one hand, it is clear that $\Upsilon = (\bbR\Lambda)^*_\bbC = \Omega$ when $\bbW$ is finite (equivalently when $I=\{0\}$); in particular, the hyperplane complements $\Upsilon_{\reg}$ and $\Omega_{\reg}$ agree when $\bbW$ is finite.  The goal of the remainder of this section is to establish a relationship between $\Upsilon_{\reg}$ and $\Omega_{\reg}$ when $\bbW$ is infinite.  We will show in \cref{cor:homotopyequivalence} that $\Upsilon_{\reg}$ is homotopy equivalent to $S^1 \times \Omega_{\reg}$ in these cases. 
\begin{lemma}
Both $\Upsilon$ and $\Upsilon_{\reg}$ are open subsets of $(\bbR\Lambda)^*_\bbC$.
\end{lemma}
\begin{proof}
The case where $I=\{0\}$ is immediate from the identification $\Upsilon = (\bbR\Lambda)^*_\bbC$ and $\Upsilon_{\reg} = \Omega_{\reg}$. So in the rest of the proof we assume $I \neq \{0\}$.

Denote the sphere of length one vectors by $S := \{ v \in \bbR\Lambda \mid ||v|| = 1\}$; here $\bbR\Lambda$ is equipped with the standard norm.
Since $(\bbR\Lambda)^*_\bbC$ is a finite-dimensional vector space, all norms induce the same topology.
For the sake of argument we shall use the operator norm: 
\[
||Z|| := \sup_{v \in S} \{ |Z(v)| \}.
\]

We prove that $\cup_{v \in I \setminus \{0\}} H_v$ is closed in $(\bbR\Lambda)^*_\bbC$ by showing that it contains all of its limit points.
Suppose $(Z_j)_{j=1}^\infty$ is a sequence in $\cup_{v \in I \setminus \{0\}} H_v$ converging to a point $Z \in (\bbR\Lambda)^*_\bbC$ under the operator norm.
By definition of the $Z_j$'s, we have for each $j$ some $v_j \in I$ such that $Z_j(v_j) = 0$
Since $I$ is a closed convex cone in $\bbR\Lambda$, the normalised set $\hat{I} := I \cap S$ is a compact set in $\bbR\Lambda$.
By linearity of $Z_j$, we may normalise $v_j$ so that $\hat{v}_j \in \hat{I}$ and still have $Z_j(\hat{v}_j) = 0$.
Since $Z$ is the point of convergence of $(Z_j)_{j=1}^\infty$, we have for each $n \in \bbN$ some $j_n$ such that $|| Z - Z_{j_n} || < 1/n$.
By definition of the operator norm we get
\begin{equation} \label{eqn:sequenceofZpoints}
|Z(\hat{v}_{j_n})| = |Z(\hat{v}_{j_n}) - Z_{j_n}(\hat{v}_{j_n})| < 1/n.
\end{equation}
Using compactness of $\hat{I}$ we obtain a subsequence of $(v_{j_n})_{n=1}^\infty$ that converges to some $\hat{v} \in \hat{I}$, and by \eqref{eqn:sequenceofZpoints} we get $Z(\hat{v}) = 0$.
This shows that $Z \in \cup_{v \in I \setminus \{0\}} H_v$, which proves that $\cup_{v \in I \setminus \{0\}} H_v$ is closed in $(\bbR\Lambda)^*_\bbC$.
We can then conclude that $\Upsilon = (\bbR\Lambda)^*_\bbC \setminus \cup_{v \in I \setminus \{0\}} H_v$ is open in $(\bbR\Lambda)^*_\bbC$.

With the above, to show that $\Upsilon_{\reg}$ is open in $(\bbR\Lambda)^*_\bbC$, it is sufficient to show that it is open in $\Upsilon$.
We do this by showing again that $\cup_{\alpha \in \Phi^+} H_\alpha \cap \Upsilon$ is closed in $\Upsilon$, and by showing that it contains all of its limit points -- in $\Upsilon$.
The general strategy is similar as before, where we consider instead the normalised set $\widehat{\bbR_{\geq 0} \Lambda} := \bbR_{\geq 0}\Lambda \cap S$, which is again compact since $\bbR_{\geq 0}\Lambda$ is also a closed convex cone.
Note that this subset contains all of the normalised positive roots $\hat{\Phi}^+ := \Phi^+ \cap S$.

Suppose $(Z_j)_{j=1}^\infty$ is a sequence in $\cup_{\alpha \in \Phi^+} H_\alpha \cap \Upsilon$ converging to a point $Z \in \Upsilon$ under the operator norm.
We have for each $j$ some $\hat{\lambda}_j \in \hat{\Phi}^+$ such that $Z_j(\hat{\lambda}_j) = 0$.
By the same argument as before we get some subsequence $(\hat{\lambda}_{j_n})_{n=1}^\infty$ satisfying
\[
|Z(\hat{\lambda}_{j_n})| = |Z(\hat{\lambda}_{j_n}) - Z_{j_n}(\hat{\lambda}_{j_n})| < 1/n.
\]
By compactness of $\widehat{\bbR_{\geq 0} \Lambda}$ we have a subsequence of $(\hat{\lambda}_{j_n})_{n=1}^\infty$ that converges to some point $\lambda \in \widehat{\bbR_{\geq 0} \Lambda}$, which moreover satisfies $Z(\hat{\lambda})=0$.
We claim that $\hat{\lambda}$ is in fact a normalised positive root.
This is because any accumulation point of $\hat{\Phi}^+$ would lie in $\hat{I}$ by \cref{prop:imagcone}; if $\hat{\lambda}$ was an accumulation point, then we would have $Z(\hat{\lambda}) = 0$ for $\hat{\lambda} \in \hat{I}$, contradicting the fact that $Z \in \Upsilon$.
It follows that $\hat{\lambda} \in \hat{\Phi}^+$, which shows that $Z \in \cup_{\alpha \in \Phi^+} H_\alpha$, as required.
\end{proof}

\begin{lemma}
Suppose $I \neq \{0\}$.
For all $Z \in \Upsilon$, $Z(I \setminus \{0\})$ is a convex set not containing $0$, and $Z(I)$ is a (real) closed convex cone in $\bbC$ that is not the whole of $\bbC$.
\end{lemma}
\begin{proof}
Since $Z$ is $\bbR$-linear and $I$ is a closed convex cone by \cref{prop:imagcone}, it follows that $Z(I)$ is a closed convex cone.
Note that $I \setminus \{0\}$ is still a convex set, and so $Z(I \setminus \{0\})$ is also convex.
Since $Z \in \Upsilon$, by definition $ 0 \not\in Z(I \setminus \{0\})$.
It follows that, the arguments of any two complex numbers in $Z(I \setminus \{0\})$ may only differ by strictly less than $\pi$.
In particular, $Z(I) = Z(I \setminus \{0\}) \cup Z(0)$ cannot be the whole of $\bbC$.
\end{proof}
%When $I \neq \{0\}$, by the lemma above the set of arguments $\arg(Z(I\setminus \{0\}))$ is given by a closed interval $[a,b] \subset [0, 2\pi)$ with $b-a < \pi$.
%We define $\phi^I(Z) := \frac{a+b}{2}$ to be the argument of the middle ray of the cone.
When $I \neq \{0\}$, the lemma above shows that $Z(I)$ is a proper closed convex cone in $\bbC$, in which there is a well-defined middle ray.
We define $\phi^I(Z)$ to be the common argument of the (non-zero) elements of the middle ray.
See \cref{fig:ZIcone}.

\begin{figure}
\begin{tikzpicture}
\def\R{1.5}
\begin{scope}[rotate=90]
% Define cone and middle node
\coordinate (a) at (0:\R);
\coordinate (b) at (120:\R);
\coordinate (p) at (60:\R);

% Draw disk
\fill[gray!10] (0,0) circle (\R);

% Draw shade for cone
\fill[blue!20] (0,0) -- (\R,0) arc[start angle=0, end angle=120, radius=\R] -- cycle;
% Node labels
%\node[above] at (a) {$a$};
%\node[left] at (b) {$b$};
\node[above left] at (p) {$\phi^I(Z)$};

% Draw cone: a triangle shape
\draw[thick] (0,0) -- (a);
\draw[thick] (0,0) -- (b);
% Triangle for cone's surface
\draw[thick, dashed] (0,0) -- (p);  % The center axis of the cone
\end{scope}

\begin{scope}[shift={(5,0)}]
\begin{scope}[rotate=30]
% Define cone and middle node
\coordinate (a) at (0:\R);
\coordinate (b) at (120:\R);
\coordinate (p) at (60:\R);

% Draw disk
\fill[gray!10] (0,0) circle (\R);

% Draw shade for cone
\fill[blue!20] (0,0) -- (\R,0) arc[start angle=0, end angle=120, radius=\R] -- cycle;
% Node labels
%\node[above] at (a) {$a$};
%\node[left] at (b) {$b$};
\node[above] at (p) {$\phi^I(Z')$};

% Draw cone: a triangle shape
\draw[thick] (0,0) -- (a);
\draw[thick] (0,0) -- (b);
% Triangle for cone's surface
\draw[thick, dashed] (0,0) -- (p);  % The center axis of the cone
\end{scope}
\end{scope}
\end{tikzpicture}
\caption{The images $Z(I), Z'(I) \subset \bbC$ of the imaginary cone when $I \neq \{0\}$, shaded in blue. $Z'$ (right) is normalised, whereas $Z$ (left) is not.}
\label{fig:ZIcone}
\end{figure}
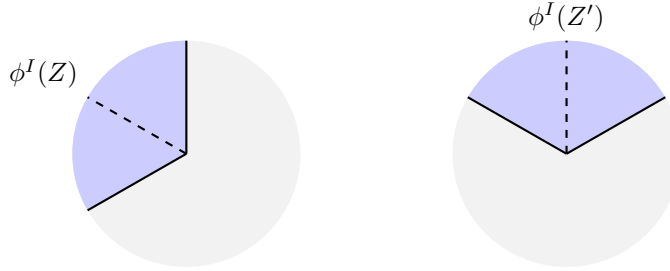

\begin{definition} \label{defn:normalisedspace}
Suppose $I \neq \{0\}$.
We define the normalised space $\Upsilon_{\reg}^N$ by
\[
\Upsilon_{\reg}^N := \{Z \in \Upsilon_{\reg} \mid \phi^I(Z) = \pi/2 \}
\]
and the normalised complexified chamber $C^N$ by
\[
C^N := C \cap \Upsilon_{\reg}^N = \{ Z \in C \mid \phi^I(Z) = \pi/2 \}.
\]
To simplify our exposition later, when $I = \{0\}$ we will set $\Upsilon_{\reg}^N := \Upsilon_{\reg}$, so that $C^N = C$ as well.
\end{definition}
Note that $S^1 = \{ e^{i\theta} \mid \theta \in [0,2\pi) \}$ acts freely on $\Upsilon_{\reg}$ via scalar multiplication.
In particular, when $I \neq \{0\}$ any $Z \in \Upsilon_{\reg}$ can be normalised uniquely into an element in $\Upsilon_{\reg}^N$ via $Z \mapsto e^{i(-\phi^I(Z)+\pi/2)}\cdot Z$.
It follows that ($I=\{0\}$ case is by convention):
\begin{equation} \label{eqn:normaliseddecomp}
\Upsilon_{\reg} \cong 
	\begin{cases}
	S^1 \times \Upsilon_{\reg}^N, &\text{ if } I \neq \{0\}; \\
	\Upsilon_{\reg}^N, &\text{ if } I=\{0\}.
	\end{cases}
\end{equation}

\begin{lemma} \label{lem:WactionN}
The actions of $S^1$ and $\bbW$ on $\Upsilon_{\reg}$ commute with each other.
Moreover, $\Upsilon_{\reg}^N$ is a $\bbW$-invariant subset of $\Upsilon_{\reg}$, and the homeomorphism in \eqref{eqn:normaliseddecomp} can be chosen to be $\bbW$-equivariant.
\end{lemma}
\begin{proof}
The fact that the two actions commute is immediate from the definition.
The last two statements for the $I=\{0\}$ case are also immediate from our convention.
Now suppose $I \neq \{0\}$.
Since $\bbW\cdot I = I$, $\phi^I(Z) = \phi^I(w\cdot Z)$ for all $w \in \bbW$ and so $\Upsilon_{\reg}^N$ is closed under the action of $\bbW$.
The final statement follows directly from the commutativity of the actions.
\end{proof}

\begin{lemma} \label{lem:imaginaryinTits}
For any $Z = Z_{\re} + iZ_{\im} \in \Upsilon_{\reg}^N$, $Z_{\im} \in T_\bbR^\circ$.
\end{lemma}
\begin{proof}
The $I=\{0\}$ case is immediate since $T_\bbR^\circ = \bbR$.
Now suppose $I \neq \{0\}$ and let $Z \in \Upsilon_{\reg}^N$.
Then $Z(I \setminus \{0\}) \subset \bbH$, and so $Z_{\im}(v) > 0$ for all $v \in I \setminus \{0\}$.
By \cref{prop:titscone=positiveimagcone}, we see that $Z_{\im}$ is in the interior $T_\bbR^\circ$ of the Tits cone.
\end{proof}

\begin{proposition} \label{prop:WandS1actiongenerates}
For any $Z \in \Upsilon_{\reg}$, there exists $k \in S^1$ and $w \in \bbW$ such that $w \cdot k \cdot Z \in C^N$.
\end{proposition}
\begin{proof}
Suppose $I \neq \{0\}$.
Any element in $Z \in \Upsilon_{\reg}$ can be normalised (uniquely) to an element $k\cdot Z \in \Upsilon_{\reg}^N$ for some (unique) $k \in S^1$.
So without lost of generality we may, and shall, assume that $Z \in \Upsilon_{\reg}^N$.
By the lemma above $Z_{\im} \in T_\bbR^\circ$, so there exists $w \in \bbW$ such that $w \cdot Z_{\im} \in \overline{C}_\bbR \cap T_\bbR^\circ$.
Set $Z' := w\cdot Z$, so that $Z'_{\im} = w\cdot Z_{\im} \in \overline{C}_\bbR \cap T_\bbR^\circ$.
In particular, we that $Z'(\alpha_s) \in \bbH \cup \bbR$ for all $s \in \Gamma_0$.
Consider the set
\[
J:= \{ s \in \Gamma_0 \mid Z'_{\im}(\alpha_s) = 0 \}.
\]
Note that $J$ generated a finite parabolic subgroup $\bbW_J$ of $\bbW$ since $Z'_{\im}\in \overline{C}_\bbR \cap T_\bbR^\circ$ (see \cref{prop:titscone}).
Moreover, we have chosen $J$ in such a way that
\[
Z'(\alpha_s) \in
	\begin{cases}
	\bbR, &\text{ if } s \in J; \\
	\bbH, &\text{ if } s \in \Gamma_0 \setminus J.
	\end{cases}
\]
We claim that we can find an element $u$ in the finite parabolic subgroup $\bbW_J \subseteq \bbW$ such that 
\[
(u\cdot Z')(\alpha_s) \in
	\begin{cases}
	\bbR_{<0}, &\text{ if } s \in J; \\
	\bbH, &\text{ if } s \in \Gamma_0 \setminus J.
	\end{cases}
\]
Consider the subspace 
\[
\bbR\Lambda_J:= \bigoplus_{t\in J \subseteq \Gamma_0} \bbR\{\alpha_t\} \subseteq \bbR\Lambda
\]
associated to the finite Coxeter group $\bbW_J$.
Take $Z'_{\re}|_{\bbR\Lambda_J} \in (\bbR\Lambda_J)^*_\bbR$ to be $Z'_{\re}$ restricted to the subspace $\bbR\Lambda_J$.
By \cref{prop:titscone}, $(\bbR\Lambda_J)^*_\bbR$ is equal to the Tits cone $T_{J,\bbR}$ associated to $\bbW_J$, so we have some (unique) element $u \in \bbW_J$ such that $(u\cdot Z'_{\re}|_{\bbR\Lambda_J})(\alpha_t) \leq 0$ for all $t \in J$; this is just saying that under the action of $u \in \bbW_J$ we send $Z'_{\re}|_{\bbR\Lambda_J}$ to an element in the negative closed fundamental chamber $-\overline{C}_{J,\bbR} := \{ Z \in (\bbR\Lambda_J)^*_\bbR \mid Z(\alpha_t) \leq 0 \text{ for all } t\in J\}$.
Now apply this element $u \in \bbW_J \subseteq \bbW$ to $Z'$.
Notice that for all $s \in \Gamma_0$, $u^{-1}\cdot \alpha_s - \alpha_s \in \bbR\Lambda_J$, so $Z'_{\im}(u^{-1}\cdot \alpha_s - \alpha_s) = 0$.
As such,
\[
(u\cdot Z'_{\im})(\alpha_s) = Z'_{\im}(u^{-1}\cdot \alpha_s) = Z'_{\im}(\alpha_s), \quad \text{for all } s \in \Gamma_0;
\]
in other words the imaginary part of $Z'$ is left unchanged.
It follows that for all simple roots $\alpha_s$ with $s \in \Gamma_0 \setminus J$, $Z'(\alpha_s)$ is still in $\bbH$.
Moreover, $(u\cdot Z')_{\re}(\alpha_t) \in \bbR_{\leq 0}$ for all $t \in J$ by the choice of $u$ (it does not matter where $Z'_{\re}(\alpha_s)$ lives for $s \in \Gamma_0 \setminus J$).
Since $Z'$ is in $\Upsilon_{\reg}^N$ and $\Upsilon_{\reg}^N$ is closed under the action of $\bbW$, $(u\cdot Z')_{\im}(\alpha_t)=0 \implies (u\cdot Z')_{\re}(\alpha_t) \neq 0$, so $(u\cdot Z')_{\re}(\alpha_t) \in \bbR_{< 0}$ for all $t \in J$ as required.

The proof for the $I=\{0\}$ case is a simpler version of the proof above, where we do not need to normalise the elements ($\Upsilon_{\reg}^N = \Upsilon_{\reg}$ and $C^N = C$ by convention) and $J$ is always a finite parabolic subgroup since $\bbW$ is finite itself (see \cref{prop:imagcone}).
\end{proof}

\begin{proposition}\label{prop:Wactionfreepropdiscont}
\begin{enumerate}[(i)]
\item The commuting actions of $S^1$ and $\bbW$ on $\Upsilon_{\reg}$ are both free and properly discontinuous.
\item If $I \neq \{0\}$, a fundamental domain of $\bbW$ is $S^1 \times C^N \subset S^1 \times \Upsilon_{\reg}^N \cong \Upsilon^{\reg}$, and a fundamental domain of $S^1 \times \bbW$ is $C^N \subset \Upsilon_{\reg}$.
\item If $I = \{0\}$,  a fundamental domain of $\bbW$ is $C (= C^N) \subset \Upsilon_{\reg} (= \Upsilon_{\reg}^N)$.
\end{enumerate}
\end{proposition}
\begin{proof}
The fact that the actions of $S^1$ and $\bbW$ commute with each other is immediate; so is the freeness and properly discontinuousness of the $S^1$-action.

First suppose $I\neq \{0\}$.
By the previous proposition we have that the $\bbW$-orbit of $S^1 \times C^N$ generates the whole space $\Upsilon_{\reg}$.
To show that the action of $\bbW$ is free, it is sufficient to show that no points in $S^1 \times C^N$ is fixed under the group action of $\bbW$.
Similarly, to show that the action of $\bbW$ is properly discontinuous, it is sufficient to only look at the points $Z \in S^1 \times C^N$.
With that being said, take $Z \in S^1 \times C^N \subseteq \Upsilon_{\reg}$.
By rotating $Z$ via the $S^1$-action if necessary, we may further assume that $Z(\alpha_s) \in \bbH$ for all $s \in \Gamma_0$; this is harmless since the $S^1$-action is free and it commutes with the $\bbW$-action.
This ensures that $Z_{\im}(\alpha_s) > 0$ for all the simple roots $\alpha_s$, so $Z_{\im}$ is in the open chamber $C_\bbR$.
By \cref{prop:titscone}, $C_\bbR$ is a fundamental domain of $T_{\bbR,\reg} := T_\bbR \setminus \cup_{\alpha \in \Phi^+} H_{\alpha,\bbR}$ and the action of $\bbW$ on $T_{\bbR,\reg}$ is moreover free and properly discontinuous.
It follows that the action of $\bbW$ on $\Upsilon_{\reg}$ is free since no element in $\bbW$ fixes the imaginary part $Z_{\im}$ of $Z$ other than identity element for all $Z \in S^1 \times C^N$.
Similarly, for each $Z \in S^1 \times C^N$ we can choose an open neighbourhood $U$ of $Z$ with the property that on the imaginary parts we have $w\cdot U_{\im} \cap U_{\im} \neq \emptyset$ if and only if $w = \id$.
As such, the action of $\bbW$ on $\Upsilon_{\reg}$ is also properly discontinuous.
The fact that $S^1\times C^N$ is a fundamental domain for the $\bbW$-action follows immediately.

The case $I=\{0\}$ follows from the same proof, where we do not need to normalise with $k \in S^1$ and we replace $S^1 \times C^N$ with $C$ ($=C^N$) throughout the proof.
\end{proof}

\subsection{Relating the hyperplane complements and the fundamental groups} \label{sec:relatehyperplanecomp}
To understand the space $\Upsilon_{\reg}$ and its fundamental group, we will relate its normalised part $\Upsilon_{\reg}^N$ to the more commonly studied complexified hyperplane complement $\Omega_{\reg}$ (see \cref{defn:usualhyperplanecomplement}).

\begin{theorem}
There is a $\bbW$-equivariant inclusion $\Upsilon_{\reg}^N \rightarrow \Omega_{\reg}$ and a $\bbW$-equivariant deformation retraction 
$\Omega_{\reg} \rightarrow \Upsilon_{\reg}^N$.
\end{theorem}
\begin{proof}
When $I=\{0\}$ this is trivial since $\Upsilon_{\reg}^N = \Upsilon_{\reg}$ by convention and $\Upsilon_{\reg} = \Omega_{\reg}$.

Now suppose $I \neq \{0\}$.
The fact that $\Upsilon_{\reg}^N \subseteq \Omega_{\reg}$ follows immediately from \cref{lem:imaginaryinTits}.
The $\bbW$-action on $\Upsilon_{\reg}^N$ and $\Omega_{\reg}$ are defined in exactly the same way and so the natural inclusion is indeed $\bbW$-equivariant.

For each $t \in [0,1]$, consider the map $h_t: \Omega_{\reg} \ra \Omega_{\reg}$ where
\[
h_t(Z) := e^{it(\pi/2-\phi^I(Z))}Z.
\]
By construction we have that $h_0 = \id$.
Moreover, $h_1(\Omega_{\reg}) = \Upsilon_{\reg}^N$ (note that $e^{i(-\phi^I(Z)+\pi/2)}$ is exactly the element in $S^1$ used to normalised elements in $\Upsilon_{\reg}$), and $h_1 = \id|_{\Upsilon_{\reg}^N}$.
It follows that this is a deformation retraction.

Since the $\bbW$-action commutes with $e^{it(\pi/2-\phi^I(Z))}$ for all $t$ (this is just part of the $S^1$-action), the deformation retraction is also $\bbW$-equivariant.
\end{proof}
\begin{remark} \label{rmk:hyperplanecomplementrelation}
Suppose $I \neq \{0\}$.
In fact, we have a chain of inclusions $\Upsilon_{\reg}^N \subseteq \Omega_{\reg} \subseteq \Upsilon_{\reg}$, with all inclusions being $\bbW$-equivariant.
Roughly speaking, the space $\Upsilon_{\reg}$ includes both the positive chambers and the negative chambers, consisting of two copies of $\Omega_{\reg}$ (positive side and negative side) which are ``glued together'' nicely.
The extra $S^1$-component is a loop that goes around the union of hyperplanes associated to the closed imaginary cone in the complex space.
The picture for $\Gamma$ the affine $\hat{A}_1$ type will be quite illuminating.
\end{remark}

\begin{corollary} \label{cor:homotopyequivalence}
We have the homotopy equivalences:
\[
\Upsilon_{\reg} \simeq 
	\begin{cases}
	\Omega_{\reg},  &\text{ if } \bbW \text{ is finite, equivalently } I = \{0\}; \\
	S^1 \times \Omega_{\reg}, &\text{ if } \bbW \text{ is infinite, equivalently } I \neq \{0\}.
	\end{cases}
\]
Moreover, the fundamental group of $\Upsilon_{\reg}/\bbW$ is given by
\begin{equation} \label{eqn:fundgrpArtin}
\pi_1(\Upsilon_{\reg}/\bbW) \cong 
\begin{cases}
\B(\Gamma), &\text{ if } \bbW \text{ is finite, equivalently } I = \{0\}; \\
\bbZ \times \B(\Gamma), &\text{ if } \bbW \text{ is infinite, equivalently } I \neq \{0\}.
\end{cases}
\end{equation}
\end{corollary}
\begin{proof}
The first statement is immediate; $\Upsilon_{\reg}$ is in fact equal to $\Omega_{\reg}$ when $\bbW$ is finite, whereas the case when $\bbW$ is infinite follows from the previous theorem and the description $\Upsilon_{\reg} \cong S^1 \times \Upsilon_{\reg}^N$.

For the second statement, the case when $\bbW$ is finite is exactly Van der Lek's result that $\pi_1(\Omega_{\reg}/\bbW) \cong \B(\Gamma)$ \cite{VdL_thesis}.
When $\bbW$ is infinite, the theorem above shows that the homotopy equivalence $\Upsilon_{\reg}^N \simeq \Omega_{\reg}$ is $\bbW$-equivariant.
Together with \cref{prop:Wactionfreepropdiscont}, we have that
\[
\Upsilon_{\reg}/\bbW \cong S^1 \times (\Upsilon_{\reg}^N/\bbW) \simeq S^1 \times (\Omega_{\reg}/\bbW).
\]
The rest once again follows from Van der Lek's result \cite{VdL_thesis}.
\end{proof}

%
%Let $\check{\Gamma}$ be the unfolded Coxeter graph associated to $\Gamma$.
%Note that $\Omega_{\check{\Gamma}} = \{ Z \in \Hom_{\bbZ}(\Lambda, \bbC) \mid Z_I \in T_{\check{\Gamma}} \}$, and moreover $\Omega \subseteq \Omega_{\check{\Gamma}}$.
%\begin{lemma}
%$N(\Gamma) \subseteq N(\check{\Gamma})$.
%\end{lemma}
%\todo{The above should be a result that I will prove with Luis Paris.}

\section{Fusion categories and fusion quivers}\label{sec:fusion1}
\subsection{Tensor and fusion categories}\label{sect: fusion}
We briefly review the notion of fusion categories.  The definitions of algebra and and module objects in such categories are standard; we use these notions freely in the sequel, referring the reader to \cite{etingof_nikshych_ostrik_2005, EGNO15} for more details.
In what follows $\Bbbk$ denotes an algebraically closed field.

\begin{definition}\label{defn: tensor fusion category}
A (strict) \emph{tensor category} $\cC$ is a $\Bbbk$-linear, finite abelian category equipped with a (strict) monoidal structure $- \otimes -$ that is bilinear on morphisms, such that every object is rigid (it has both a left and a right dual) and the monoidal unit $\1$ is simple.
If $\cC$ is moreover semisimple with finitely many isomorphism classes of simple objects, we say that $\cC$ is a \emph{fusion category}.
\end{definition}
Note that in a fusion category, left and right duals are isomorphic; in particular we have $X \cong X^{**}$ with $(X \otimes Y)^{**} \cong X^{**} \otimes Y^{**}$ (even though these isomorphisms may not collective form a pivotal structure).

A first example of a fusion category is the category of finite dimensional $\Bbbk$-vector spaces.  Other fusion categories of relevance in the current paper will be defined in \cref{subsection: TLJ}.

\begin{definition}\label{defn: fusion ring}
Let $\cC$ be a fusion category and $\Irr(\cC) =\{ S_k \}_{k=0}^n$ be the representatives of simple objects in $\cC$ with $S_0:= \1$ the monoidal unit.
The \emph{fusion ring} $K_0(\cC)$ of $\cC$ is the Grothendieck group of $\cC$ equipped with  multiplication $[S_i]\cdot[S_j] := [S_i \otimes S_j]$, which by semisimplicity, can be written as
\[
[S_k] \cdot [S_i] = \sum_{i=0}^k {}_kr_{i,j}, \quad {}_kr_{i,j} \in \mathbb{N}_0.
\]
The ring unit is the class $[S_0] = [\1]$ of the monoidal unit.
\end{definition}

The following notion of (real-valued) dimensions for fusion categories will play an important role in future sections.
\begin{definition}\label{defn:FPdim}
Let $\cC$ be a fusion category with $\Irr(\cC) =\{ S_k \}_{k=0}^n$.
For each simple $S_k$, we define $\FPdim([S_k]) \in \R$ to be the Frobenius--Perron eigenvalue (the unique largest real eigenvalue) of the non-negative matrix $({}_k r_{i,j})_{i,j \in \{0,1,...,n\}}$.
This extends to a ring homomorphism \cite[Proposition 3.3.6(1)]{EGNO15} that we denote by 
\[
\FPdim\colon K_0(\cC) \ra \R.
\]
The \emph{Frobenius--Perron dimension}, $\FPdim(Y)$, of an object $Y \in \cC$ is defined to be $\FPdim([Y])$.
\end{definition}
If we agree that the dimension of any object should be positive and dimension should respect tensor product and direct sums (multiplicative and additive respectively), then $\FPdim$ is the unique ring homomorphism satisfying such properties \cite[Proposition 3.3.6(3)]{EGNO15}.
\begin{remark}
Note that $\FPdim(Y) \geq 0$ for any object $Y \in \cC$ and is zero if and only if $Y \cong 0$.
Moreover, a numerical coincidence in relation to Coxeter theory is that $\FPdim(Y) < 2$ if and only if $\FPdim(Y) = 2\cos(\pi/m)$ for some integer $m \geq 2$.
See \cite[\S 3.3]{EGNO15} for more details.
\end{remark}

Later on, we will construct new tensor categories from old via taking the Deligne's tensor product of tensor categories.
\begin{definition}
Let $\cC$ and $\cC'$ be fusion categories.
The \emph{(Deligne's) tensor product} $\cC \boxtimes \cC'$ is the fusion category with objects and morphism spaces given as follows:
\begin{enumerate}
\item Obj$(\cC \boxtimes \cC')$ consists of finite direct sums of the formal objects $A \boxtimes A'$, with $A$ and $A'$ objects in $\cC$ and $\cC'$ respectively; %Note that (A,0) \cong 0 \cong (0,B)
\item $\Hom_{\cC \boxtimes \cC'} \left(\bigoplus A_i \boxtimes A_i', \bigoplus B_j \boxtimes B_j'\right) = \bigoplus_{i,j} \Hom_\cC \left( A_i, B_j \right) \otimes \Hom_{\cC'} \left( A_i', B_j' \right)$.
\end{enumerate}
The monoidal structure on $\cC \boxtimes \cC'$ is defined component-wise by
\[
(A \boxtimes A') \otimes (B \boxtimes B') := (A \otimes B) \boxtimes (A' \otimes B'),
\] 
with monoidal unit $\1_\cC \boxtimes \1_{\cC'}$ and direct sums in each $\boxtimes$-component is distributive over $\boxtimes$.
The simple objects are given by $S \boxtimes S'$ for $S$ and $S'$ simple objects in $\cC$ and $\cC'$ respectively.
The left (resp.\ right) dual of $S \boxtimes S'$ is just the left (resp.\ right) dual of each component.
\end{definition}
Given a finite set $\Gamma_0$ and with fusion category $\cC_s$ for each $s \in \Gamma_0$, we may define the fusion category $\bigboxtimes_{s \in \Gamma_0} \cC_s$  in a similar fashion.
It is easy to see that the fusion ring of $\bigboxtimes_{s \in \Gamma_0} \cC_s$ is just a tensor product of the individual component 
\[
K_0 \Big(\bigboxtimesp{}{s \in \Gamma_0} \cC_s \Big) \cong \bigotimes_{s \in \Gamma_0} K_0(\cC_s),
\]
where the tensor products of the fusion rings are taken over $\Z$.
The $\FPdim$ on $K_0\big(\bigboxtimes_{s \in \Gamma_0} \cC_s\big)$ is induced by the $\FPdim$ defined on each component $K_0(\cC_s)$, so that 
\[
\FPdim\left(\bigboxtimes_{s \in \Gamma_0} A_s \right) = \prod_{s\in \Gamma_0} \FPdim(A_s).
\]
%\begin{remark}
%The Deligne's tensor product for abelian categories is by definition a universal object; when it exists, it is unique up to a unique isomorphism. 
%The above ``naive'' tensor product of categories is just an explicit realisation of the Deligne's tensor product in the special case of fusion categories (see for example \cite[Chapter 5]{Row19}).
%\end{remark}

\subsection{Graded algebras and module categories}\label{subsect: module categories}
Throughout this subsection, $\cC$ will be a tensor category with monoidal unit $\1$.

We recall the definition of algebras and their (bi)modules in tensor categories, following \cite[\S 7.8]{EGNO15}.
An \emph{algebra} in $\cC$ is an object $A$ in $\cC$ equipped with two maps:
\begin{align*}
\mu &: A \otimes A \ra A \quad \text{(multiplication)} \\
\eta &: \1 \ra A \quad \text{(unit)}
\end{align*}
satisfying the associativity and unital conditions.
Similarly, a left $A$-\emph{module} is an object $M \in \cC$ equipped with an action map 
\[
\alpha_M : A \otimes M \ra M \quad \text{(action)}
\]
satisfying the associativity condition.
A \emph{morphism} between two left $A$-modules $(M,\alpha_M)$ and $(N, \alpha_N)$ is by definition a morphism $\varphi: M \ra N$ in $\cC$ that commutes with the action maps of $M$ and $N$.
The notions of \emph{right} $A$-modules and morphisms between them are defined similarly.
Given two algebras $A$ and $B$, an $(A,B)$-bimodule $M$ is therefore an object $M$ which is both a left $A$-module and a right $B$-module, such that the left and right actions on $M$ commute; morphisms between bimodules are those which are both morphisms of left and right modules.

In our setting, we will be interested in algebras and (bi)modules equipped with a $\Z$-grading.
To this end, we say that an algebra $A$ in $\cC$ is ($\bbZ$-)\emph{graded} if $A$ has a decomposition $A = \bigoplus_{i \in \Z} A_i$ such that $\mu:A_i \otimes A_j \ra A_{i+j}$.
Note that the unital condition of an algebra dictates that $\eta: \1 \ra A_0 \subseteq A$.
Whenever the algebra $A$ considered is graded, we always require the corresponding (bi)modules to be \emph{graded} as well; i.e.\ $M = \bigoplus_{i \in \Z} M_i$ such that $\alpha: A_i \otimes M_j \ra M_{i+j}$ for all $i, j \in \Z$ -- similarly for right modules and bimodules.
In addition, morphisms between graded modules are required to be grading preserving.

We denote the category of left $A$-modules and right $A$-modules by $A\lmod$ and $\rmod A$ respectively; following our convention, when $A$ is graded the same notation will be used denote the category of graded modules (with grading preserving morphisms).
Similarly, we use $A\bimod B$ to denote the category of (graded) $(A,B)$-bimodules.
The grading shift functor on the category of graded (bi)modules will be denoted by $\<1\>$, where a graded module $M$ with decomposition $M = \bigoplus_{i \in \Z} M_i$ is sent to the graded module $M\<1\>$ with its homogeneous degree $i+1$ piece given by $(M\<1\>)_{i+1} := M_i$.

Note that the categories $A\lmod$, $\rmod A$ and $A\bimod B$ are all locally finite (in fact, finite if non-graded) abelian categories induced by the finite abelian category structure of $\cC$ underlying them.
Furthermore, the grading shift functor $\<1\>$ is exact by construction.
As such, if $(M, \alpha, \theta) \in$ $A$-$\mod$-$B$ and $(N, \theta', \alpha') \in $ $B$-$\mod$-$C$, their \emph{tensor product over $B$}, $M\otimes_B N$ is defined to be the coequaliser in $A\bimod C$:
\begin{center}
\begin{tikzcd}
M\otimes B \otimes N \ar[r, shift left=0.75ex, "\theta \otimes \id"] \ar[r, shift right=0.75ex, swap, "\id \otimes \theta'"] & M\otimes N \ar[r, "\Theta"] & M\otimes_B N;
\end{tikzcd}
\end{center}
namely $M\otimes_B N$ is the cokernel of the morphism $\theta\otimes \id - \id\otimes \theta'$.
Note that we always have $A \otimes_A M \cong M$ and $M \otimes_B B \cong M$ (with isomorphisms in their respective bimodule categories).
Moreover, if $A = \1 = C$ is the trivial algebra, so that $M$ and $N$ are simply right and left $B$-modules respectively, the tensor product $M \otimes_B N$ is the coequaliser in $\cC = \1\bimod \1$.

Let $\cC^{\otimes^\text{op}}$ denote the category $\cC$ with the order of the monoidal tensor reversed; this is \emph{not} the same as the opposite category, which is the category obtained by reversing all morphisms.
Given a $\Bbbk$-linear additive category $\cA$, we let $\cEnd(\cA)$ denote the category of linear, additive endofunctors of $\cA$.
Note that the category $\cEnd(\cA)$ is additive and monoidal, with monoidal structure given by composition.

\begin{definition}
A {\it right module category} over $\cC$ (or right $\cC$-module category) is an additive category $\cA$ together with an additive monoidal functor $\cG: \cC^{\otimes^\text{op}} \ra \cEnd(\cA)$.
When $\cA$ is abelian or triangulated, we require $\cG$ to map into the category of exact endofunctors.
When $\cA$ is equipped with an internal grading shift functor $\<1\>$ (e.g.\ $\cA = A\lmod$ for $A$ a graded algebra), we also require $\cG$ to map to endofunctors which commutes trivially with $\<1\>$.
The functor $\cG(Y)$ for $Y \in \cC$ will usually be denoted as $- \otimes Y$, and we will also say that $\cC$ \emph{acts} on $\cA$.

An endofunctor $\cF$ of a $\cC$-module category $\cA$ is said to be a \emph{$\cC$-module functor} if $\cF$ commutes with the $\cC$-action.
\end{definition}

\begin{remark}
Our stated notion of a module category for when $\cA$ is abelian is slightly weaker than what is normally considered (e.g.\ in \cite[\S 7.3]{EGNO15}).
Namely, the monoidal functor $\cG$ is also usually required to be \emph{exact} (into the category of left exact endofunctors of $\cA$, which is an abelian category), but we will not require this.
%In particular, the functor $\cG(-)_M = M \otimes - : \cC \rightarrow \cA$ need not be exact.
When $\cA$ is instead triangulated, the category of exact endofunctors may not even have a natural abelian category structure, thus we shall not attempt to ask for $\cG$ to be ``exact''.
\end{remark}

\begin{remark}
More generally, a $\cC$-module endofunctor $\cF$ includes the data of a natural isomorphism $\psi_{M,Y}: \cF(M\otimes Y) \xra{\cong} \cF(M)\otimes Y$ satisfying the associativity (pentagon) and unital conditions.
In this paper the natural isomorphism $\psi$ will always be the identity.
\end{remark}

The main examples of right module categories over $\cC$ in this paper will be obtained as follows.
\begin{example}
Let $A$ be a (graded) algebra in $\cC$.  
\begin{enumerate}
\item Give a left (graded) module $M$ over $A$ with action map $\alpha_M$, $M\otimes Y$ for $Y \in \cC$ is still naturally a left (graded) module over $A$ with action map $\alpha_M \otimes \id_Y$.
Moreover, one can readily check that $\cG(Y) := -\otimes Y$ is an exact functor for each $Y \in \cC$, so $A\lmod$ is an abelian right module category over $\cC$.
\item The endofunctors $- \otimes Y$ have left and right adjoints (given by tensoring with the corresponding duals of $Y$), so it follows that $- \otimes Y$ sends projective (graded) modules to projective (graded) modules.
In particular, the additive full subcategory $A\lprmod$ of projective (graded) left $A$-modules is an additive right module category over $\cC$.
\item If $\cA$ is an additive right module category over $\cC$, then the bounded homotopy category of (cochain) complexes $\Kom^b(\cA)$ is naturally a triangulated right module category over $\cC$ -- exactness of $- \otimes Y$ comes for free.
In particular, $\Kom^b(A\lprmod)$ is a triangulated right module category over $\cC$.
\end{enumerate}
\end{example}

Given a module category $\cA$ over $\cC$, let $K_0(\cA)$ denote the Grothendieck group of $\cA$; when $\cA$ is additive we take the split Grothendieck group, and when $\cA$ is abelian or triangulated we take the exact Grothendieck group. 
Then $\cG: \cC \ra \cEnd(\cA)$ induces a ring homomorphism 
\begin{align*}
K_0(\cC^{\otimes^\text{op}}) &\ra \End(K_0(\cA)) \\
[Y] &\mapsto \{ [M] \mapsto [M \otimes Y] \}, \quad Y \in \cC,
\end{align*}
which makes $K_0(\cA)$ into a right module over $K_0(\cC)$.
In particular, we will denote the right action of $K_0(\cC)$ on $K_0(\cA)$ by
\[
[M] \cdot [Y] := [M \otimes Y]
\] 
for all $Y \in \cC$ and $M \in \cA$.
It follows that any endofunctor $\cF: \cA \ra \cA$  which commutes with the $\cC$ action on $\cA$ descends to a $K_0(\cC)$-module endomorphism on $K_0(\cA)$.
When $\cA$ is equipped with an internal grading shift functor $\<1\>$, we insisted that the action of $\cC$ on $\cA$ commutes trivially with $\<1\>$.
As such, $K_0(\cA)$ is naturally a right module over $K_0(\cC)[q^{\pm 1}]$, where $[M] \cdot q := [M \<1\>]$.

\subsection{Zigzag and preprojective algebras}\label{zigzagpreproj}
In this section we define fusion quivers and the zigzag and preprojective algerbas associated to them.  Later, in Section \ref{sec:quivercoxeter}, we will associate a fusion quiver to each Coxeter system.

Let $\cC$ be a fusion category, and denote by $\mathfrak{G} = (\mathfrak{G}_0, \mathfrak{G}_1)$ a quiver with set of vertices $\mathfrak{G}_0$ and set of arrows $\mathfrak{G}_1$.
\begin{definition}
A {\it fusion quiver} over $\cC$ is a quiver $\mathfrak{G}$ together with an assignment of non-zero objects $L_e \in \cC$ for each arrow $e \in \mathfrak{G}_1$. 
\end{definition}
%We regard a fusion quiver over $\cC$ as a quiver with each arrow $e \in \mathfrak{G}_1$ labelled by the corresponding object $L_e \in \cC$.
Recall that the double quiver $\overline{\mathfrak{G}}$ of an (ordinary) quiver $\mathfrak{G}$ is constructed by adding an arrow $e^*$ in the opposite direction for each arrow $e$ in the original quiver.
\begin{definition}
The {\it right double} fusion quiver $(\overline{\mathfrak{G}}, \overline{L}_?)$ of $(\mathfrak{G},L_?)$ is the fusion quiver whose underlying quiver is the double quiver of $\mathfrak{G}$, where the added arrows $e^*$ are labelled by the right dual $L_e^*$ of the label $L_e$ of $e$; i.e.\ $\overline{L}_{e^*} = L_e^*$.  The {\it left double} is defined similarly using left duals.  
\end{definition}
We will use the right double in the constructions below, and refer to it simply as the double fusion quiver.
The double fusion quiver comes with an involutive operation on the arrows sending $e \mapsto e^*$, with $(e^*)^* = e$ for all $e \in \mathfrak{G}_1$.
Note that on the labelling-objects $\overline{L}_e$, we have $\overline{L}_{(e^*)^*} = \overline{L}_e^{**} \cong \overline{L}_e$ for all $e \in \mathfrak{G}_1$ \cite[Proposition 4.8.1]{EGNO15}.

Given a double fusion quiver $(\overline{\mathfrak{G}}, \overline{L}_?)$, we can construct its {\it path algebra} (or {\it tensor algebra}) in the ind completion of $\cC$ as follows.
Note that the monoidal unit $\1$ of $\cC$ has a unique algebra structure.
For each $i \in \overline{\mathfrak{G}}_0$, we assign the trivial algebra $\1_i := \1$ and consider the (semisimple) product algebra 
\[
S := \bigoplus_{i \in \mathfrak{G}_0} \1.
\]
For each arrow $e: i \ra j \in \overline{\mathfrak{G}}_1$, consider the $(S,S)$-bimodules ${}_i(\overline{L}_e)_j$, whose underlying object is $\overline{L}_e$, with its natural $(S,S)$ bimodule structure.
Denote the direct sum over all arrows by 
\[
D:= \bigoplus_{e: i \ra j \in \overline{\mathfrak{G}}_1} {}_i(\overline{L}_e)_j
\]
We let $T(\overline{\mathfrak{G}}, \overline{L}_?)$ be the tensor algebra generated by $D$ over $S$:
\[
T(\overline{\mathfrak{G}}, \overline{L}_?) := S \oplus D \oplus (D \otimes_{S} D) \oplus ... ,
\]
which is an algebra in the ind completion of $\cC$.
This algebra is also naturally graded by setting $D$ to have degree $1\in\bbZ$; we refer to this grading as the {\it path length grading} on $T(\overline{\mathfrak{G}}, \overline{L}_?)$.

We will now consider two ``quadratic dual'' quotients of $T(\overline{\mathfrak{G}}, \overline{L}_?)$, which we call the zigzag algebra and the preprojective algebra.
In what follows, for each arrow $e: i \ra j$, we denote the unit map of the duality between ${}_i(\overline{L}_e)_j$ and ${}_j(\overline{L}_{e^*})_i$ by 
            \[
            \cup_e: \1 \ra {}_i(\overline{L}_e)_j \otimes_S {}_j(\overline{L}_{e^*})_i.
            \]
\begin{definition}\label{defn:generalzigzagpreproj}
    The \emph{zigzag algebra} of the double fusion quiver $(\overline{\mathfrak{G}}, \overline{L}_?)$ is the quotient of $T(\overline{\mathfrak{G}}, \overline{L}_?)$ by the two-sided ideal generated by the following objects in $D\otimes_S D$:
    \begin{itemize}
        \item For each pair of vertices $i \neq k$, the object ${}_i(\overline{L}_e)_j \otimes_S {}_j(\overline{L}_e)_k$ is in the ideal.
        \item For each vertex $i \in \overline{\mathfrak{G}}_0$, the cokernel of the map $\1 \xra{\oplus_e \cup_e} \bigoplus_{e} {}_i(\overline{L}_e)_j \otimes_S {}_j(\overline{L}_{e^*})_i$ is in the ideal.
    \end{itemize}
The \emph{preprojective algebra} is the quotient of $T(\overline{\mathfrak{G}}, \overline{L}_?)$ by the two-sided ideal generated by the images of the maps $\1 \xra{\oplus_e \cup_e} \bigoplus_{e} {}_i(\overline{L}_e)_j \otimes_S {}_j(\overline{L}_{e^*})_i$, for each vertex $i\in \mathfrak{G}_0$.
\end{definition}
Note that the zigzag and preprojective algebras both inherit the path-length grading and are moreover quadratic, in the sense that the ideals are generated in path-length degree 2.
\begin{remark}
Just as in \cite{HueKho}, we can define skew-zigzag (and skew-preprojective) algebras in $\cC$ by introducing signs in the map $\1 \xra{\oplus_e \cup_e} \bigoplus_{e} {}_i(\overline{L}_e)_j \otimes_S {}_j(\overline{L}_{e^*})_i$.  We will not use them explicitly here, so we omit the details and refer to \cite{bapat2020thurston} for a more a discussion of classical skew and ordinary zigzag algebras.
\end{remark}

\begin{remark}
When the underlying fusion category $\cC$ is the category of vector spaces, the (classical) preprojective algebra admits another construction as an endomorphism algebra  in the representation category of the quiver \cite[Theorem 2.3]{CB_preproj}.  A similar construction can be done here, using instead the \emph{internal} hom spaces in the representation category of fusion quivers (see \cite{EGNO15} for the definition of internal homs and \cite{EH_fusionquiver} for the definition of fusion quiver representations).  
\end{remark}

\section{Zigzag algebras associated to Coxeter diagrams}\label{zigzagCoxeter}
In this section, to each Coxeter diagram $\Gamma$, we will associate a fusion category $\cC(\Gamma)$ in which we will construct our zigzag algebra $\zig(\Gamma)$.
This fusion category $\cC(\Gamma)$ was also used in \cite{heng2023coxeter}.

\subsection{Temperley--Lieb--Jones category at root of unity} \label{subsection: TLJ}
We first briefly recall the fusion categories that will be used to construct $\cC(\Gamma)$.
These fusion categories are strict fusion categories known as the \emph{Temperley--Lieb--Jones category evaluated at} $q = e^{i\pi/n}$, which we denote as $\TLJ_n$.
More precisely, it is a diagrammatic category constructed as the additive completion of the category of Jones--Wenzl projectors that is semi-simplified by killing the negligible $(n-1)$th Jones--Wenzl projector.
We will only provide the minimal description of this category required for this paper, where the reader can refer to \cite{chen2014temperleylieb} and \cite{wang_2010, turaev_2016} for more details.
\begin{remark}
Readers familiar with quantum groups could also use the semi-simplified category $\overline{\rep}(U_q(\mathfrak{s}\mathfrak{l}_2))$ of $U_q(\mathfrak{s}\mathfrak{l}_2)$-tilting modules at $q = e^{i\frac{\pi}{n}}$, which is equivalent to $\TLJ_n$ as (braided) fusion categories; see \cite[Chapter XI, Section 6]{turaev_2016} and \cite[Section 5.5]{SnyTing09}.
\end{remark}

The Temperley--Lieb--Jones category $\TLJ_n$ at $q = e^{i\pi/n}$ \cite[Definition 5.4.1]{chen2014temperleylieb} is a strict, $\bbC$-linear fusion category that is generated additively and monoidally by the following $n-1$ simple objects:
\[
\Irr(\TLJ_n) = \{\Pi_0, \Pi_1, ..., \Pi_{n-2}\},
\]
with $\Pi_0$ being the monoidal unit.
The \emph{fusion rules} in $\TLJ_n$ are described as follows:
\begin{equation} \label{eq: fusion rule}
\Pi_a \otimes \Pi_b \cong 
	\begin{cases}
	\Pi_{|a-b|} \oplus \Pi_{|a-b| + 2} \oplus \cdots \oplus \Pi_{a+b}, &a+b \leq n-2; \\
	 \Pi_{|a-b|} \oplus \Pi_{|a-b| + 2} \oplus \cdots \oplus \Pi_{2n-(a+b)-4}, &a+b > n-2.
	\end{cases}
\end{equation}
Notice that no simple object appear more than once in the semi-simple decomposition above.
All the simples $\Pi_a$ are also self-dual (so are all objects in $\TLJ_n$), so we have
\[
\Hom_{\TLJ_n}(\Pi_a \otimes \Pi_b, \Pi_c) \cong \Hom_{\TLJ_n}(\Pi_c, \Pi_a \otimes \Pi_b) \cong 
\begin{cases}
\bbC, &\text{ if } \Pi_c \overset{\oplus}{\subseteq} \Pi_a \otimes \Pi_b; \\
0,  &\text{ otherwise}.
\end{cases}
\]
In particular, when $a = b$ and $c=0$, the hom spaces are generated respectively by the counit and unit of the self-duality, which we usually denote by ``caps'' and ``cups'': 
\[
\cap: \Pi_a \otimes \Pi_a \ra \Pi_0, \quad \cup: \Pi_0 \ra \Pi_a \otimes \Pi_a.
\]
The cases where $a=1,n-2$ and $n-3$ will also be used later, so we record them here:
\begin{itemize}
\item $a=1$:
\begin{equation} \label{eq: 1 b fusion rule}
\Pi_1 \otimes \Pi_b \cong \Pi_b \otimes \Pi_1 \cong
	\begin{cases}
	\Pi_1, &b=0 \\
	\Pi_{n-3}, &b=n-2\\
	\Pi_{b-1} \oplus \Pi_{b+1}, &\text{ otherwise}.
	\end{cases}
\end{equation}
\item $a = n-2$:
\begin{equation} \label{eq: z/2 fusion rule}
\Pi_{n-2} \otimes \Pi_b \cong \Pi_{n-2-b}.
\end{equation}
\item $a=n-3$: \eqref{eq: 1 b fusion rule} and \eqref{eq: z/2 fusion rule} imply that
\begin{equation} \label{eqn: n-3 b fusion rule}
\Pi_{n-3} \otimes \Pi_b \cong \Pi_{n-2} \otimes \Pi_1 \otimes \Pi_b \cong 
	\begin{cases}
	\Pi_{n-3}, &b=0 \\
	\Pi_{1}, &b=n-2\\
	\Pi_{n-2-b-1} \oplus \Pi_{n-2-b+1}, &\text{ otherwise}.
	\end{cases}
\end{equation}
\end{itemize}
Note that \eqref{eq: fusion rule} is also symmetric along $a$ and $b$, so we have $\Pi_a \otimes \Pi_b \cong \Pi_b \otimes \Pi_a$.
As such, its fusion ring $K_0(\TLJ_n)$ is a commutative ring.
Nonetheless, the category $\TLJ_n$ is not symmetric, but instead \emph{braided}; see \cite[Chapter XII, Section 6]{turaev_2016}.

Moreover, one sees that if $a$ and $b$ are both even, then their semisimple decomposition contains only even-labelled objects.
It follows that the full subcategory of $\TLJ_n$ generated the even-labelled object forms a fusion subcategory of $\TLJ_n$; we denote this subcategory by $\TLJ_n^{even}$.

\begin{example} \label{eg:TLJeven}
	\begin{itemize}[(i)]
	\item When $n = 3$, the full subcategory $\TLJ_3^{even}$ is equivalent to the category $\vec_\bbC$ of finite dimensional vector spaces over $\bbC$.
	\item When $n = 5$, the full subcategory $\TLJ_5^{even}$ is equivalent to the Fibonacci (or golden ratio) fusion category $\Fib$. This category has two simple objects $\1 = \Pi_0$ (monoidal unit) and $\Pi := \Pi_2$, where $\Pi \otimes \Pi \cong \1 \oplus \Pi$.
	\end{itemize}
\end{example}

%For $k \in \Z_{\geq 0}$, the $k$th quantum number
%\[
%[k]_q := \frac{q^{k} - q^{-k}}{q - q^{-1}}
%\]
%also satisfy the recurrence relation of the Chebyshev polynomial, where $[k+1]_q = \Delta_k(q+q^{-1})$.
%
The fusion ring of $\TLJ_n$ can be identified with the quotient polynomial ring
\[
K_0(\TLJ_n) \cong \Z[d]/\< \Delta_{n-1}(d) \>,
\]
where $\Delta_k(d)$ is the Chebyshev polynomial (of second kind), defined by the following recurrence relation
\[
\Delta_0(d) = 1, \quad \Delta_1(d) = d, \quad \Delta_{k+1}(d) = d \Delta_k(d) - \Delta_{k-1}(d).
\]
Note that each basis element $[\Pi_k]$ is mapped to $\Delta_k(d)$.
Under this identification, the fusion ring of the fusion subcategory $\TLJ_n^{even}$ is therefore the subring generated by the even degree polynomials.
Note also that $\FPdim(\Pi_k) = \Delta_k(2\cos(\pi/n))$ (same for $\TLJ_n^{even}$).

\subsection{Fusion categories associated to Coxeter diagrams}
\label{sec:fusioncoxeter}
Let $\Gamma$ be a Coxeter diagram.
We denote its set of vertices by $\Gamma_0$ and its set of edges by $\Gamma_1$.
The label on each edge $e=(s,t) \in \Gamma_1$ is denoted by $m_e$ and let 
\[
M(\Gamma) := \{ m_e \mid e \in \Gamma_1 \} \subset \{3, 4,...,\infty\} 
\]
denote the set of labels that appears on the edges of $\Gamma$.

For each $n \neq \infty$, let us denote
\[
\cC_n := 
	\begin{cases}
	\TLJ_n, &\text{when $n$ is even}; \\
	\TLJ^{even}_n, &\text{when $n$ is odd},	
	\end{cases}
\]
with $\TLJ_n$ and $\TLJ_n^{even}$ defined in \cref{subsection: TLJ}.
To each Coxeter diagram $\Gamma$, we associate the following fusion category:
\begin{equation}\label{eqn:fusioncoxeter}
\cC(\Gamma) := 
	\begin{cases}
	\TLJ^{even}_3 \cong \vec_\bbC, &\text{if } M(\Gamma) = \emptyset \text{ or } M(\Gamma) = \{\infty\}; \\
	\bigboxtimes_{n \in M(\Gamma) \setminus \{\infty\}} \cC_{n}, &\text{ otherwise}.
	\end{cases}
\end{equation}
Note that when $\Gamma$ is a symmetric Kac--Moody Coxeter diagram, then $M(\Gamma) \subseteq \{3, \infty\}$ or $M(\Gamma) = \emptyset$, and we have that $\cC(\Gamma) = \TLJ^{even}_3 \cong \vec_{\bbC}$.

Throughout we use $\1$ to denote the monoidal unit of $\cC(\Gamma)$.
To differentiate between the simple objects in different $\TLJ_n$, we use $\nPi_i$ to denote the simple objects in $\TLJ_n$.
Note that for each $n \in M(\Gamma)$, $\TLJ_n$ (and hence $\TLJ^{even}_n$) is naturally identified with a full, fusion subcategory of $\cC(\Gamma)$.
As such, we shall abuse notation and also use $\nPi_i$ to denote the simple object in $\cC(\Gamma)$; it has $\nPi_i$ in the $\cC_n$ component and $\nPi[j]_0$ in all other $\cC_j$ components.

With $m$ fixed, the condition $\nPi[m]_c \overset{\oplus}{\subseteq} \nPi[m]_a \otimes \nPi[m]_b$ holds in $\cC(\Gamma)$ if and only if it holds in $\cC_m$.
When so, the hom spaces between $\nPi[m]_a \otimes \nPi[m]_b$ and $\nPi[m]_c$ in $\cC(\Gamma)$ are therefore also one-dimensional.
More generally, 
The unit and counit associated to the self-duality of the objects in $\cC(\Gamma)$ will still be denoted by the cap $\cap$ and the cup $\cup$.

\begin{example} \label{eg:fusioncatforcox}
	\begin{itemize}[(i)]
	\item Let
		$\Gamma = \adjustbox{scale=0.8}{ 
		% https://q.uiver.app/#q=WzAsNCxbMCwwLCJcXGJ1bGxldCJdLFsxLDAsIlxcYnVsbGV0Il0sWzIsMCwiXFxidWxsZXQiXSxbMywwLCJcXGJ1bGxldCJdLFswLDEsIjUiLDAseyJzdHlsZSI6eyJoZWFkIjp7Im5hbWUiOiJub25lIn19fV0sWzEsMiwiNSIsMCx7InN0eWxlIjp7ImhlYWQiOnsibmFtZSI6Im5vbmUifX19XSxbMiwzLCI1IiwwLHsic3R5bGUiOnsiaGVhZCI6eyJuYW1lIjoibm9uZSJ9fX1dXQ==
	\begin{tikzcd}[every arrow/.append style = {shorten <= -.5em, shorten >= -.5em}]
		\bullet & \bullet & \bullet & \bullet
		\arrow["5", no head, from=1-1, to=1-2]
		\arrow["5", no head, from=1-2, to=1-3]
		\arrow["\infty", no head, from=1-3, to=1-4]
		\end{tikzcd}
		}$.
		Then $M(\Gamma) = \{5\}$ and $\cC(\Gamma) = \TLJ_5^{even}$ is the Fibonacci category in \cref{eg:TLJeven}.
	\item Let
		% https://q.uiver.app/#q=WzAsMyxbMCwwLCJcXGJ1bGxldCJdLFsxLDAsIlxcYnVsbGV0Il0sWzIsMCwiXFxidWxsZXQiXSxbMCwxLCI0IiwwLHsic3R5bGUiOnsiaGVhZCI6eyJuYW1lIjoibm9uZSJ9fX1dLFsxLDIsIjUiLDAseyJzdHlsZSI6eyJoZWFkIjp7Im5hbWUiOiJub25lIn19fV1d
		$
		\Gamma = \begin{tikzcd}[every arrow/.append style = {shorten <= -.5em, shorten >= -.5em}]
		\bullet & \bullet & \bullet
		\arrow["4", no head, from=1-1, to=1-2]
		\arrow["5", no head, from=1-2, to=1-3]
		\end{tikzcd}
		$.
		Then $M(\Gamma) = \{4,5\}$ and $\cC(\Gamma) = \TLJ_4 \boxtimes \TLJ^{even}_5$. Up to isomorphisms there are six simple objects (the two equalities are actual equalities according to the convention introduced above):
		\[
		\Irr(\cC(\Gamma)) = 
				\left\{
					\nPi[4]_0 = \nPi[5]_0 = \nPi[4]_0 \boxtimes \nPi[5]_0, 
					\nPi[4]_1, \nPi[4]_2,
					\nPi[5]_2,
					\nPi[4]_1 \boxtimes \nPi[5]_2, \nPi[4]_2 \boxtimes \nPi[5]_2
				\right\}
		\]
	\end{itemize}
\end{example}

\subsection{Fusion quivers and zigzag algebras associated to Coxeter diagrams}
\label{sec:quivercoxeter}
To each Coxeter diagram $\Gamma$, we associate a fusion quiver $\mathfrak{G}(\Gamma)$ over $\cC(\Gamma)$ defined as follows.
For each edge $e = (s,t) \in \Gamma_1$ with label $m_{s,t}$, denote
\begin{equation} \label{eqn:coxquiverlabel}
\Pi(e) :=
	\begin{cases}
	\nPi[m_{s,t}]_{(m_{s,t}-3)}, &\text{ if } m_{s,t} \neq \infty; \\
	\1 \oplus \1, &\text{ if } m_{s,t} = \infty.
	\end{cases}
\end{equation}
We modify $\Gamma$ by replacing the label $m_{s,t}$ on each edge $e = (s,t)$ with the corresponding object $\Pi_e$ defined above.
The result is a graph labeled by objects of the fusion category.  By choosing an orientation for each edge, we obtain a fusion quiver $\mathfrak{G}(\Gamma)$.
The zigzag algebra associated to $\Gamma$ is defined as the zigzag algebra of the doubled fusion quiver $\overline{\mathfrak{G}(\Gamma)}$ (see \cref{defn:generalzigzagpreproj}); it depends only on the labelled graph, not on the choice of orientation.
In fact, since all objects in $\cC(\Gamma)$ are self-dual, the pair of arrows $e$ and $e^*$ in opposite direction will always be labelled by the same object $\Pi_e$.

\begin{remark} \label{rem:fusioncatcoxeter}
Our choice of fusion category $\cC(\Gamma)$ and fusion quiver associated to a Coxeter diagram are convenient, but not canonical. 
The important point is that an edge $e = (s,t)$ with label $m_{s,t}$ is labelled by an object $\Pi(e)$ with $\FPdim(\Pi(e)) = 2\cos(\pi/m_{s,t})$.  
Moreover, if $m_{s,t}=\infty$ one can choose any object $\Pi(e)$ with $\FPdim(\Pi(e)) \geq 2$.
What we prove in the sequel will require minor modification for different choices of $\Pi(e)$ and $\cC(\Gamma)$. 
A convenient feature of the fusion category $\cC(\Gamma)$ we use is that it is equivalent to the category of finite dimensional vector spaces whenever all $m_{ij}\in \{2,3,\infty\}$.
%Namely, given a Coxeter graph $\Gamma$ with $M(\Gamma) \subset \{3,4,..., \infty}$ the set of edge labels, one could choose any fusion category $\cC$ that contains objects $X_m \in \cC$ with $\FPdim(X_m) = 2\cos(\pi/m)$ for all $m \in M(\Gamma)$ ($2\cos(\pi/\infty) := 2$), and replace each edge with label $m$ by $X_m$.
\end{remark}
\subsection{Zigzag algebras from Coxeter diagrams} \label{subsect: zigzag alg}

In the rest of the paper we will solely work with zigzag algebras constructed from Coxeter diagrams as above, and not comment further on the preprojective algebras.   In this section we give an alternative description of the zigzag algebra by generators and relations.  
%Since the path-length grading on this algebra will be crucial in later sections, instead of viewing the zigzag algebra as a graded algebra in $\cC(\Gamma)$, we shall directly construct it as an algebra in the category of $\bbZ$-graded objects in $\cC(\Gamma)$, which we define below.  (This should be compared with the notion of a $\bbZ$ graded algebra versus an algebra in the category of $\bbZ$-graded vector spaces.)
%\begin{definition}
%Let $\cC(\Gamma)$ be the fusion category associated to the Coxeter graph $\Gamma$.
%Then $g\cC(\Gamma)$ will denote its associated semisimple tensor category of $\bbZ$-graded objects.
%More precisely, its objects are of the form $\bigoplus_{i \in \bbZ} A_i \<i\>$ for $A_i \in \cC(\Gamma)$, where only finitely many $A_i$ are non-zero.
%An object in $g\cC(\Gamma)$ of the form $A\<m\>$ is said to be homogeneous in degree $m$.
%The morphisms in this category are given by degree zero morphisms, so that
%\[
%\Hom_{g\cC(\Gamma)}\left( A\<i\>, B\<j\>\right) = 
%	\begin{cases}
%	\Hom_{\cC(\Gamma)}(A,B), &\text{if } i=j; \\
%	0, &\text{otherwise}.
%	\end{cases}
%\]
%The monoidal structure is defined additively over the $\bbZ$-grading (and distributive over $\oplus$): $A\<i\> \otimes B\<j\> := (A\otimes B) \<i+j\>$.
%\end{definition}
%In particular, when $\cC(\Gamma) = \vec_\bbC$, $g\cC(\Gamma)$ is just the category of $\bbZ$-graded finite-dimensional vector spaces (over $\bbC$).
%Note also that $\cC(\Gamma)$ sits naturally as a tensor subcategory of $g\cC(\Gamma)$, consists of objects that are homogeneous in degree zero.

We first introduce some notation.
Given an object $Y \in \cC(\Gamma)$, we shall use $Y\<k\>$ (with $k \in \bbZ$) to denote a graded object homogeneous in degree $k$.
The point is that we will define a graded object $M$ in $\cC(\Gamma)$ as a direct sum
\[
M = \bigoplus_{i \in \Xi} M_i\<k_i\>,
\]
over some index set $\Xi$, with each $M_i \in \cC(\Gamma)$ and $k_i \in \bbZ$, so that the summand of homogeneous degree $k$ of $M$ is given by $\bigoplus_{k_i = k} M_i\<k_i\>$.
We now define some distinguished graded objects in $\cC(\Gamma)$.
In what follows, $\1$ denotes the monoidal unit of $\cC(\Gamma)$ and $\Pi(e) \in \cC(\Gamma)$ is as defined in \eqref{eqn:coxquiverlabel}.
\begin{itemize}
\item For each $s \in \Gamma_0$, we define two graded objects of $\cC(\Gamma)$ (with different degrees) as follows:
\begin{enumerate}[(i)]
\item The \emph{constant path object on vertex $s$} with degree zero:
	\[
	e_s := \1\<0\>.
	\]
\item The \emph{loop object on vertex $s$ (starts and ends at $s$)} with degree two:
	\[
	X_s := \1\<2\>.
	\]
\end{enumerate}
\item For each $e = (s,t) = (t,s) \in \Gamma_1$, we define two graded objects of $\cC(\Gamma)$ with degree one as follows:
\begin{enumerate}[(i)]
\item The \emph{path object that starts at $s$ and ends at $t$}:
	\[
	(s|t) := \Pi(e)\<1\>.
	\]
\item The \emph{path object that starts at $t$ and ends at $s$}:
	\[
	(t|s) := \Pi(e)\<1\>.
	\]
\end{enumerate}
\end{itemize}
\begin{remark}
As graded objects in $\cC(\Gamma)$, there is no difference between the objects $(s|t)$ and $(t|s)$; the purpose of having them separately is that they will behave differently under the multiplication map of the algebra that we will next define.
\end{remark}
The \emph{zigzag algebra} $\zig(\Gamma)$ associated to $\Gamma$ can be constructed explicitly as a graded algebra in $\cC(\Gamma)$ as follows.
As a graded object in $\cC(\Gamma)$, it is given by
\[
\zig(\Gamma) := \left( \bigoplus_{s \in S} e_s \oplus X_s \right) \oplus \left( \bigoplus_{e=(s,t) \in \Gamma_0} (s|t) \oplus (t|s) \right).
\]
(Note that we do not need the work in the ind completion.) 
As our notation suggests, one should think of the summands of the form $e_s, X_s, (s|t)$ and $(t|s)$ as ``paths'' (with path length gradings) and we call each of these summands a \emph{path summand} of $\zig(\Gamma)$.

The multiplication map on $\zig(\Gamma)$
\begin{equation} \label{eqn:zigzagmult}
\mu : \zig(\Gamma) \otimes \zig(\Gamma)\ra \zig(\Gamma)
\end{equation} 
will be defined as if we were concatenating paths -- multiplication is zero if the ends do not meet --, modulo some relations.  
Precisely, the multiplication map $\mu$ on $E \otimes E'$ for each pair of path summands $E$ and $E'$ in $\zig(\Gamma)$ is defined as follows:
\begin{itemize}
\item if the ending vertex of $E$ does not match with the starting vertex of $E'$, then $\mu|_{E\otimes E'}$ is the zero map;
\item if the total degree of $E\otimes E'$ is greater than or equal to 3, then $\mu|_{E\otimes E'}$ is the zero map.
\end{itemize}
What remains is to define the product of path summands $E$ and $E'$ whose end points match and which have total degree less than or equal to 2.
\begin{itemize}
\item Consider the case where both $E$ and $E'$ have degree 1, so that $E = (s|t)$ and $E' = (t|s')$. \\
If $s \neq s'$, then we define $\mu|_{E\otimes E'}$ to be the zero map.\\
Otherwise, $s = s'$ and so $E = E' = \Pi(e)\<1\>$ as graded objects in $\cC(\Gamma)$, with $e = (s,t) \in \Gamma_1$.
We define $\mu|_{E\otimes E'}: (s|t) \otimes (t|s) \ra X_s$ to be counit map of the self-duality $\cap: \Pi(e) \otimes \Pi(e) \ra \1$ of their underlying (non-graded) object $\Pi(e)$ in $\cC(\Gamma)$; note that $\mu|_{E\otimes E'}$ is indeed grading preserving.
\item We are now left with the case where at least one of $E$ or $E'$ is equal to $e_s$ for some $s \in \Gamma_0$, hence by definition of $e_s = \1$ we have
\[
E'' := E \otimes E' =
	\begin{cases}
	E, &\text{ if } E' = e_s; \\
	E', &\text{ if } E = e_s
	\end{cases}.
\]
In this case we just define $\mu|_{E\otimes E'}: E\otimes E' \ra E''$ to be the identity map of $E''$; the fact that $\mu|_{E\otimes E'}$ is grading preserving is immediate since $e_s$ is homogeneous degree zero.
\end{itemize}
The associativity of $\mu$ is a straightforward verification, since (by definition) the multiplication of three path summands is zero unless at least one of them is the constant path $e_s$.

Let $E$ be a path summand in $\zig(M)$.
Note that $\mu|_{e_s \otimes E}$ is non-zero if and only if $E$ is a path summand with starting vertex $s$.
Moreover, when $\mu|_{e_s \otimes E} : e_s \otimes E = E \ra E$ is non-zero, it must be the identity map by definition.
Similarly $\mu|_{E \otimes e_s}$ is non-zero if and only if $E$ is a path summand with ending vertex $s$, and when it is non-zero it must be the identity map.

Let us denote ${}_sP$ (resp. $P_s$) to be the direct sum of the path summands of $\zig(\Gamma)$ starting (resp. ending) with vertex $s$, so that, as graded objects of $\cC(\Gamma)$,
\[
\zig(\Gamma) = \bigoplus_{s \in \Gamma_0} {}_sP \quad \text{and} \quad 
\zig(\Gamma) = \bigoplus_{s \in \Gamma_0} P_s.
\]
It follows that $\mu|_{e_s \otimes \zig(\Gamma)}$ only maps into the summand ${}_sP$ of $\zig(\Gamma)$ and 
	\[
	\mu|_{e_s \otimes {}_tP} = 
	\begin{cases}
	\id_{{}_sP}, &\text{ if }t = s; \\
	0          , &\text{ otherwise.}
	\end{cases};
	\]
Similarly, $\mu|_{\zig(\Gamma) \otimes e_s}$ only maps into the summand $P_s$ of $\zig(\Gamma)$ and 
	\[
	\mu|_{P_t \otimes e_s} = 
	\begin{cases}
	\id_{P_s}, &\text{ if }t = s; \\
	0          , &\text{ otherwise.}
	\end{cases}.
	\]
%As we shall see later, the objects ${}_sP$ and $P_s$ will turn out to be right module objects and left module objects of $\zig(M)$ respectively.

The unit map $\eta: \1 \ra \zig(\Gamma)$ is defined by the following composition (the final morphism being the trivial inclusion):
\[
\eta: \1 \xra{ \left[\id_{\1} \right]_{s\in \Gamma_0} } \bigoplus_{s \in \Gamma_0} e_s \hookrightarrow \zig(\Gamma).
\]
The left and right unital conditions follows essentially from the path multiplication rule of $\mu$.
Namely, using the decomposition $\zig(\Gamma) = \bigoplus_{t \in \Gamma_0} {}_tP$, we have the following commutative diagram which amounts to the identity on $\bigoplus_{t \in \Gamma_0} {}_tP$ (hook arrow $\hookrightarrow$ denotes summand inclusion):
\[
\begin{tikzcd}[row sep = large, column sep = large]
%row 1
\1 \otimes \bigoplus_{t \in \Gamma_0} {}_tP 
	\ar{rr}{\left[\id_{\1} \right]_{s\in \Gamma_0} \otimes \id_{\zig(\Gamma)}} &
{}
	{} &
\left( \bigoplus_{s \in \Gamma_0} e_s \right) \otimes \bigoplus_{t \in \Gamma_0} {}_tP 
	 \ar[equal]{d} 
	 \ar[hook]{r} &
\zig(\Gamma) \otimes \zig(\Gamma)
	 \ar{d}{\mu} 
	 \\
%row 2
{}
	{} &
{}
	{} &
\bigoplus_{s,t \in \Gamma_0} e_s \otimes {}_tP
	\ar{r}{\left[\mu|_{e_s \otimes {}_tP } \right]_{s,t\in \Gamma_0}}  &
\bigoplus_{s \in \Gamma_0} {}_sP 
	{}
\end{tikzcd}
\]
The right unital condition is similar.

We now define a comultiplication and a counit, which makes $\zig(\Gamma)$ also a Frobenius algebra.
The $(\zig(\Gamma),\zig(\Gamma))$-bimodule morphism 
\begin{equation}\label{eqn:frobcomultcomponent}
\gamma_s : \zig(\Gamma) \ra P_s \otimes {}_sP \<-2\>
\end{equation}
is defined as the morphism from each summand $e_t$ for $t \in \Gamma_0$ as follows:
\begin{itemize}
	\item if $t=s$,
	\[
	e_s \xra{
		\begin{bsmallmatrix}
		\id_\1 \\
		\id_\1
		\end{bsmallmatrix}
		} 
			(e_s \otimes X_s \<-2\>) \oplus (X_s \otimes e_s \<-2\>) \overset{\oplus}{\subseteq} P_s \otimes {}_sP\<-2\>;
	\]
	\item if $(s,t) \in \Gamma_1$:
	\[
	e_{t} 
		\xra{
		\cup
		} 
		\left( (t | s) \otimes (s | t) \<-2\>\right) \overset{\oplus}{\subseteq} P_s \otimes {}_sP\<-2\>,
	\]
	where $\cup$ denotes the unit map of the self-duality of $\Pi(e)$ (noting that $\Pi(e) \otimes \Pi(e) = (t | s) \otimes (s | t) \<-2\>$);
	\item otherwise $e_t$ is mapped to zero.
\end{itemize}
Direct computation shows that this uniquely determines a well-defined bimodule morphism.
We then assemble these into a bimodule morphism defining the comultiplication:
\begin{equation} \label{eqn:frobcounit}
\gamma : \zig(\Gamma) \xra{[\gamma_s]_{s\in\Gamma}} \bigoplus_{s\in \Gamma_0} P_s \otimes {}_sP \<-2\> \overset{\oplus}{\subseteq} \zig(\Gamma) \otimes \zig(\Gamma) \<-2\>.
\end{equation}
The counit is a morphism
\[
\omega: \zig(\Gamma)\<-2\> \ra \1
\]
which on the summands $X_s\<-2\> = \1$ is the identity morphism and is zero on all other summands.

Exactly as in \cite{HueKho}, we have the following.
\begin{proposition}
$(\zig(\Gamma), \mu, \eta)$ as defined above defines a graded algebra in $\cC(\Gamma)$. Moreover, the structures $(\gamma, \omega)$ makes $\zig(\Gamma)$ a Frobenius algebra.
\end{proposition}
The following is an explicit example of a zigzag algebra associated to a rank two Coxeter group.
The reader is encouraged to check that $\zig(\Gamma)$ agrees (via the identification $\cC(\Gamma) \cong \vec_\bbC$) with the classical definition of zigzag algebras when $\Gamma$ is a symmetric Kac--Moody diagram.
\begin{example}
Let $\Gamma = I_2(5) = 
	\begin{tikzcd}[every arrow/.append style = {shorten <= -.2em, shorten >= -.2em}, column sep = small]
	1 & 2
	\arrow["5", no head, from=1-1, to=1-2]
	\end{tikzcd}$. 
	Then 
	\begin{align*}
	\zig(\Gamma) 
		&= e_1 \oplus e_2 \oplus (1|2) \oplus (2|1) \oplus X_1 \oplus X_2 \\
		&= \Pi_0 \oplus \Pi_0 \oplus \Pi_2\<1\> \oplus \Pi_2\<1\> \oplus \Pi_0\<2\> \oplus \Pi_0\<2\>.
	\end{align*}
	The non-zero morphism in $\mu: \zig(\Gamma) \otimes \zig(\Gamma) \ra \zig(\Gamma)$ are given by the summands
	\begin{alignat*}{3}
	e_1 \otimes e_1 \xra{\id_{\Pi_0}} e_1 \qquad 
		& e_2 \otimes (2|1) \xra{\id_{\Pi_2}\<1\>} (2|1) \qquad
		&& (1|2) \otimes (2|1) \xra{\cap\<2\>} X_1 \\
	e_2 \otimes e_2 \xra{\id_{\Pi_0}} e_2 \qquad
		& (1|2) \otimes e_2 \xra{\id_{\Pi_2}\<1\>} (1|2) \qquad 		
		&& (2|1) \otimes (1|2) \xra{\cap\<2\>} X_2, \\
	{}
		& e_1 \otimes (1|2) \xra{\id_{\Pi_2}\<1\>} (1|2) \qquad
		&& {} \\
	{}
		& (2|1) \otimes e_1 \xra{\id_{\Pi_2}\<1\>} (2|1) \qquad
		&& {}
	\end{alignat*}
	where $\cap: \Pi_2 \otimes \Pi_2 \ra \Pi_0$ denotes the counit of the self-duality of $\Pi_2$.
\end{example}

\subsection{Projective graded modules over zigzag algebra}
For the rest of this section, we will fix the Coxeter graph $\Gamma$, so we shall drop the descriptor ``$\Gamma$'' in our notations: $\B:= \B(\Gamma)$, $\cC:= \cC(\Gamma)$ and $\zig:= \zig(\Gamma)$ will denote the \emph{graded} algebra equipped with the multiplication and unit map $\mu$ and $\eta$ respectively, as defined in the previous section.

Since $\zig$ is a graded algebra, by our convention all (bi)modules over $\zig$ will be graded, and the category of left (resp.\ right) graded modules over $\zig$ will simply be denoted by $\zig\lmod$ (resp.\ $\rmod \zig$). 

Recall the objects $P_s$ and ${}_sP$ defined in the previous section are naturally left and right graded modules over $\zig$ respectively.
%The multiplication map $\mu$ on $\zig$ restricts to $\mu|_{\zig \otimes P_s} : \zig \otimes P_s \ra P_s$, where it follows from the path concatenation definition of $\mu$ that it indeed maps into $P_s$.
%The left module condition follows easily from the unital and associativity structure of the algebra object $\zig$.
%Similarly, ${}_sP$ can be endowed with a right module structure over $\zig$ using $\mu|_{{}_sP \otimes \zig} : {}_sP \otimes \zig \ra {}_sP$.

\begin{proposition} \label{tensor sPt}
Let ${}_sP_t$ be the graded object in $\cC$ defined as the direct sum of all the path summands (grading included) in $\zig$ that start with $s$ and end with $t$.
Together with the morphism $\mu|_{{}_s P \otimes P_t}: {}_sP \otimes P_t \ra {}_sP_t$, we have that ${}_sP_t$ is the coequaliser of the following diagram of grading preserving morphisms of graded objects in $\cC$:
\[
\begin{tikzcd}[column sep=2cm]
{}_sP \otimes \zig \otimes P_t \ar[r, shift left=0.75ex, "\mu|_{{}_sP\otimes \zig} \otimes \id"] \ar[r, shift right=0.75ex, swap, "\id \otimes \mu|_{\zig\otimes P_t}"] & {}_sP\otimes P_t.
\end{tikzcd}
\]
In particular, we have ${}_sP \otimes_{\zig} P_t \cong {}_sP_t$ as graded objects in $\cC$. 
\end{proposition}
\begin{proof}
Throughout this proof, $\hookrightarrow$ and $\twoheadrightarrow$ will be use the denote the obvious inclusion and projection map of the summands respectively.
Let $\varphi: {}_sP \otimes P_t \ra \xi$ be a map satisfying 
\[
\varphi\circ \left( \mu|_{{}_sP \otimes \zig} \otimes \id_{P_t} \right) =
\varphi\circ \left( \id_{{}_sP} \otimes \mu|_{\zig \otimes P_t} \right).
\]
Consider the commutative diagram
\[
\begin{tikzcd}[column sep = 2.5cm, row sep = large]
%% row 1
%%
{}_sP \otimes \zig \otimes P_t 
	\ar[r, shift left=0.75ex, "\mu|_{{}_sP \otimes \zig} \otimes \id_{P_t}"] 
	\ar[r, shift right=0.75ex, swap, "\id_{{}_sP} \otimes \mu|_{\zig \otimes P_t}"] 
	\ar[d, hook] & 
{}_sP \otimes P_t 
	\ar[r, "\mu|_{{}_sP \otimes P_t}"] 
	\ar[d, hook] &
{}_sP_t 
	\ar[d, hook]\\
%%
%% row 2
%%
\zig \otimes \zig \otimes \zig 
	\ar[r, shift left=0.75ex, "\mu \otimes \id_{\zig}"] 
	\ar[r, shift right=0.75ex, swap, "\id_{\zig} \otimes \mu"] & 
\zig \otimes \zig 
	\ar[r, "\mu"]
	\ar[d, two heads]
& \zig 
	\ar[d, dashed, "\exists !"] \\
%%
%% row 3
%%
{}
	{} &
{}_sP \otimes P_t
	\ar[r, "\varphi"] &
\xi
	{} 
\end{tikzcd}
\]
where the existence of the unique map follows from $\zig \otimes_{\zig} \zig \cong \zig$.
It follows that the induced map from ${}_sP_t$ to $\xi$ is indeed unique and so ${}_sP_t$ does satisfy the required universal property.
\end{proof}

\begin{remark}
Note that ${}_sP_t$ is given by
\[
{}_sP_t = 
\begin{cases}
e_s \oplus X_s, &\text{if } s=t; \\
(s|t), &\text{if } (s,t) \in \Gamma_1; \\
0; &\text{otherwise},
\end{cases}
\]
with the path summands equipped with the path length grading.
\end{remark}

\begin{definition} \label{defn: prmod}
We define $\zig\lprmod$ to be the smallest additive, full subcategory of $\zig\lmod$ containing the objects
	\[
	\{P_s\<k\> \otimes E \mid 
		s \in \Gamma_0, \
		k\in \Z, \ 
		E \text{ a simple in } \cC\}.
	\]
Similarly, $\rprmod \zig$ is the smallest additive, full subcategory of $\rmod\zig$ containing 
	\[
	\{E \otimes {}_sP\<k\> \mid 
		s \in \Gamma_0, \
		k\in \Z, \ 
		E \text{ a simple in } \cC\}.
	\]
\end{definition}
Note that by construction, $\zig\lprmod$ (resp.\ $\rprmod\zig$) is closed under the action of $\cC$ on the right (resp.\ left) given by tensoring (in $\cC)$; in other words, it is a right (resp.\ left) module category over $\cC$.
Both $\zig\lprmod$ and $\rprmod\zig$ are also closed under the grading shift functor $\<k\>$ by construction.

\begin{remark}
Although we will not need this fact, we note that $\zig\lprmod$ (resp.\ $\rprmod\zig$) is in fact the category of projective graded left (resp.\ right) $\zig$-modules.
Indeed, $\cC$ is a fusion category, so the (graded) algebra $\zig$ is a projective (graded) module over itself and moreover any projective (graded) $\zig$-module appears as a direct summand of $\zig \otimes Y$ for some $Y \in \cC$ (up to grading shift).
\end{remark}

\subsection{Adjunctions and morphisms between projective modules}\label{sec:projandadjunc}
The aim of this subsection is to fully describe the morphism spaces in $\zig\lprmod$.

Recall that category of ($\bbZ$-)graded objects of $\cC$ is denoted by $g\cC$ (see notation subsection at the end of \cref{sec:intro}).
We begin by providing some adjunctions for functors between $\zig \lprmod$ and $g\cC$ that will aid us in computing the morphism spaces in $\zig \lprmod$:
\begin{proposition} \label{biadjoint pair}
Consider the additive functors 
\[
P_s \otimes - : g\cC \ra \zig\lprmod
\]
and
\[
{}_sP \<-2\> \otimes_{\zig} - \ , {}_sP \otimes_{\zig} - : \zig\lprmod \ra g\cC.
\]
Then
\begin{enumerate}
\item  $P_s \otimes - $ is left adjoint to ${}_sP \otimes_{\zig} - $ \ , and \label{leftadj}
\item $P_s \otimes - $ is right adjoint to ${}_sP \<-2\> \otimes_{\zig} -$ \ .
\end{enumerate}
\end{proposition}
\begin{proof}
The counit and unit morphisms for the corresponding adjunctions above are as follow:
\begin{enumerate}
\item 
	\begin{enumerate}
	\item The counit is induced by the $(\zig,\zig)$-bimodule morphism given by the restriction of the multiplication map $\mu$ of $\zig$ to
	\[
	\mu|_{P_s \otimes {}_sP}: P_s \otimes {}_sP \ra \zig;
	\]
	\item The unit is induced by the grading preserving morphism in $\cC$ defined by
	\[
	\alpha_s : \1 \xra{
		\begin{bsmallmatrix}
		\id \\
		0
		\end{bsmallmatrix}
		}
		(e_s \oplus X_s) = {}_sP \otimes_{\zig} P_s.
	\]
	\end{enumerate}
\item
	\begin{enumerate}
	\item The counit is induced by the grading preserving morphism in $\cC$ defined by 
	(cf.\ \eqref{eqn:frobcounit}):
	\[
	\omega_s : {}_sP \otimes_{\zig} P_s \<-2\> = (e_s \<-2\> \oplus X_s \<-2\>) \xra{
		\begin{bsmallmatrix}
		0 & \id
		\end{bsmallmatrix}		 
		}
		\1.
	\]
	\item The unit is induced by the $(\zig,\zig)$-bimodule morphism 
	\[
	\gamma_s : \zig \ra P_s \otimes {}_sP \<-2\>
	\]
	from \eqref{eqn:frobcomultcomponent}.
	\end{enumerate}	
\end{enumerate}
%The fact that $\gamma_s$ is indeed a morphism of $(\zig,\zig)$-bimodules, and that the counit and unit satisfy the required relations, follows from tedious but direct calculations, which 
We leave it to the reader to verify of the required unit-counit relations of the adjunctions. See \cite[Proposition 2.1.9]{Heng_PhDthesis} for some details in the rank two cases (proofs of which essentially work here too).
\end{proof}

\begin{proposition}\label{graded hom space}
Let $E_1, E_2$ be simple objects in $\cC$.
The morphism spaces between $P_s \otimes E_1 \<k_1\>$ and $P_t \otimes E_2 \<k_2\>$ in $\zig\lprmod$ are given by:
\[
\Hom_{\zig\lprmod} \left( P_s \otimes E_1\<k_1\> , P_t \otimes E_2\<k_2\> \right) \cong
\begin{cases}
\bbC, &\text{for } s=t, \ k_1 = k_2, \ E_1=E_2; \\
\bbC, &\text{for } s=t, \ k_1 - k_2 = 2, \ E_1=E_2; \\
\bbC, &\text{for } (s,t)\in \Gamma_1, \ m_{s,t} < \infty, \ k_1 - k_2 = 1, \\
	& \ E_1 \overset{\oplus}{\subseteq} \nPi[m_{s,t}]_{m_{s,t}-3} \otimes E_2; \\
\bbC^2, &\text{for } (s,t)\in \Gamma_1, \ m_{s,t} = \infty, \ k_1 - k_2 = 1, \\
	& \ E_1=E_2; \\
0, &\text{otherwise}.
\end{cases}
\]
\end{proposition}
\begin{proof}
Using the adjunction between $P_s \otimes -$ and ${}_sP \otimes_{\zig} -$ from \cref{biadjoint pair}, we get that 
\[
\Hom_{\zig\lprmod} \left( P_s \otimes E_1\<k_1\> , P_t \otimes E_2\<k_2\> \right) \cong \Hom_{g\cC} \left(E_1\<k_1\> , {}_sP_t \otimes E_2 \<k_2\>\right),
\]
where ${}_sP_t \cong {}_sP \otimes_{\zig} P_t$ as shown in Proposition \ref{tensor sPt}.
Now compute $\Hom_{g\cC} \left(E_1\<k_1\> , {}_sP_t \otimes E_2 \<k_2\>\right)$ for each possible cases of $s,t \in \Gamma_0$ as follows:
\begin{enumerate}
\item When $s = t$, we have ${}_sP_t = e_s \oplus X_s = \1 \oplus (\1 \<2\>)$, hence
\[
{}_sP_t \otimes E_2\<k_2\> = E_2\<k_2\> \oplus \left(E_2\<k_2 + 2\> \right).
\]
Since all grading shifts of $E_1$ and $E_2$ are simple in $g\cC$, we must have that
\begin{align*}
\Hom_{g\cC} \left(E_1\<k_1\> , {}_sP_t \otimes E_2 \<k_2\>\right) \cong
\begin{cases}
\bbC, &\text{for } E_1=E_2, k_1 - k_2 = 0; \\
\bbC, &\text{for } E_1=E_2, k_1 - k_2 = 2; \text{ and }\\
0, &\text{otherwise}.
\end{cases}
\end{align*}
\item When $(s,t)\in \Gamma_1$, we have ${}_sP_t = (s|t)$. Set $m := m_{s,t}$.
	\begin{enumerate}
	\item If $m<\infty$, then $(s|t) = \Pi(e) = \mPi_{m-3} \<1\>$. By the fusion rule of $\TLJ_\ell$, the object $\mPi_{m-3} \otimes E_2 \in g\cC$ can contain at most one summand of $E_1$ (cf.\ \eqref{eqn: n-3 b fusion rule}), so we have
	\begin{align*}
	\Hom_{g\cC} \left(E_1 \<k_1\>, {}_sP_t \otimes E_2 \<k_2\>\right) \cong
	\begin{cases}
	\bbC, &\text{for } E_1 \overset{\oplus}{\subseteq} \mPi_{m-3} \otimes E_2, \ k_1-k_2=1; \text{ and}\\
	0, &\text{otherwise}.
	\end{cases}
	\end{align*}
	\item If $m_{s,t}=\infty$, then $(s|t) = \Pi(e) = \1 \<1\> \oplus \1 \<1\> \in g\cC$, so 
	\begin{align*}
	\Hom_{g\cC} \left(E_1\<k_1\> , {}_sP_t \otimes E_2 \<k_2\>\right) \cong
	\begin{cases}
	\bbC^2, &\text{for } E_1=E_2, \ k_1 - k_2 = 1; \text{ and}\\
	0, &\text{otherwise}.
	\end{cases}
	\end{align*}
	\end{enumerate} 
\item When $(s,t) \not\in \Gamma_1$, then ${}_sP_t = 0$ and so $\Hom_{g\cC} \left(\mPi_a , {}_sP_t \otimes \nPi_b \<k\>\right) \cong 0$.
\end{enumerate}
This concludes the proof.
\end{proof}

\begin{remark}
Note that the adjunction also says that each morphism in $\zig\lprmod$ is completely determined by its restriction to the summand $e_s$ of $P_s$ as a grading preserving morphism in $g\cC$.
\end{remark}

As a consequence of \cref{graded hom space}, we have the following:
\begin{proposition}\label{prop: indecomposable Krull-Sch}
The set of indecomposable objects (up to isomorphism) in $\zig\lprmod$ is given by
\[
\{P_s \<k\> \otimes E \mid
					s \in \Gamma_0, \
                    k \in \Z, \
					E \in \Irr(\cC).
					\}.
\]
Moreover, $\zig\lprmod$ is Krull--Schmidt, and so is $\Kom^b(\zig\lprmod)$.
\end{proposition}

We end this section with the following result, which shows that $\Kom^b(\zig\lprmod$) enjoys a bigraded refinement of a 2-Calabi--Yau duality.  
\begin{proposition}\label{prop: 2CY variant}
Let $X$ and $Y$ be objects in $\cD:= \Kom^b(\zig\lprmod)$.
Then we have the following functorial isomorphism
\[
\Hom_\cD (X,Y\<j\>[k]) \cong \Hom_\cD (Y,X\<-j-2\>[-k])^\vee.
\]
\end{proposition}
\begin{proof}
Define the same trace map on $\Hom_{\zig\lprmod}(P_s \otimes E, P_s \otimes E\<-2\>)$ as in \cite[Proposition 2.2]{BDL_root} and extend it to a non-degenerate (functorial) bilinear form on $\Hom_\cD (X,Y\<j\>[k]) \otimes \Hom_\cD (Y,X\<-j-2\>[-k])$, which proves the statement.
\end{proof}
\begin{remark}
Our grading shift $\<-\>$ is denoted by $\{-\}$ in \cite[Proposition 2.2]{BDL_root}).
\end{remark}

\subsection{The 2-Calabi--Yau category of a Coxeter system}\label{sec:2CYcat}
We note that the triangulated category $\Kom^b(\zig(\lprmod))$ of the previous section can be used to define a 2-Calabi--Yau (2-CY) category, which we denote by $\cD^{2CY}_\Gamma$. 
The construction is essentially the same as is done classically for symmetric Kac--Moody quivers (see e.g. \cite{BDL_root}): we regard the zigzag algebra $\zig$ as a differential graded algebra (dga) object with trivial differential; namely an algebra object with the trivial differential in the category of complexes.  
Let $\tilde{\cD}_\Gamma$ be the category whose objects are differential graded left modules which are isomorphic, as objects of $g\cC$, to a finite direct sum of simple objects of $g\cC$.  Morphisms in $\tilde{\cD}_\Gamma$ are homotopy classes of maps of differential graded modules. 
The category $\tilde{\cD}_\Gamma$ is a triangulated category and we define $\cD^{2CY}_\Gamma$ to be the smallest full, strict triangulated subcategory of $\tilde{\cD}_\Gamma$ containing the objects $P_s$:
\[
\cD^{2CY}_\Gamma := \langle P_s \mid s\in \Gamma_0 \rangle \subset \tilde{\cD}_\Gamma.
\]
In particular, the argument in \cite[Proposition 2.3]{BDL_root} shows that the category $\cD^{2CY}_\Gamma$ is 2-CY.

\section{Artin--Tits (braid) group action on the homotopy category} \label{sec: action and categorification}
We continue to fix the Coxeter graph $\Gamma$, which we omit from the notation, so that
$\B := \B(\Gamma)$, $\cC := \cC(\Gamma)$, $\zig:= \zig(\Gamma)$, etc.

In this section we define an action of the Artin--Tits group $\B$ on the triangulated module category $\Kom^b(\zig\lprmod)$ over $\cC$.
The construction extends the rank two case given in \cite{Heng_PhDthesis}, and generalises the construction for symmetric Kac--Moody types in \cite{khovanov_seidel_2001} and \cite{HueKho}.  
The generating equivalences are certain two-term complexes of bimodules over $\zig$ ($\cC$-spherical twists; see \cref{defn:fusionsphericaltwist}), so the resulting Artin--Tits group action commutes with the right $\cC$-action on $\Kom^b(\zig\lprmod)$.  The proofs that these define an action of the associated Artin--Tits group $\B$ is essentially the same as the rank two case in \cite{Heng_PhDthesis}.
%With this said, however, we shall introduce the complexes required through the notion of spherical twists, which will simplify the proof of the $n$-braiding relation greatly.

\subsection{Fusion spherical twists}
Recall the categories $\zig\lprmod$ and $\rprmod\zig$ defined in \cref{defn: prmod}.
%With $M \in \zig\lprmod$ and $N \in \rprmod\zig$, 
%$M \otimes N$ is naturally a graded $(\zig, \zig)$-bimodule with left and right module structure induced from $M$ and $N$ respectively.
With $M \in \zig\lprmod$ and $N \in \rprmod\zig$, note that $M \otimes N$ is naturally a $(\zig, \zig)$-bimodule.

\begin{definition} \label{defn: bimodule category}
We define the additive monoidal category $\mathbb{U}:= \mathbb{U}(\Gamma)$ to be the smallest additive, monoidal subcategory of $\zig\bimod \zig$ containing the objects:
\[
\{ M \otimes N \mid 
	M \in \zig\lprmod, \ 
	N \in \rprmod\zig
	\}
			\cup \{ \zig \<k\> :  k \in \Z\}
\]
with tensor product $- \otimes_{\zig} -$ and with $\zig$ as the monoidal unit. 
\end{definition}
%Each $M \otimes N$ in $\mathbb{U}_{\Gamma}$ induces a functor $(M\otimes N)\otimes_{\zig} - : \zig\lprmod$ $\ra \zig\lprmod$ defined by 
%\[
%(M \otimes N) \otimes_{\zig} M' := M \otimes (N \otimes_{\zig} M').
%\]
%Note that the functor $\zig \otimes_{\zig} - : \zig\lprmod$ $\ra \zig\lprmod$ is by definition equal to the identity functor.
%
%Dont think these are needed
%The functors 
%\[
%- \otimes_\I U, - \otimes_\I \I : \text{$\prmod$-}\I \ra \text{$\prmod$-}\I
%\] 
%are defined analogously.
By extending the tensor product to complexes, every complex $U$ in $\Kom^b(\mathbb{U})$ induces an endofunctor on $\Kom^b(\zig\lprmod)$
\[
U \otimes_{\zig} -: \Kom^b(\zig\lprmod) \ra \Kom^b(\zig\lprmod).
\]
The following facts are immediate from the definition.
\begin{itemize}
\item $U \otimes_{\zig} -$ is an exact endofunctor, commuting trivially with the triangulated shift functor $[1]$ and the internal grading shift functor $\<1\>$.
\item The triangulated category $\Kom^b(\zig\lprmod)$ is a right module category over $\cC$ and the functor $U \otimes_{\zig} -$ commutes with the right action of $\cC$ on $\Kom^b(\zig\lprmod)$.
\end{itemize}

Similarly, complexes $X\in \Kom^b(\zig\lprmod$) and $Y \in \Kom^b(\rprmod\zig)$ also induce functors
\[
X\otimes - : \Kom^b(g\cC) \ra \Kom^b(\zig\lprmod), \quad Y \otimes_{\zig} -: \Kom^b(\zig\lprmod) \ra \Kom^b(g\cC).
\]

The following generalised notion of spherical twist is adapted from \cite{anno_logvinenko_2017}.
\begin{definition} \label{defn:fusionsphericaltwist}
Let $X\in \Kom^b(\zig\lprmod)$ and $X^\ell, X^r \in \Kom^b(\rprmod\zig)$, which we consider as functors as described above:
\[
X \otimes - : \Kom^b(g\cC) \leftrightarrows \Kom^b(\zig\lprmod): X^\ell \otimes_{\zig} -, X^r \otimes_{\zig} -.
\]
Suppose that $X^\ell \otimes_{\zig} -$ and $X^r \otimes_{\zig} -$ are left and right adjoints of $X \otimes -$ respectively.
We define:
\begin{enumerate}
\item the \emph{twist} with respect to $X$ as the complex of $(\zig,\zig)$-bimodules:
\[
\sigma_X:= \cone \left( X\otimes X^r \xrightarrow{\varepsilon} \zig[0] \right) \in \Kom^b(\mathbb{U}),
\]
with $\varepsilon$ defining the counit of the adjunction $X \otimes - \dashv X^r \otimes_{\zig} -$; and
\item the \emph{dual twist} of $X$ as the complex of $(\zig,\zig)$-bimodules:
\[
\sigma'_X := \cone \left(  \zig[0] \xrightarrow{\nu} X\otimes X^\ell  \right) \in \Kom^b(\mathbb{U}),
\]
with $\nu$ defining the unit of the adjunction $X^\ell \otimes_{\zig} - \dashv X \otimes -$.
\end{enumerate}
The twist and dual twist with respect to $X$ are both complexes in $\mathbb{U}$, which can therefore be viewed as endofunctors of $\Kom^b(\zig\lprmod)$ via tensoring over $\zig$.
We say that $\sigma_X$ is a \emph{$\cC$-spherical twist} if the twist and dual twist are mutual inverses as endofunctors.
\end{definition}

\subsection{The action and braiding relations}
For each vertex $s \in \Gamma_0$, recall that $P_s \otimes -$ have right and left adjoints ${}_s P \otimes_{\zig} -$ and ${}_s P \<-2\> \otimes_{\zig} -$ respectively (see \cref{biadjoint pair}).

Let ${}_sQ := {}_sP\<-1\>$.
We are particularly interested in the twist and dual twist with respect to the objects $P_s$:
\begin{equation} \label{eqn:sphericaltwist}
\sigma_{P_s} = 0 \ra P_s \otimes {}_sQ\<1\> \xra{\beta_i} \zig \ra 0
\end{equation}
and
\begin{equation} \label{eqn:spehricaltwistinverse}
\sigma'_{P_s} = 0 \ra \zig \xra{\gamma_s} P_i\otimes {}_sQ\<-1\> \ra 0,
\end{equation}
both with $\zig$ sitting in cohomological degree 0. % (The shift in ${}_sQ := {}_sP\<-1\>$ was chosen so that the differentials in $\sigma_{P_s}$ and $\sigma'_{P_s}$ are of the same cohomological and path-length degree.)

\begin{proposition}[Inverse and braiding relations] \label{braid relation}
For each vertex $s \in \Gamma_0$, we have the following isomorphisms in $\Kom^b(\mathbb{U})$:
\[
\sigma_{P_s} \otimes_{\zig} \sigma_{P_s}' \cong \zig[0] \cong \sigma_{P_s}' \otimes_{\zig} \sigma_{P_s},
\]
In particular, $\sigma_{P_s}$ are $\cC$-spherical twists.

Moreover, for each pair of distinct vertices $s,t \in \Gamma_0$ with entry in the Coxeter matrix $m_{s,t}<\infty$, we have the following isomorphism in $\Kom^b(\mathbb{U})$:
\[
\underbrace{\cdots \otimes_{\zig} \sigma_{P_s} \otimes_{\zig} \sigma_{P_t} \otimes_{\zig} \sigma_{P_s}  }_{m_{s,t} \text{ times}}
	\cong 
	\underbrace{\cdots \otimes_{\zig} \sigma_{P_t} \otimes_{\zig} \sigma_{P_s} \otimes_{\zig} \sigma_{P_t} }_{m_{s,t} \text{ times}}.
\]
\end{proposition}
\begin{proof}
The proof is exactly as in the rank two case -- observe that the relations only involve (at most) two distinct vertices in the Coxeter diagram.

Spelling out some of the details, that $\sigma_{P_s}$ and $\sigma_{P_s}'$ are mutual inverses follows from a simple direct computation.
For the braiding relation, suppose $s, t$ are distinct vertices satisfying $m_{s,t} < \infty$.
For ease of notation we set $1 := s$, $2 := t$ and $m := m_{s,t}$.
Since $\cC$-spherical twists are defined via cones of adjunctions, it is sufficient to show that we have the following isomorphism in $\Kom^b(\zig\lprmod)$ up to shifts $[k]$ and $\<\ell\>$
\[
\underbrace{\cdots \otimes_{\zig} \sigma_{P_1} \otimes_{\zig} \sigma_{P_2} \otimes_{\zig} \sigma_{P_1}  }_{(m-1) \text{ times}} \otimes_{\zig} P_2 \cong 
	\begin{cases}
	P_1 \otimes \mPi_{m-2}, &\text{ if } m \text{ is odd}; \\
	P_2 \otimes \mPi_{m-2},&\text{ if } m \text{ is even}.
	\end{cases}
\]
(Noting also that the twists with respect to $P_i$ and with respect to $P_i \otimes \mPi_{m-2}$ are isomorphic regardless of the shift in cohomology and internal gradings.)
This can be done by showing that for
(the subscript $i \pm 1$ is defined according to $i$: $i\pm 1 = 1$ if $i=2$, and $i\pm 1 = 2$ if $i=1$)
\[
X := P_{i\pm 1} \otimes \mPi_a\<k\>[\ell] 
		\xra{ } 
	P_i \otimes \mPi_{a-1} \<k-1\>[\ell-1],
\]
we have
\begin{equation} \label{eqn:alternatebraid}
\sigma_{P_i} \otimes_{\zig}  X
\cong 
\begin{cases}
P_i \otimes \mPi_{a+1}\<k+1\>[\ell+1] 
	\xra{ } 
	P_{i\pm 1} \mPi_a\<k\>[\ell]; &\text{for } a\neq m-2, \\
P_{i\pm 1} \otimes \mPi_a\<k\>[\ell]; &\text{for } a=m-2.
\end{cases}
\end{equation}
For details refer to \cite[\S 2.2]{Heng_PhDthesis}.
\end{proof}
\begin{remark}
The analogous statement for zigzag algebras of fusion quivers over general fusion category $\cC$ is as follows. 
Suppose the arrow from $s$ to $t$ is labelled by $\Pi \in \cC$, so that the opposite arrow $t$ to $s$ is labelled by $\Pi^*$.
If ($\FPdim(\Pi^*) =$)$\FPdim(\Pi) < 2$, which therefore $\FPdim(\Pi) = 2\cos(\pi/m)$ for some integer $m \geq 2$, then $\sigma_{P_s}$ and $\sigma_{P_t}$ satisfy the braiding relation of length $m$.
The calculation of the analogue of \eqref{eqn:alternatebraid} will require techniques similar to that of reflection functors carried out in \cite[Lemma 5.9]{EH_fusionquiver}.
We leave the details to the interested reader.
\end{remark}

\begin{definition}\label{defn:sphericaltwistgroup}
The \emph{$\cC$-spherical twist group}, denoted by $\Br^{\ST}:= \Br^{\ST}(\Gamma)$, is defined as the subgroup of invertible complexes of $(\zig,\zig)$-bimodules generated by the $\cC$-spherical twists $\sigma_{P_s}$ for all $s \in \Gamma_0$ (composition is given by tensoring over $\zig$).
\end{definition}
Viewed as autoequivalences, the complexes in $\Br^{\ST}$ are therefore exact autoequivalences of $\Kom^b(\zig\lprmod)$ that moreover commute with the right $\cC$-action and the internal grading functor $\<1\>$.
As a consequence of \cref{braid relation}, we obtain the following theorem.
\begin{theorem} \label{weak braid action}
The group homomorphism $\B \ra \Br^{\ST}$ sending each standard generator $\sigma_s \mapsto \sigma_{P_s}$ is well-defined.
In particular, we have a (weak) $\B$-action on $\Kom^b(\zig\lprmod)$, given by exact, $\cC$-module autoequivalences that commutes with $\<1\>$.
\end{theorem}

%%%%%%%%%%%%%%%%%%%%%%%%%%%
%
%%%%%%%%%%%%%%%%%%%%%%%%%%%
\section{The Coxeter group action on the Grothendieck group} \label{sec:Grothendieckandlattice}
We continue to fix the Coxeter graph $\Gamma$, which we omit from the notation.
In this section we will also let $\cD:= \Kom^b(\zig(\Gamma)\lprmod)$.

\subsection{The standard heart of linear complexes} \label{sec:linearheart}
The triangulated category $\cD$ has an internal grading shift functor $\<1\>$ coming from the grading on $\zig$; this grading is not to be confused with the cohomological degree shift functor $[1]$.  Let $\cH:= \cH(\Gamma) \subset \cD$ be the abelian category of linear complexes.  Objects of $\cH$ are complexes homotopic to one whose cochain object in cohomological degree $-k$ consists of direct sums of the objects of the form $P_s \otimes E \<k\>$. (Here $s$ and $E$ may vary, but the grading shift $k$ must match with the negative of the cohomological degree.)
This abelian category $\cH$ is the heart of a bounded $t$-structure on $\cD$, and we refer to it in what follows as the \emph{linear heart}.
$\cH$ is a locally finite abelian category, with simple objects 
\begin{equation} \label{eqn:simplesoflinearheart}
\Irr(\cH) = \{ P_s \otimes E\<k\>[k] \mid s \in \Gamma_0, k \in \Z \text{ and } E \in \Irr(\cC) \}.
\end{equation}
Indeed, the fact that these are simple in $\cH$ follows from proposition \ref{graded hom space}.

\subsection{Grothendieck group and the induced action}
The Grothendieck groups of $\cD$ and $\cH$ can be identified, as $\Z$-modules, as follows
\begin{equation} \label{eq:grothendieck}
K_0(\cD) \cong K_0(\cH)  \cong \bigoplus_{s \in \Gamma_0, E \in \Irr(\cC)} \Z[q^{\pm 1}] \cdot [P_s \otimes E],
\end{equation}
with $q$ acting on the class of objects by $q \cdot [X] = [X\<1\>]$.
Note that $\<1\>$ and $[1]$ acts differently on $K_0(\cD)$: $[X[1]] = -[X]$.

Since $\cD$ has a (right) action of $\cC$,  its Grothendieck group $K_0(\cD)$ has a (right) $K_0(\cC)$-module; since $K_0(\cC)$ is a commutative ring, we will equivalently regard $K_0(\cC)$ as acting on the left on $K_0(\cD)$.
As $K_0(\cC)$-modules, we have 
\begin{equation} \label{eq:fusiongrothendieck}
K_0(\cD) \cong K_0(\cH)  \cong \bigoplus_{s \in \Gamma_0} K_0(\cC)[q^{\pm 1}] [P_s].
\end{equation}
In particular, $K_0(\cD)$ is a free $K_0(\cC)[q^{\pm 1}]$-module of rank $ = |\Gamma_0|$.

The spherical twist group $\Br^{\ST}$ acts on $\cD$ via exact, $\cC$-module autoequivalences, and this action commutes with the internal grading shift functor $\<1\>$.  It follows that $\Br^{\ST}$ acts on the free $K_0(\cC)[q^{\pm 1}]$-module $K_0(\cD)$ via $K_0(\cC)[q^{\pm 1}]$-linear automorphisms.
The action on each basis element is described as follows.
\begin{proposition} \label{prop:STactiononGrothendieck}
The generators $\sigma_{P_s} \in \Br^{\ST}$ acts on the $K_0(\cC)[q^{\pm 1}]$-basis element $[P_t] \in K_0(\cD)$ via
\begin{equation} \label{eqn:Buraurep}
\sigma_{P_s}\cdot [P_t] = 
	\begin{cases}
	-q^2[P_t], &\text{ if } s=t; \\
	[P_t] - [\Pi(e)]q \cdot [P_s], &\text{ if } e=(s,t) \in \Gamma_1; \\
	[P_t], &\text{ otherwise},
	\end{cases}
\end{equation}
where $\Pi(e)$ is defined as in \eqref{eqn:coxquiverlabel}.
\end{proposition}
\begin{proof}
This follows from a direct computation.
\end{proof}

\subsection{The fusion-lattice and the Coxeter group action} \label{sec:fusionlatticecox}
For what follows, it will be convenient to consider the quotient $K_0(\cD)$ by setting $q$ evaluated to $-1$ to obtain a finite rank lattice (i.e.\ a finite rank free $\bbZ$-module).
%
%Note that if $\Gamma$ is of symmetric Kac--Moody type, then $\cC = \TLJ^{even}_3 \cong \vec_\bbC$, so there is only a non-trivial fusion category action outside of symmetric Kac--Moody type.  
%As a $\mathbb{Z}$-module, the Grothendieck group $K_0(\cD)$ has infinite rank, and we replace it by a finite-rank $\mathbb{Z}$-module in order to construct a finite dimensional moduli space whose points are stability conditions.  
%In fact, in order to construct a finite-dimensional stability manifold from the triangulated category $\cD$, there are two ways to proceed, both of which will result in equivalent constructions of a stability space.  
%
%The first approach is to replace the triangulated category $\cD$ by the associated 2CY category $\cD^{2CY}$ defined (see \cref{sec:2CYcat}); the Grothendieck group $K_0(\cD^{2CY})$ is a finite-rank free $K_0(\cC)$-module (and also a free finite-rank $\mathbb{Z}$ module), so one may consider the moduli space of stability conditions on $\cD^{2CY}$ directly.
%The second approach, which is the one we follow below, is to continue to use the category $\cD$, but to fix a surjection of $K_0(\cD)$ onto the finite rank free $K_0(\cC)$-module $\Lambda$ defined by:
%\begin{align*}
%K_0(\cD) &\twoheadrightarrow \Lambda \\
%P_s &\mapsto \alpha_s \\
%q^{\pm 1} &\mapsto -1
%\end{align*}
%and extended $K_0(\cC)$-linearly.
%We may then consider the moduli space $\Stab(\cD)$ of stability conditions whose central charge factors through $\Lambda$.  Thus, in what follows we will define central charges on $\cD$ to be $\Z$-module homomorphisms from $\Lambda$ to $\bbC$.
\begin{definition} \label{defn:latticeforstab}
Let $\Lambda := \Lambda(\Gamma)$ be the free $K_0(\cC)$-module (of rank = $|\Gamma_0|$) defined by
\[
\Lambda := \bigoplus_{s \in \Gamma_0} K_0(\cC) \cdot \alpha_s.
\]
We fix the surjective $K_0(\cC)$-module homomorphism (hence also a surjective group homomorphism) $\nu: K_0(\cD) \twoheadrightarrow \Lambda$ uniquely defined by $q \mapsto -1$ and $[P_s] \mapsto \alpha_s$ for each $s \in \Gamma_0$.
\end{definition}
The following is immediate from the fusion ring structure of $K_0(\cC)$:
\begin{lemma} \label{lem:C-latticeaslattice}
$\Lambda$ is a free $\bbZ$-module generated by $[E]\cdot \alpha_s$ for each $E \in \Irr(\cC)$ and $s \in \Gamma_0$.
(In particular, $\Lambda$ is a finite rank lattice.)
\end{lemma}

Observe that the action of $\Br^{\ST}$ on $K_0(\cD)$ factors uniquely through the surjection $\nu$; indeed, the action of $\Br^{\ST}$ on the lattice $\Lambda$ is obtained by setting $q=-1$ in \eqref{eqn:Buraurep} (and identifying $\alpha_s$ with $[P_s]$).
In fact, this action agrees with a faithful action of $\bbW$ on $\Lambda$, which we now describe.

Recall that to each edge $e$ connecting two vertices $s$ and $t$, we associated an object $\Pi(e) \in \cC$ (cf.\ \eqref{eqn:coxquiverlabel}).
Consider the symmetric, $K_0(\cC)$-bilinear form $B_\cC(-,-): \Lambda \times \Lambda \ra K_0(\cC)$ on $\Lambda$ defined on the basis elements (over $K_0(\cC)$) by
\begin{equation} \label{eqn:fusionbilinearform}
B_\cC(\alpha_s, \alpha_t) := 
	\begin{cases}
	2\cdot[\1], &\text{ if } s = t;\\
	-[\Pi(e)], &\text{ if } e=(s,t) \text{ is an edge in } \Gamma; \\
	0, &\text{ otherwise}.
	\end{cases}
\end{equation}
For each standard generator $s \in \Gamma_0 \subseteq \bbW$, define 
\begin{equation} \label{eqn:Wactionfusionlattice}
s \cdot v = v - B_\cC(\alpha_s, v)\cdot \alpha_s, \qquad v \in \Lambda.
\end{equation}
More explicitly, we have
\begin{equation} \label{eqn:fusiongeomrep}
s\cdot \alpha_t = 
	\begin{cases}
	-\alpha_t, &\text{ if } s=t; \\
	\alpha_t + [\Pi(e)]\cdot \alpha_s, &\text{ if } e=(s,t) \in \Gamma_1; \\
	\alpha_t, &\text{ otherwise};
	\end{cases}
\end{equation}
compare this with \eqref{eqn:Buraurep}.

\begin{lemma} \label{lem:fusionTitsrep}
The assignment in \eqref{eqn:Wactionfusionlattice} defines a $K_0(\cC(\Gamma))$-linear representation of $\bbW$ on $\Lambda$.
\end{lemma}
\begin{proof}
This can be done via a direct calculation.
A general proof that works with arbitrary fusion categories can also be found in \cite[Section 6.1]{EH_fusionquiver}.
\end{proof}
\begin{remark}
The lemma above can also be deduced from \cite[Section 4.3]{Dyer09} via embedding of root system, which is related to the unfolding that will be discussed later in \cref{sec:unfolding}.
\end{remark}

We now argue that the representation of $\bbW$ on $\Lambda$ is faithful.
Recall that $K_0(\cC)$, being a fusion ring, has ring homomorphism into $\bbR$ defined by assigning objects $Y \in \cC$ its Frobenius--Perron dimension $\FPdim(Y) \in \bbR_{\geq 0}$ (see \cref{defn:FPdim}):
\[
\FPdim: K_0(\cC) \ra \bbR (\subset \bbC).
\]
This equips both $\bbR$ and $\bbC$ with the structure of a $K_0(\cC)$-module.
As a consequence, we can consider the contragradient action of $\bbW$ on $\Hom_{K_0(\cC)}(\Lambda,\bbC)$ (the dual representation), defined by
\begin{equation} \label{eqn:Wcontragradientaction}
(w \cdot \underline{Z})(v) = \underline{Z}(w^{-1} \cdot v), \qquad \text{ for each } w \in \bbW.
\end{equation}
Note that faithfulness of the contragradient representation above implies faithfulness of the representation on $\Lambda$.
We claim that the representation $\Hom_{K_0(\cC)}(\Lambda,\bbC)$ is isomorphic to the associated contragradient representation (over $\bbC$) on $(\bbR\Lambda)^*_\bbC := \Hom_\bbR(\bbR\Lambda, \bbC)$ associated to the Tits representation $\bbR\Lambda$, as considered in \cref{sec:coximagconehyperplane}.
\begin{proposition} \label{prop:identifywithdualspace}
The map $\Hom_{K_0(\cC)}(\Lambda,\bbC) \ra (\bbR\Lambda)^*_\bbC$ defined by sending $\underline{Z} \in \Hom_{K_0(\cC)}(\Lambda,\bbC)$ to the unique $\underline{Z}' \in (\bbR\Lambda)^*_\bbC$ such that $\underline{Z}'(\alpha_s) := \underline{Z}(\alpha_s)$, is a $\bbW$-equivariant isomorphism.
\end{proposition} 
\begin{proof}
The above linear map is a surjective map between finite dimensional vector spaces of the same dimension, so what remains is to show that the linear map is $\bbW$-equivariant.
Let $\widetilde{\FPdim}: \Lambda \ra \bbR\Lambda$ be the group homomorphism defined by $[Y]\cdot \alpha_s \mapsto \FPdim([Y])\alpha_s$ and extended $\bbZ$-linearly.
Then under this map, the $K_0(\cC)$-bilinear form $B_{\cC}(-,-)$ is identified with the $\bbR$-bilinear form $B(-,-)$.
Since the two $\bbW$-actions are defined using the respective bilinear forms and the $K_0(\cC)$-module structure on $\bbC$ is defined precisely via $\FPdim$, $\bbW$-equivariance follows.
\end{proof}

Since the Tits representation and its the corresponding contragradient representation are faithful, we obtain as an immediate consequence:
\begin{corollary} \label{cor:Wcontragradientfaithful}
The $\bbW$-action on $\Lambda$ and $\Hom_{K_0(\cC)}(\Lambda,\bbC)$ are faithful.
\end{corollary}

A priori, the assignment $\sigma_{P_s} \mapsto s$ need not give a well-defined group homomorphism from $\Br^{\ST}$ to $\bbW$ (we do not yet know if $\Br^{\ST} \cong \B$).
Nonetheless, their common action on $\Lambda$ combined with the fact that the action of $\bbW$ is faithful show that this is well-defined.
We record our finding for later use.
\begin{proposition} \label{prop:STactionagreesW}
The assignment $\sigma_{P_s} \mapsto s$ extends to a well-defined surjective group homomorphism $\Br^{\ST}\longrightarrow \bbW$.  
This group homomorphism, together with the surjection $\nu: K_0(\cD) \twoheadrightarrow \Lambda$, intertwine the $\Br^{\ST}$-action on $K_0(\cD)$ and the $\bbW$-action on $\Lambda$.
Dually, they also intertwine the $\Br^{\ST}$-action on $\Hom_{K_0(\cC)}(K_0(\cD), \bbC)$ and the $\bbW$-action on $\Hom_{K_0(\cC)}(\Lambda,\bbC)$.
\end{proposition}

\section{Fusion equivariant stability conditions and the covering property}\label{sec:stability}
This section contains the main theorem of this paper (\cref{thm:maintheorem}), which states that a connected component of the submanifold of fusion-equivariant stability conditions associated to the triangulated category $\Kom^b(\zig(\Gamma)\lprmod)$ is a covering space of the hyperplane complement $\Upsilon_{\reg}$ associated to the Coxeter graph $\Gamma$. 
\subsection{Fusion-equivariant stability conditions}
We begin by recalling some general results about stability conditions and their fusion-equivariant versions.

\begin{definition}\label{defn:stabwithsupport}
Let $\cD$ be a triangulated category, and fix a surjective group homomorphism $\nu: K_0(\cD) \twoheadrightarrow \Lambda$ from the Grothendieck group of $\cD$ onto a finite rank $\bbZ$-lattice. 
A Bridgeland stability condition $\sigma=(\cP,Z)$ on $\cD$ is a pair consisting of a \emph{slicing} $\cP = \{\cP(\phi)\}_{\phi\in \bbR}$ of $\cD$ and a \emph{central charge} $Z \in \Hom_{\Z}(K_0(\cD), \bbC)$.  The slicing and central charge are required to be compatible in the sense that for all $E\in \cP(\phi)$,
\[
Z(E) = m(E)e^{i\pi\phi}, \text{ with } m(E) \in \R_{>0};
\]
and are required to satisfy the \emph{support property}:
\begin{enumerate}
\item the central charge $Z$ factors through $\nu$; and
\item given a norm $|| \cdot ||$ on $\Lambda \otimes_{\bbZ} \bbR$, there exists $K >0$ such that for all objects $E\in \cP_\phi$, we have $||v(E)|| \leq K|Z(E)|$.
\end{enumerate}
\end{definition}
The objects $E\in \cP(\phi)$ are called \emph{semistable of phase $\phi$.}
We refer the reader to \cite[Appendix B]{BM_localproj} and \cite[Appendix A]{bayerSpaceStabilityConditions2016} for other formulations and features of the support property.
%A simpler way to construct stability conditions is as follows.
%\begin{proposition}[\protect{\cite[Proposition 5.3]{bridgeland_2007}}] \label{prop:stabonab}
%To give a stability condition on $\cD$ is equivalent to giving a heart of a bounded $t$-structure $\cH$ of $\cD$ and a stability function on $\cH$ satisfying the Harder--Narasimhan property and the support property.
%\end{proposition}
%Note that if $\cH \subset \cD$ is a finite-length heart, any stability function will automatically satisfy the HN property and the support property.

Since $\nu: K_0(\cD) \twoheadrightarrow \Lambda$ is surjective, for each stability condition $\sigma=(\cP,Z)$ there exists a unique $\underline{Z} \in \Hom_{\Z}(\Lambda, \bbC)$ such that $\underline{Z} \circ \nu = Z$.  We denote by $\cZ$ the (forgetful) map which takes a stability condition to the unique $\underline{Z} \in \Hom_{\Z}(\Lambda, \bbC)$ defined by the central charge:
\begin{equation} \label{eqn:stablocalhomeo}
\begin{split}
\cZ : \Stab(\cD) &\ra \Hom_{\Z}(\Lambda, \bbC) \\
(\cP, Z = \underline{Z} \circ \nu) &\mapsto \underline{Z}.
\end{split}
\end{equation}
The main theorem of Bridgeland \cite[Theorem 1.2]{bridgeland_2007} (see also \cite{Bayer_shortproof} and \cite[Appendix A]{bayerSpaceStabilityConditions2016}) says that $\cZ$ is a local homeomorphism onto its image, which provides $\Stab(\cD)$ with the structure of a complex manifold.  
The (complex) dimension of the manifold $\Stab(\cD)$ is equal the rank of the lattice $\Lambda$.

When $\cD$ is a triangulated module category for a fusion category $\cC$, the notion of a stability condition can be refined to include compatability with the $\cC$ action.   
Fusion-equivariant stability conditions, introduced in \cite{Heng_PhDthesis} and \cite{DHL_fusionstab}, are defined as follows.
\begin{definition}\label{defn: C equivariant stab}
Let $\cD$ be a triangulated (right) module category over $\cC$ and let $\sigma = (\cP, Z)$ be a stability condition on $\cD$.
A stability condition $(\cP,Z) \in \Stab(\cD)$ is said to be \emph{fusion-equivariant} over $\cC$, or \emph{$\cC$-equivariant}, if it satisfies the following:
\begin{enumerate}
\item $\cP(\phi) \otimes Y \subseteq \cP(\phi)$ for all $Y \in \cC$ and all $\phi \in \R$; and \label{item:semistabletosemistable}
\item $Z \in \Hom_{K_0(\cC)}(K_0(\cD), \bbC) \subseteq \Hom_{\bbZ}(K_0(\cD), \bbC)$, where $\bbC$ is viewed as a $K_0(\cC)$-module via the Frobenius--Perron dimension map $\FPdim: K_0(\cC) \ra \R \subset \bbC$. \label{item:fusionequivcentralcharge}
\end{enumerate}
\end{definition}
Note that if $\cC = \vec$ is the category of finite dimensional vector spaces, then all stability conditions are $\cC$-equivariant ($\cP(\phi)$ is closed under taking direct sums and $K_0(\cC) = \bbZ$).
Note also that in general the standard heart $\cP(0,1]$ of any $\cC$-equivariant stability condition is closed under the $\cC$-action, i.e.\ $\cP(0,1]$ is an abelian $\cC$-module category.

The following result is (a weaker version of) Theorem 3.15 in \cite{DHL_fusionstab}:
\begin{proposition}[\protect{\cite[Theorem 3.15]{DHL_fusionstab}}] \label{prop:Cequivariantconditionfromheart}
Let $\sigma = (\cP, Z)$ be a stability condition of a triangulated $\cC$-module category $\cD$.
Suppose its standard heart $\cP(0,1]$ is closed under the $\cC$-action.
Then $\sigma$ is $\cC$-equivariant if and only its stability function on $\cP(0,1]$ is a $K_0(\cC)$-module homomorphism.
\end{proposition}

We use $\Stab_{\cC}(\cD) \subseteq \Stab(\cD)$ to denote the subset of all $\cC$-equivariant stability conditions (where $\Stab_{\cC}(\cD) = \Stab(\cD)$ if $\cC=\vec$).
%Indeed, the monoidal unit $\1$ acts as the identity functor and all objects in $\vec_{\bbC}$ are direct sums of $\1$.
%In turn, (1) translates to $\cP(\phi)$ being additive subcategories (which they are) and the $K_0(\cC)\cong \bbZ$-module structures on $K_0(\cD)$ and $\bbC$ are the usual $\bbZ$-module structures, so $\Hom_{K_0(\cC)}(\Lambda, \bbC) = \Hom_{\bbZ}(\Lambda, \bbC)$.
%In particular, $\Stab_\cC(\cD) = \Stab(\cD)$ when $\Gamma$ is of symmetric Kac--Moody type.
%Moreover, by \eqref{eq:fusiongrothendieck}, we have that $\Lambda$ is also free over $K_0(\cC)$, so that
%\[
%\Lambda = \bigoplus_{s \in \Gamma_0} K_0(\cC)\{P_s\}.
%\]
%As such, condition (\ref{item:fusionequivcentralcharge}) in \cref{defn: C equivariant stab} can be equivalently stated as
%\[
%Z \in \Hom_{K_0(\cC)}(\Lambda, \bbC) \subseteq \Hom_{\bbZ}(\Lambda, \bbC) \iff 
%\]
%Denote by $\Hom_\bbZ^\cC(\Lambda, \bbC)$ the subspace of linear maps coming from $K_0(\cC)$-equivariant central charges:
%\begin{equation}\label{eqn:fusioncentralcharge}
%\Hom_\bbZ^\cC(\Lambda, \bbC) := \{ \underline{Z} \in \Hom_\bbZ(\Lambda, \bbC) \mid \underline{Z}\circ \nu \in \Hom_{K_0(\cC)}(K_0(\cD),\bbC) \}.
%\end{equation}
We have the following theorem from \cite{DHL_fusionstab}.
\begin{theorem}[\protect{\cite[Theorem A]{DHL_fusionstab}}] \label{thm:closedsubmfld}
The local homeomorphism $\cZ: \Stab(\cD) \ra \Hom_\bbZ(\Lambda, \bbC)$ restricts to a local homeomorphism from $\Stab_{\cC}(\cD)$ into a $\bbC$-linear subspace of $\Hom_\bbZ(\Lambda, \bbC)$, so that $\Stab_{\cC}(\cD)$ is a complex submanifold of $\Stab(\cD)$.
Moreover, $\Stab_{\cC}(\cD)$ is closed in $\Stab(\cD)$.
\end{theorem}
%Since $\Hom_{K_0(\cC)}(K_0(\cD),\bbC)$ is a $\bbC$-linear subspace of $\Hom_\bbZ(K_0(\cD),\bbC)$, the subset of all $\underline{Z} \in \Hom_\bbZ(\Lambda, \bbC)$ satisfying $\underline{Z} \circ \nu \in \Hom_{K_0(\cC)}(K_0(\cD),\bbC)$ is also a $\bbC$-linear subspace of $\Hom_\bbZ(\Lambda, \bbC)$.

\subsection{Statement of the main theorem}
For the rest of section \cref{sec:stability}, we will fix a Coxeter graph $\Gamma$.  
As such we once again drop the descriptor ``$\Gamma$'' to simplify our notation:
$\cC:= \cC(\Gamma)$, $\bbW := \bbW(\Gamma)$, $\B := \B(\Gamma)$ and $\Br^{\ST}:= \Br^{\ST}(\Gamma)$ etc. 

We denote $\cD:= \Kom^b(\zig(\Gamma)\lprmod)$, and fix the surjection $\nu: K_0(\cD) \ra \Lambda$ onto the finite rank lattice $\Lambda$, as in \cref{defn:latticeforstab}.
%Recall that $\Hom_{\bbZ}^\cC(\Lambda, \bbC)$ denotes the subspace of group homomorphisms $\underline{Z} \in \Hom_\bbZ(\Lambda, \bbC)$ such that $\underline{Z}\circ \nu \in \Hom_\bbZ(K_0(\cD),\bbC)$ is  a $K_0(\cC)$-module homomorphism (see \eqref{eqn:fusioncentralcharge}).
%The following is immediate from the definition of $\Lambda$ and $\nu$.
%\begin{lemma}
%$\Hom_\bbZ^\cC(\Lambda, \bbC) = \Hom_{K_0(\cC)}(\Lambda, \bbC)$.
%\end{lemma}
%The lemma above together with \cref{prop:identifywithdualspace} gives us the following reformulation of \cref{thm:closedsubmfld} in our setting:
%\begin{corollary}
%The 
%Let $\cZ_\cC$ denote the forgetful map $\cZ: \Stab(\cD) \ra \Hom_\bbZ(\Lambda,\bbC)$ restricted to the subspace $\Stab_\cC(\cD)$
%\begin{equation} \label{eqn:localhomeofusionstab}
%\cZ_\cC: \Stab_\cC(\cD) \ra \Hom_{K_0(\cC)}(\Lambda, \bbC).
%\end{equation}
%Then $\cZ_\cC$ is a local homeomorphism.
%%composed with the linear isomorphism in \cref{prop:identifywithdualspace}: 
%%\underset{\cong}{\xra{\text{\scriptsize Prop. } \ref{prop:identifywithdualspace}}} (\bbR\Lambda)^*_\bbC.
%\end{corollary}
The main theorem of this paper will strengthen \cref{thm:closedsubmfld} above in the case of a distinguished connected component of $\Stab_\cC(\cD)$, which we now describe.

Firstly, recall that we have an $\bbW$-equivariant isomorphism $(\bbR\Lambda)^*_\bbC \cong \Hom_{K_0(\cC)}(\Lambda,\bbC)$ from \cref{prop:identifywithdualspace}.
As such, the hyperplane complement $\Upsilon_{\reg}$ (see \cref{defn:hyperplanecomplement}) can be viewed as a subset of $\Hom_{K_0(\cC)}(\Lambda,\bbC)$ under this isomorphism.
From now on, we will simply view $\Upsilon_{\reg} \subset \Hom_{K_0(\cC)}(\Lambda,\bbC)$.

Let $\cH$ denote the linear heart of $\cD$ (see \cref{sec:linearheart}) and recall that $\cH$ is a finite-length abelian category with finitely many ($=|\Gamma_0 \times \Irr(\cC)|$) simple objects up to linear shift $\<1\>[1]$  (see \eqref{eqn:simplesoflinearheart}).
By the definition of $\nu$, any group homomorphism $Z: K_0(\cH) \ra \bbC$ that factors through $\nu$ is completely determined by $Z(P_s\otimes E)$ for all $s \in \Gamma_0, E \in \Irr(\cC)$. 
Moreover, $Z(P_s\otimes E)$ are all in $\bbH \cup \bbR_{<0} \subset \bbC$ (the strict upper half plane union the negative reals) if and only if $Z$ induces a stability condition with standard heart $\cP(0,1] = \cH$; the Harder--Narasimhan property and support property are immediate.
It follows that the subset $\Stab(\cH)\subset \Stab(\cD)$ of all stability conditions with standard heart $\cP(0,1] = \cH$ is homeomorphic to the space $(\bbH \cup \bbR_{<0})^{|\Gamma_0 \times \Irr(\cC)|}$; hence it is non-empty and connected.

Let $\Stab_\cC(\cH)$ denote the subset of stability conditions in $\Stab(\cH)$ which are in addition $\cC$-equivariant.
%Notice that two simple objects of $\cH$ have the same image under the the surjection $\nu: K_0(\cD) \twoheadrightarrow \Lambda$ if and only if they differ by a diagonal shift $\<k\>[k]$.  
%Moreover, the linear heart $\cH$ is closed under the action of $\cC$ -- tensoring with objects in $\cC$ does not change the internal and cohomological gradings of each cochain object.
Notice that the linear heart $\cH$ is closed under the action of $\cC$ -- tensoring with objects in $\cC$ does not change the internal and cohomological gradings of each cochain object.
By \cref{prop:Cequivariantconditionfromheart}, $\Stab_\cC(\cH)$ is the subset of all stability conditions in $\Stab(\cH)$ whose stability functions are moreover $K_0(\cC)$-module homomorphisms.
These stability functions are therefore completely determined by $Z(P_s) = Z(P_s \otimes \1)$, since being $K_0(\cC)$-module homomorphisms dictates that $Z(P_s \otimes E) = \FPdim(E)Z(P_s)$ for all $E \in \Irr(\cC)$.
It follows that $\cZ$ identifies $\Stab_\cC(\cH)$ with the complexified chamber $C \subset \Upsilon_{\reg}$ (see \eqref{eqn:complexiefiedchamber}).
We record this observation that we will use later.
\begin{lemma}\label{lem:linearstabfundchamber}
$\cZ$ maps $\Stab_\cC(\cH)$ homeomorphically onto the complexified chamber $C \subset \Upsilon_{\reg}$.
\end{lemma}
In particular, $\Stab_\cC(\cH) \subset \Stab_\cC(\cD)$ is also non-empty and connected.
\begin{definition} \label{defn:distinguishedStab}
We define $\Stab_\cC^\dagger(\cD)$ to be the connected component of $\Stab_\cC(\cD)$ that contains $\Stab_\cC(\cH)$; similarly $\Stab^\dagger(\cD)$ denotes the connected component of $\Stab(\cD)$ that contains $\Stab(\cH)$.
\end{definition}
Note that by our arguments before, we have the containments 
\[
\begin{tikzcd}[column sep = small, row sep = small]
\Stab_\cC(\cH) \ar[r,phantom, "\subseteq"] \ar[d, sloped, phantom, "\subseteq"] 
	& \Stab(\cH) \ar[d, sloped, phantom, "\subseteq"]  \\
\Stab_\cC^\dagger(\cD) \ar[r,phantom, "\subseteq"] 
	& \Stab^\dagger(\cD).
\end{tikzcd}
\]

We are now ready to state the main theorem.  
Let $\cZ^\dagger$ denote the restriction of the local homeomorphism $\cZ: \Stab(\cD) \ra \Hom_\bbZ(\Lambda, \bbC)$ from \eqref{eqn:stablocalhomeo} to the connected component $\Stab^\dagger(\cD)$, and let $\cZ^\dagger_{\cC}$ denote the further restriction to the connected submanifold $\Stab^\dagger_\cC(\cD)$.  
Let $\widehat{\Br}^{\ST}\subset \mbox{Aut}(\cD)$ denote the subgroup of autoequivalences generated by $\Br^{\ST}$ and the cohomological shift $[2]$.
Note that we have a natural surjection from $\pi_1(\Upsilon_{\reg}/\bbW)$ to $\widehat{\Br}^{\ST}$ induced by the group homomorphism $\B \ra \Br^{\ST}$ in theorem \ref{weak braid action}, and when $\Gamma$ is not finite-type, by sending the generator of the $\bbZ$ factor to $[2]$ (cf.\ \eqref{eqn:fundgrpArtin}).
\begin{theorem}\label{thm:maintheorem}
The local homeomorphism 
\[
\cZ^\dagger_{\cC}: \Stab_\cC^{\dagger}(\cD) \ra \Hom_\bbZ(\Lambda, \bbC)
\] 
is a covering map onto $\Upsilon_{\reg}$.
Moreover, the action of $\pi_1(\Upsilon_{\reg}/\bbW)$ on $\Stab^\dagger_\cC(\cD)$ by deck transformations of the composite covering map
\[
\underline{\pi}: \Stab^\dagger_\cC(\cD) \xra{\cZ^\dagger_\cC} \Upsilon_{\reg} \twoheadrightarrow \Upsilon_{\reg}/\bbW
\]
factors through the action of $\widehat{\Br}^{\ST}$ on $\Stab^\dagger_\cC(\cD)$ by autoequivalences.
\end{theorem}

In other words, the two a priori distinct actions of $\pi_1(\Upsilon_{\reg}/\bbW)$ on $\Stab^\dagger_\cC(\cD)$ -- one by deck transformations, the other by autoequivalences (via $\widehat{\Br}^{\ST}$) -- agree.
\begin{remark}
When $\Gamma$ is finite-type, the action of the shift $[2]$ on $\Stab^\dagger_\cC(\cD)$ agrees with the action of the ``full twist'', that is, the positive element of $\Br^{\ST}$ which generates the center.
\end{remark}

If $\Gamma$ is symmetric Kac--Moody type, we have that $\cC \cong \vec_\bbC$ by construction.
In particular, all stability conditions are automatically $\cC$-equivariant, where we have $\Stab^\dagger_\cC(\cD) = \Stab^\dagger(\cD)$ and $\cZ^\dagger_\cC = \cZ^\dagger$.
As such, \cref{thm:maintheorem} recovers the following, which is the main theorem of \cite{ikeda2014stability}.
\begin{corollary}\label{cor:symKMthm}
Suppose $\Gamma$ is symmetric Kac--Moody type.
Then $\cZ^{\dagger}: \Stab^\dagger(\cD) \ra \Hom_{\bbZ}(\Lambda, \bbC)$ is a covering map onto $\Upsilon_{\reg}$.
Moreover, the action of $\pi_1(\Upsilon_{\reg}/\bbW)$ on $\Stab^\dagger_\cC(\cD)$ by deck transformations of the composite covering map $\underline{\pi}$ factors through the action of $\widehat{\Br}^{\ST}$ on $\Stab^\dagger_\cC(\cD)$ by autoequivalences.
\end{corollary}

A description of the entire space $\Stab^{\dagger}(\cD)$ when $\Gamma$ is not of symmetric Kac--Moody type will be given in the next section.
More precisely, we will show that $\cZ^\dagger: \Stab^{\dagger}(\cD) \ra \Hom_\bbZ(\Lambda, \bbC)$ is also a covering onto its image, which is instead homeomorphic to a hyperplane complement associated to an unfolded Coxeter system (see \cref{thm:embeddings} and \cref{cor:unfolding}).

\subsection{Conjectures and group theoretic consequences} \label{sec:conjandconsequence}
Before giving the proof of \cref{thm:maintheorem}, we note a few reasons why this theorem is relevant for the study of Artin--Tits groups.

\begin{conjecture}\label{conj:faithfulness}
The action of the Artin--Tits group $\B$ on the triangulated category $\cD$ is faithful.
\end{conjecture}
By \cref{thm:maintheorem}, the conjecture above is equivalent to the statement that the connected component $\Stab^\dagger_\cC(\cD)$ is the universal cover of $\Upsilon_{\reg}$, i.e.\ $\Stab^\dagger_\cC(\cD)$ is simply-connected.

Proving the above conjecture would solve the word problem in the Artin--Tits group $\B$, via the following observation.
\begin{proposition}
The $\cC$-spherical twist group $\Br^{\ST}$ has a solvable word problem.
\end{proposition}
\begin{proof}
We will show later in \cref{prop:faithfulactiononStab} that an autoequivalence $\Psi\in \Br^{\ST}$ is isomorphic to the identity functor if and only if $\Psi(P_s) \cong P_s$ for all $s$.  But checking whether or not $\Psi(P_s) \cong P_s$ holds in the triangulated category $\cD$ is a question of (finite) linear algebra, answered by performing Gaussian elimination on the matrix for the differential in the complex 
$\Psi(P_s)$.  Thus, given $\Psi$, we may algorithmically check whether or not $\Psi\cong \id$ by checking whether or not $\Psi(P_s) \cong P_s$ for each $s$.
\end{proof}
Conjecture \ref{conj:faithfulness} also has group theoretical consequences in relation to LCM-homomorphisms; see \cref{sec:LCM} for more details.

In fact, by analogy with other expectations in the theory of Bridgeland stability conditions, we conjecture the following $K(\pi,1)$ statement.
\begin{conjecture}\label{conj:contractibility}
The connected component $\Stab^\dagger_\cC(\cD)$ is contractible, hence $\Upsilon_{\reg}/\bbW$ and $\Upsilon_{\reg}$ are $K(\pi,1)$ spaces.
\end{conjecture}
Note that conjecture \ref{conj:contractibility} implies conjecture \ref{conj:faithfulness}.

\begin{theorem} 
Conjectures \ref{conj:faithfulness} and \ref{conj:contractibility} hold when the Coxeter group $\bbW$ is finite.
\end{theorem}

The proof of the above theorem follows almost verbatim from the existing arguments in simply-laced type $ADE$.  Since there are no more new ideas needed in the argument, we we omit the details in the interest of space, though we say a bit about the proof here.

A proof of conjecture \ref{conj:faithfulness} in type $ADE$ was first given by Brav--Thomas in \cite{brav_thomas_2010} (see also \cite{LQ_ADE}). There are two essential points about $ADE$ braid groups used in the Brav--Thomas argument.  The first is that every braid can be written as a positive braid (that is, a word in the positive Artin generators $\sigma_s$) times the inverse of a positive braid.  The second essential point is that positive braids have greedy/Garside normal forms which uniquely determine them.  Both of these facts hold not just in type $ADE$, but for all finite Coxeter types (and the second point holds in any Coxeter type).  As a result, conjecture \ref{conj:faithfulness} can be proven by showing the following.

\begin{proposition} Let $\beta$ be a positive element in a finite type Artin-Tits group.  Then the greedy/Garside normal form of $\beta$ can be determined from the action $\beta$ on $\cD$.  
\end{proposition}
The proof of this proposition, which relates the Garside length to the spread in a $t$-structure given in \cite{brav_thomas_2010} carries over directly to our setting.
An analogous theorem in the setting of the Hecke category is proven by Jensen in \cite{jensen_2016}, again for all finite types.

In finite type it is known that the universal cover of $\Upsilon_{\reg} = \Omega_{\reg}$ is contractible \cite{Deligne_72}; see also \cite{McCammond2017TheMG, Paris_Kpi1survey} and references therein.  Thus, when $\bbW$ is finite, the fact that $\Stab^\dagger_\cC(\cD)$ is the universal cover of $\Upsilon_{\reg}$ proves conjecture \ref{conj:contractibility}.
It is also interesting to prove directly that $\Stab^\dagger_\cC(\cD)$ is contractible without relying on known results about the $K(\pi,1)$ conjecture; see \cite{QW_18,AW_22,PSZ_18} for contractibility theorems in related settings.
In particular, such a proof for finite $\bbW$ is given in \cite{QZ_fusion-stable}.

\subsection{Outline of the proof of the main theorem}
The rest of this section is devoted to the proof of \cref{thm:maintheorem}.
After describing the relevant group actions on $\Stab^\dagger_\cC(\cD)$ in \cref{sec:groupactionsonstab}, 
we identify $\Upsilon_{\reg}$ with the image of $\cZ^\dagger_\cC$ in \cref{sec:imageofZ}.
Then, in \cref{sec:coveringstructures}, we show that $\Br^{\ST}$ (and trivially, $[2\bbZ]$) acts on $\Stab_\cC^\dagger(\cD)$ freely and properly discontinuously; finally, we show that the $\cZ^\dagger_\cC$ induces a homeomorphism $\Stab_\cC^\dagger(\cD)/\widehat{\Br}^{\ST}\cong \Upsilon_{\reg}/\bbW$.

This implies that $\widehat{\Br}^{\ST}$ is isomorphic to the group of deck transformations of the covering map $\Stab_\cC^\dagger(\cD) \twoheadrightarrow \Stab_\cC^\dagger(\cD)/\widehat{\Br}^{\ST} \cong \Upsilon_{\reg}/\bbW$.  The fact that $\Stab_\cC^\dagger(\cD)$ covers the intermediate space $\Upsilon_{\reg}$ follows from looking at the smaller quotient $\Stab_\cC^\dagger(\cD)/\widehat{\PBr}^{\ST}$ by the subgroup generated by the ``pure'' spherical twist subgroup and $[2]$.

\begin{remark}
An alternative method of establishing the covering property of $\cZ$ would be to follow Bridgeland's method in \cite{bridgelandStabilityConditionsK32008} and also Bayer's work \cite{Bayer_shortproof}.  Our approach is less direct, but has the virtue that it includes a proof that $\Br^{\ST}$ acts on $\Stab_\cC^\dagger(\cD)$ freely and properly discontinuously.  
\end{remark}

\subsection{Group actions on stability conditions} \label{sec:groupactionsonstab}
Several group actions on $\Stab(\cD)$ are used in what follows:
\begin{itemize}
\item The additive group $\bbC$ acts freely on $\Stab(\cD)$ via $z \cdot (\cP, Z) := (\cP', Z')$ where
\[
\cP'(\phi) := \cP(\phi+\re(z)), \qquad Z' := e^{-z(i\pi)} \cdot Z.
\]
It is immediate that this $\bbC$-action preserves each connected component of the the set of fusion-equivariant stability conditions, so $\bbC$ acts on $\Stab_\cC^\dagger(\cD)$.  The action of $a \in \bbR \subset \bbC$ in the above shifts the slicings $\cP'(\phi) = \cP(\phi+a)$ and rotates the central charge $Z' = e^{-i\pi a} \cdot Z$.  In other words, the action of $\bbR$ on stability conditions lifts the rotational $S^1$ action on central charges.

\item Each $\Psi \in \Br^{\ST}$ acts on $\Stab(\cD)$ via $\Psi \cdot (\cP, Z) := (\cP', Z')$ where
\[
\cP'(\phi) := \Psi(\cP(\phi)), \qquad Z' := \Psi \cdot Z = Z \circ \Psi^{-1}
\]
By proposition \ref{prop:STactionagreesW}, $Z'$ as defined above still factors through $\nu$ (and agrees with the corresponding action of the image of $\Psi$ in $\bbW$).
%This is the only subtle point; the second condition of the support property and the fact that $\cP'$ defines a slicing on $\cD$ are immediate.
Moreover,  $\Br^{\ST}$ acts on the submanifold $\Stab_\cC(\cD)$, as $\Br^{\ST}$ acts via $\cC$-module autoequivalences.
\end{itemize}

%Recall from \cref{prop:STsurjectsonW} that the group homomorphism $\Br^{\ST} \ra \bbW$ defined by $\sigma_{P_s} \mapsto s$ is well-defined and surjective.
The following lemma shows that the forgetful map $\cZ_\cC$ intertwines the $\Br^{\ST}$-action on $\Stab_\cC(\cD)$ and the $\bbW$-action on $(\bbR\Lambda)^*_\bbC$.
\begin{lemma} \label{lem:STacitononStabreducestoW}
For each $s$ and each $\tau \in \Stab_\cC(\cD)$, we have
\[
\cZ_\cC(\sigma_{P_s} \cdot \tau) = s\cdot \cZ_\cC(\tau).
\]
\end{lemma}
\begin{proof}
Let $\tau = (\cP, Z = \underline{Z}\circ \nu) \in \Stab_\cC(\cD)$ with $\underline{Z} \in \Hom_{K_0(\cC)}(\Lambda,\bbC)$.
Under the identification in \cref{prop:identifywithdualspace}, it is sufficient to show that
\[
\sigma_{P_s} \cdot \underline{Z} = s \cdot \underline{Z}.
\]
This is exactly the final statement of \cref{prop:STactionagreesW}.
\end{proof}

Recall that $\cH$ is the heart of linear complexes of $\cD$ and $\Stab_\cC(\cH)$ is the set of $\cC$-equivariant stability conditions whose standard heart is $\cH$.
The following lemma follows exactly as in \cite{bridgeland_2009, ikeda2014stability}, and we omit the proof.
\begin{lemma}\label{lem:STactiononboundary}
Let $\tau = (\cP, Z) \in \Stab_\cC(\cH)$ be a boundary point (in $\Stab^\dagger_\cC(\cD)$).
If $Z([P_s]) \in \bbR_{<0}$ for exactly one $s \in \Gamma_0$, then $\sigma^{-1}_{P_s}(\tau)$ is also a boundary point of $\Stab_\cC(\cH)$.
Similarly, if $Z([P_s]) \in \bbR_{>0}$ for exactly one $s \in \Gamma_0$, then $\sigma_{P_s}(\tau)$ is also a boundary point of $\Stab_\cC(\cH)$ .
\end{lemma}

It follows from this that if $\tau\in \Stab_\cC(\cH)$, then $\tau$ and $\sigma_{P_s}(\tau)$ are in the same connected component 
$\Stab_\cC^\dagger(\cD)$.  In particular, we have the following.

\begin{corollary}
The action of $\Br^{\ST}$ on $\Stab_\cC(\cD)$ preserves the connected component $\Stab_\cC^\dagger(\cD)$.  
%Moreover, any $\tau\in\Stab_\cC^\dagger(\cD)$ is in the $\Br^{\ST}$ orbit of a point in the closure of $\Stab^\dagger_\cC(\cH)$.
\end{corollary}

\subsection{Identifying the image} \label{sec:imageofZ}
We identify the image of $\cZ^\dagger_\cC: \Stab^\dagger_\cC(\cD) \ra \Hom_{K_0(\cC)}(\Lambda,\bbC)$ in this subsection.
We remind the reader that we regard the hyperplane complement $\Upsilon_{\reg}$ as a subset of $\Hom_{K_0(\cC)}(\Lambda,\bbC)$ under the identification $(\bbR\Lambda)^*_\bbC \cong \Hom_{K_0(\cC)}(\Lambda,\bbC)$ from \cref{prop:identifywithdualspace}.
\begin{lemma}
The image of $\cZ^\dagger_\cC$ contains $\Upsilon_{\reg}$.
\end{lemma}
\begin{proof}
Recall that the map $\cZ^\dagger_\cC$ intertwines the $\Br^{\ST}$ action on $\Stab_\cC^\dagger(\cD)$ and the $\bbW$ action on $\Hom_{K_0(\cC)}(\Lambda,\bbC)$.  
Moreover, the action of $\bbR\subset \bbC$ on $\Stab_\cC^\dagger(\cD)$ reduces to the action of $S^1$ on $\Hom_{K_0(\cC)}(\Lambda,\bbC)$ via $\cZ^\dagger_\cC$. 
By \cref{prop:WandS1actiongenerates}, we only need to show that the image of $\cZ^\dagger_\cC$ contains the normalised subset $C^N \subseteq C$ of the complexified chamber (see \cref{defn:normalisedspace}).  
But we already showed in \cref{lem:linearstabfundchamber} that the set $C$ is the (homeomorphic) image of $\Stab_\cC(\cH) \subseteq \Stab^\dagger_\cC(\cD)$ under $\cZ$, hence the proof is complete.
%But any central charge where all the simple roots are in the half-closed upper half plane
%is clearly in the image of  $\cZ^\dagger_\cC$: simply regard such a central charge as a central charge of the linear heart $\cH\subset \cD$, thus defining a point in $\Stab_\cC(\cH)$.  It follows that the subset $C^N\subset C$ of normalised central charges is in the image.
\end{proof}

The proof of the other containment $\img(\cZ^\dagger_\cC) \subseteq  \Upsilon_{\reg}$ will require some more work.
Recall again from \cref{lem:linearstabfundchamber} that $\cZ^\dagger_\cC$ restricts to a homeomorphism $\Stab_\cC(\cH) \cong C \subseteq \Upsilon_{\reg}$.
We denote the preimage of the normalised complexified chamber $C^N \subseteq C$ under this identification by $\Stab_\cC(\cH)^N$:
\[
\Stab_\cC(\cH)^N := \{ \tau \in \Stab_\cC(\cH) \mid \cZ^\dagger_\cC(\tau) \in C^N \}.
\]
We remind the reader of our convention that $\Upsilon_{\reg}^N := \Upsilon_{\reg}$ and $C^N = C$ when $\bbW$ is finite, hence also $\Stab_\cC(\cH)^N = \Stab_\cC(\cH)$.
%(This is not the same as $(\cZ^\dagger_\cC)^{-1}(C^N)$.)
The proof of the following lemma is exactly as in \cite[Proposition 4.13]{ikeda2014stability} and \cite[Lemma 3.6]{bridgeland_2009}, and we omit it.
The main idea is that the walls of the chambers in $(\cZ^\dagger_{\cC})^{-1}(\Upsilon^N_{\reg})$ -- locally controlled by the hyperplanes $\cup_{\alpha \in \Phi^+} H_\alpha$ -- are locally finite, which allow us to apply a finite iteration of Lemma \ref{lem:STactiononboundary}.
\begin{lemma} \label{lem:orbitnormalisedheart}
Let $\tau \in (\cZ^\dagger_{\cC})^{-1}(\Upsilon^N_{\reg}) \subseteq \Stab^\dagger_\cC(\cD)$. 
There exists $\Psi \in \Br^{\ST}$ such that $\Psi \cdot \tau \in \Stab_\cC(\cH)^N$.
\end{lemma}
%\begin{proof}
%Choose a part
%Strategy: Focus on the normalised version as we can always rescale: with $S^1$ locally and with $\bbR$ on $\Stab$. 
%Then we use the local picture. The real-codimension 1 walls, given by complex hyperplane from roots + positive reals for one and + negative reals for the other (see section 2.4 of Ikeda), are locally finite (imaginary part is locally finite). 
%Any path connecting a point in $\Stab_\cC(\cH)^N$ to a point in $(\cZ^\dagger_{\cC})^{-1}(\Upsilon^N_{\reg})$ can then be deformed to pass through finitely many walls, and moreover only pass through codimension one walls (avoiding intersections).
%Using \cref{lem:STactiononboundary}, every time we pass through a wall we apply the corresponding spherical twist. Since walls are sent to walls, this process eventually terminates. 
%\end{proof}
Since we can use the $\bbR\subset\bbC$ action to rescale, it follows directly from the above that all stability conditions in $(\cZ^\dagger_{\cC})^{-1}(\Upsilon_{\reg})$ are in the $\bbC\times \Br^{\ST}$-orbit of stability conditions in $\Stab_\cC(\cH)$.
%\begin{remark}
%We note here that given $\tau \in (\cZ^\dagger_{\cC})^{-1}(\Upsilon^N_{\reg})$, there need not be a unique $r \in \bbR$ \emph{and} $\Psi \in \Br^{\ST}$ such that $r\cdot \Psi \cdot \tau \in \Stab_\cC(\cH)^N$.
%In particular, for some $\tau \in (\cZ^\dagger_{\cC})^{-1}(\Upsilon^N_{\reg})$, there exists $\Psi \in \Br^{\ST}$ which acts on $\tau$ by a rotation, namely $\Psi \cdot \tau = r \cdot \tau$ for some $r \in \bbR$.
%\end{remark}

We now show that $(\cZ^\dagger_{\cC})^{-1}(\Upsilon_{\reg})$ is actually the whole of $\Stab^\dagger_\cC(\cD)$; this implies that the image of $\cZ_\cC^\dagger$ is indeed $\Upsilon_{\reg}$.  Essentially, we have to show that the central charge of any $\cC$-equivariant stability condition does not send the ray spanned by any (real or imaginary) root to zero.
To do so we need the following lemma.
\begin{lemma}
Let $\tau = (\cP,Z = \underline{Z}\circ \nu) \in \Stab_\cC(\cH)$ and denote $\underline{Z}' \in (\bbR\Lambda)^*_\bbC$ its corresponding image under $\cZ^\dagger_\cC$.
Then for any positive root $\alpha \in \Phi^+$, we have $\underline{Z}'(\alpha) \neq 0$.
\end{lemma}
\begin{proof}
Section 4 of the paper \cite{BDL_root} (written in the generality of 2-CY categories associated to simply-laced Kac--Moody type $\Gamma$) provides an algorithm which explicitly constructs a semistable object of class $\alpha$ for each $\alpha \in \Phi^+$ a positive \emph{real} root.
One can verify that the algorithm goes over verbatim in our setting.
(We do not need -- nor do we have -- uniqueness of the stable object of a particular phase even when $\underline{Z}'$ is generic.)
It follows that $\underline{Z}'(\alpha)$ is non-zero since $\alpha$ is the class of a semistable object.
\end{proof}

The following is the corresponding result on the closed imaginary cone.
\begin{lemma}
Let $\tau = (\cP,Z= \underline{Z}\circ \nu) \in \Stab_\cC(\cH)$ and denote $\underline{Z}' \in (\bbR\Lambda)^*_\bbC$ its image under $\cZ^\dagger_\cC$.
Then for any element $v \neq 0$ in the closed imaginary cone $I$, we have $\underline{Z}'(v) \neq 0$.
\end{lemma}
\begin{proof}
By \cref{prop:imagcone}, $I$ is the convex hull of accumulation rays of the rays of positive real roots.  The support property together with the previous lemma shows that any nonzero element $v$ on any of these accumulation rays
has $\underline{Z}'(v) \neq 0$.  The statement for any $v$ in the imaginary cone now follows by convexity.
\end{proof}

We can now complete the identification of the image of $\cZ^\dagger_\cC$.
\begin{proposition} \label{prop:imageofZ}
We have that $(\cZ^\dagger_{\cC})^{-1}(\Upsilon_{\reg}) = \Stab^\dagger_\cC(\cD)$.
In particular, the image of $\cZ^\dagger_\cC: \Stab^\dagger_\cC(\cD) \ra (\bbR\Lambda)^*_\bbC$ is exactly $\Upsilon_{\reg}$.
\end{proposition}
\begin{proof}
Since $\Upsilon_{\reg}$ is open, $(\cZ^\dagger_{\cC})^{-1}(\Upsilon_{\reg})$ is open in $\Stab^\dagger_\cC(\cD)$.  
It is therefore sufficient to show that $(\cZ^\dagger_{\cC})^{-1}(\Upsilon_{\reg}) \subseteq \Stab^\dagger_\cC(\cD)$ contains all of its limit points.
Any limit point $Z$ that is not in $(\cZ^\dagger_{\cC})^{-1}(\Upsilon_{\reg})$ must send some point $v \in \Phi^+ \cup (I \setminus \{0\})$ to zero.
But the previous two lemmas show that this is impossible.
\end{proof}

With this identification, the following proposition is an immediate strengthening of \cref{lem:orbitnormalisedheart}, which should be compared to Proposition \ref{prop:WandS1actiongenerates}.  
\begin{proposition} \label{prop:staborbitofheart}
Let $\tau \in \Stab^\dagger_\cC(\cD)$.
Then there exists $a \in \bbR$ and $\Psi \in \Br^{\ST}$ such that $a \cdot \Psi \cdot \tau \in \Stab_\cC(\cH)^N$; if $\bbW$ is finite, $a \in \bbR$ can be choosen to be zero (note that $\Stab_\cC(\cH)^N = \Stab_\cC(\cH)$ in this case).
\end{proposition}

\subsection{Covering structures} \label{sec:coveringstructures}
\begin{proposition} \label{prop:faithfulactiononStab}
Let $\Psi \in \Br^{\ST}$.  Suppose $\Psi(\cH) \subset \cH$, where $\cH$ is the heart of linear complexes in $\cD$.
Then $\Psi$ is isomorphic to the identity functor $\mbox{id}_{\cD}$.
\end{proposition}
\begin{proof}
Throughout this proof we denote $A:= \zig$.
    Each $\Psi \in \Br^{\ST}$ is, by definition, a bounded complex of $(A,A)$-bimodules.  
    The proof will follow from an explicit analysis of such complexes.  
    In particular, we write $\Psi$ as a word in the $\cC$-spherical twist generators and their inverses
    \[
    \Psi = \sigma_{P_{i_1}}^{\epsilon_1}\sigma_{P_{i_2}}^{\epsilon_2}\hdots\sigma_{P_{i_k}}^{\epsilon_k}
    \]
    so that the complex $\Psi$ is obtained by tensoring together the two-term complexes of $(A,A)$-bimodules given by:
    \[
    \sigma_{P_i} = (0 \ra P_i\otimes {}_iQ\<1\> \ra A \ra 0), \quad \sigma_{P_i}^{-1} := \sigma_{P_i}' =  (0 \ra A \ra P_i\otimes {}_iQ \<-1\> \ra 0).
    \]
    Recall that our convention is that the identity bimodule $A$ is in homological degree 0 with no internal grading shift; see \eqref{eqn:sphericaltwist} and \eqref{eqn:spehricaltwistinverse}.
    
    What is important here is that the complex $\Psi$ is homotopic to a complex which whose underlying cochain object has
    \begin{itemize}
        \item a single summand of the form $A\<0\>[0]$, and
        \item all other summands isomorphic to some shift of a projective bimodule $P_i\otimes E \otimes {}_jQ[a]\<b\>$ for some simple $E \in \Irr(\cC)$.
    \end{itemize}
    We claim that, amongst all complexes of bimodules of such form, the only one whose action preserves the category $\cH$ of linear complexes is the complex containing only the single summand $A\<0\>[0]$ (and no non-zero differential).

    To establish this claim, we will study a particular $t$-structure on the triangulated category where the bimodule complexes $\Psi$ live, and study the truncations of $\Psi$ in this $t$-structure.  Let $\mathbb{U}$ be the (idempotent complete) additive subcategory of graded bimodules over $A$, generated by the $(A,A)$ bimodules $A\<k\>$ and $P_i\otimes E \otimes {}_jQ\<k\>$, $k\in \Z$, with $E \in \Irr(\cC)$ (see Definition \ref{defn: bimodule category}).  Let $\cD(\mathbb{U}) := \Kom^b(\mathbb{U})$ denote the bounded homotopy category of the additive category $\mathbb{U}$; note that by construction $\Psi$ is an object of $\cD(\mathbb{U})$, and that tensoring (over $A$) with any object of $\cD(\mathbb{U})$ is an endofunctor of $\cD$. 

    The triangulated category $\cD(\mathbb{U})$ has a natural $t$-structure on it, whose heart $\cH(\mathbb{U})$ is the category of linear complexes of bimodules: an object of $\cH(\mathbb{U})$
    is a complex homotopic to one whose cochain objects are all direct sums of bimodules of the form $P_i\otimes E \otimes {}_jQ\<k\>[k]$ for $E \in \Irr(\cC)$ or direct sums of the form $A\<k\>[k]$.  These possible summands, regarded as single-term complexes, are the simple objects of the heart $\cH(\mathbb{U})$.  We call this the \emph{linear $t$-structure} on $\cD(\mathbb{U})$; it is similar to the $t$-structure of $\cD$ described in \S\ref{sec:linearheart}, but now for $(A,A)$-bimodules rather than left $A$-modules.
    
    Let $X\in \cD(\mathbb{U})$ be such that, after possibly replacing $X$ with a homotopic complex,  
    \begin{itemize}
    \item the cochain object of $X$ has exactly one summand of the form $A\<0\>[0]$, and no summands isomorphic to $A\<k\>[\ell]$ for $(k,\ell)\neq (0,0)$
     (the complexes $\Psi \in \Br^{\ST}$ satisfy this condition);
    \item $X\otimes_A Y \subset \cH$ for all $Y\in \cH$.
    \end{itemize}
    
    We claim that $X\cong A\<0\>[0]$ in $\cD(\mathbb{U})$.  To see this, consider the top non-zero cohomology group of $X$, where cohomology is taken with respect to the linear $t$-structure on $\cD(\mathbb{U})$.  If this top cohomology is in degree greater than $0$, then this top cohomology has a simple summand of its head of the form $P_i\otimes E \otimes {}_jQ \<k\>[\ell]$, where $k>\ell$. 
    If we consider the complex of left modules $X\otimes_A P_j$, then the term $P_i\otimes E \otimes {}_jQ \<k\>[\ell] \otimes_A P_j$ in that complex will have a nonzero summand in degree greater than $0$, and cannot Gaussian eliminate.  In other words, the existence of the term $P_i\otimes E \otimes {}_jQ \<k\>[\ell]$, where $k>\ell$, implies that the complex of left modules $X\otimes P_j\in \cD$ has $t$-cohomology (with respect to the linear $t$-structure) in degree greater than $0$, contradicting the assumption that $X \otimes_A \cH \subset \cH$.  
An analogous argument shows that the lowest non-zero cohomology of $X$ cannot be in degree $<0$. 
    
It follows that the complex of bimodules $X$ is in the heart $X\in \cH(\mathbb{U})$ of the t-structure on $\cD(\mathbb{U})$.  This means that that the complex $X$ is a linear complex of bimodules, and is therefore of the form
\[
\cdots \ra 
\oplus_{u,v} P_{u}\otimes M_{u,v} \otimes {}_{v}Q\<-1\> \ra 
A \oplus_{w,x} P_{w}\otimes M_{w,x} \otimes {}_{x}Q \ra 
\oplus_{y,z} P_{y}\otimes M_{y,z} \otimes {}_{z}Q\<1\> \ra 
\cdots ,
\]
where the term $A$ appears in cohomological degree 0 and $M_{?,?}$ are objects in the fusion category $\cC$.
Now we essentially repeat the above argument to conclude that if $X \otimes_A \cH \subset \cH$, then the only term allowed in this complex is the middle term $A\<0\>[0]$:
consider a simple summand of the head of $X$ in the abelian category of linear complexes.  If there is summand of the form $P_i \otimes E \otimes {}_jQ\<k\>[k]$, then $X\otimes P_j$ would have $t$-cohomology (with respect to the linear $t$-structure) of degree $1$, which means that $X \otimes_A -$ wouldn't preserve $\cH$.  Thus the head of $X$ is the term $A\<0\>[0]$.  Analogously, $A\<0\>[0]$ is the only possible simple quotient of $X$.  It follows that $X\cong A\<0\>[0]$, as desired.
\end{proof}

The following variant of \cref{prop:faithfulactiononStab} is not used in what follows, but it is occasionally useful so we record it here.  The proof is essentially identical to the proof above.
\begin{proposition}\label{prop:variant}
Let $\Psi\in \Br^{\ST}$.  Suppose there exists an integer $a$ such that for all $s$,
$$
\Psi(P_s) \cong P_s [a].
$$
Then $\Psi \cong [a]$.
\end{proposition}

Let $[2\bbZ]$ denote the subgroup of $\Aut(\cD)$ generated by the triangulated shift $[2]$.  
\begin{corollary}
The commuting actions of $[2\bbZ]$ and $\Br^{\ST}$ on $\Stab^\dagger_\cC(\cD)$ are free and properly discontinuous.
\end{corollary}
\begin{proof}
It is immediate that the $[2\bbZ]$-action is free and properly discontinuous.

From \cref{prop:faithfulactiononStab}, we see that for all $\Psi \in \Br^{\ST}$ and all $\tau \in \Stab_\cC(\cH)$, if $\Psi \cdot \tau = \tau$ then $\Psi = \id$.
It then follows from \cref{prop:staborbitofheart} that the action of $\Br^{\ST}$ on the whole $\Stab^\dagger_\cC(\cD)$ is free.
From this we can deduce that the $\Br^{\ST}$-action is also proper.
Indeed, $\Psi$ acts locally via its image in $\bbW$, whose action is free and proper; when the action of $\Psi \neq \id \in \Br^{\ST}$ is locally trivial (i.e.\ the image of $\Psi$ in $\bbW$ is the identity), we have that $\Psi\cdot \tau \neq \tau$ while both $\Psi\cdot \tau$ and $\tau$ have the same central charge, so the distance (under the Bridgeland metric) between $\Psi\cdot\tau$ and $\tau$ must be $\geq 1$.
\end{proof}

\begin{remark}\label{rem:Ikeda}
In \cite{ikeda2014stability}, Ikeda considers the action of the spherical twist group $\Br^{\ST}$ on the space of stability conditions on the derived category of modules over the preprojective algebra of a symmetric Kac--Moody quiver.  In loc.\ cit., a proof that the spherical twist group acts faithfully on $\Stab^\dagger(\cD)$ is missing (we thank Michael Weymss for pointing this out to us).  Proposition \ref{prop:faithfulactiononStab} above, transported by Koszul duality to the derived category of preprojective algebras, fills this gap in Ikeda's argument (and generalises it from symmetric Kac--Moody braid groups to the Artin--Tits groups of arbitrary Coxeter systems.)
See also Remark 3.7 in \cite{bridgeland_2009}.  
\end{remark}

Let $\widehat{\Br}^{\ST}$ denote the smallest subgroup of $\Aut(\cD)$ containing $[2]$ and $\Br^{\ST}$.
\begin{proposition} \label{prop:quotientisWquotient}
Let $\underline{\pi}$ denote the composition $\Stab_\cC^\dagger(\cD) \xra{\cZ^\dagger_\cC} \Upsilon_{\reg} \twoheadrightarrow \Upsilon_{\reg}/\bbW$.
Then $\underline{\pi}$ factors through $\Stab_\cC^\dagger(\cD)/\widehat{\Br}^{\ST}$, inducing a homeomorphism $\Stab_\cC^\dagger(\cD)/\widehat{\Br}^{\ST}\cong \Upsilon_{\reg}/\bbW$.
\end{proposition}
\begin{proof}
Since for all $\tau \in \Stab_\cC^\dagger(\cD)$, we have $\cZ^\dagger_\cC(\sigma_{P_s} \cdot \tau) = s \cdot \cZ^\dagger_\cC(\tau)$ for each $s\in \Gamma_0$ (by \cref{lem:STacitononStabreducestoW}), and $\cZ^\dagger_\cC([2]\cdot \tau) = \cZ^\dagger_\cC(\tau)$, it follows that $\underline{\pi}$ factors through $\Stab_\cC^\dagger(\cD)/\widehat{\Br}^{\ST}$.

First suppose $I \neq \{0\}$ (i.e.\ $\bbW$ is infinite).
Let $\tau_1, \tau_2 \in \Stab_\cC^\dagger(\cD)$ be such that $\underline{\pi}(\tau_1) = \underline{\pi}(\tau_2)$.
Using \cref{prop:staborbitofheart}, we may act by an element of $\bbR \times \Br^{\ST}$ on both $\tau_1$ and $\tau_2$ simultaneously
so that $\tau_1 \in \Stab_\cC(\cH)^N$.  So without loss of generality we assume that $\tau_1 \in \Stab_\cC(\cH)^N$ at the outset, and we need to show that $\tau_1$ and $\tau_2$ are in the same $\widehat{\Br}^{\ST}$ orbit.
%Indeed, using \cref{prop:staborbitofheart} we may find $a \in \bbR$ and $\Psi \in \Br^{\ST}$ such that $a \cdot \Psi \cdot \tau_1$ is in $\Stab_\cC(\cH)^N$.
%Then we still have $\underline{\pi}(a \cdot \Psi \cdot \tau_1) = \underline{\pi}(a \cdot \Psi \cdot \tau_2)$.
%So if we show that $a \cdot \Psi \cdot \tau_1 = (a+2m) \cdot (\Psi \Psi') \cdot \tau_2$ for some $m \in \bbZ$ and $\Psi' \in \Br^{\ST}$, then it will follow that $\tau_1 = [2m] \cdot \Psi' \cdot \tau_2$.

Now $\tau_1 \in \Stab_\cC(\cH)^N$ implies that $\cZ^\dagger_\cC(\tau_1) \in C^N$. 
Since $\underline{\pi}(\tau_1) = \underline{\pi}(\tau_2)$ by assumption, we also have $\cZ^\dagger_\cC(\tau_1) = w \cdot \cZ^\dagger_\cC(\tau_2)$ for some $w \in \bbW$.
Applying \cref{prop:staborbitofheart} again, we have that $a_2 \cdot \Psi_2 \cdot \tau_2\in \Stab_\cC(\cH)^N$ for some $a_2 \in \bbR$ and $\Psi_2 \in \Br^{\ST}$.
By \cref{lem:STacitononStabreducestoW}, we get
\[
\cZ^\dagger_\cC(a_2 \cdot \Psi_2 \cdot \tau_2) = e^{ia_2\pi}\cdot [\Psi_2] \cdot \cZ^\dagger_\cC(\tau_2) = e^{ia_2\pi}\cdot ([\Psi_2]w^{-1}) \cdot \cZ^\dagger_\cC(\tau_1)\in C^N.
\]
(Here $[\Psi_2]$ is the image of $\Psi_2$ under $\Br^{\ST} \ra \bbW$ in \cref{prop:STactionagreesW}.)
%Since the $S^1$-action on $\Upsilon_{\reg}$ is free, and moreover $C^N$ is a fundamental domain of the $S^1 \times \bbW$ action on $\Upsilon_{\reg}$.
Since $\cZ^\dagger_\cC(\tau_1)$ is already in $C^N$ and $C^N$ is a fundamental domain of the action of $S^1\times\bbW$ (see \cref{prop:Wactionfreepropdiscont}), it must be that $a_2 = 2m \in 2\cdot \bbZ$ and $[\Psi_2] = w \in \bbW$.
Hence the equation above reduces to
\[
\cZ^\dagger_\cC([2m]\cdot \Psi_2 \cdot \tau_2) = \cZ^\dagger_\cC(\tau_1)\in C^N,
\]
with $[2m]\cdot \Psi_2 \cdot \tau_2$ and $\tau_1$ both in $\Stab_\cC(\cH)^N$.
Since $\cZ^\dagger_\cC$ restricts to a homeomorphism from $\Stab_\cC(\cH)^N$ to $C^N$, it follows that $[2m]\cdot \Psi_2 \cdot \tau_2 = \tau_1$ as required.

For the case $I = \{0\}$ (i.e.\ $\bbW$ is finite), the proof is a simpler version of the above, where we do not need to normalise since $\Stab_\cC(\cH)^N = \Stab_\cC(\cH)$.
\end{proof}

As a consequence, we obtain:
\begin{corollary}
The local homeomorphism $\underline{\pi}: \Stab_\cC^\dagger(\cD) \xra{\cZ^\dagger_\cC} \Upsilon_{\reg} \twoheadrightarrow \Upsilon_{\reg}/\bbW$ is a covering map with group of deck transformations $\widehat{\Br}^{\ST}$.
\end{corollary}

Recall from \cref{prop:STactionagreesW} that there is a surjective group homomorphism $\Br^{\ST} \twoheadrightarrow \bbW$.  We denote its kernel, which we call the pure spherical twist group, by $\PBr^{\ST}$.  Let $\widehat{\PBr}^{\ST}$ denote the smallest subgroup of $\widehat{\Br}^{\ST}$ containing $[2]$ and $\PBr^{\ST}$.

\begin{proposition}
The local homeomorphism $\cZ^\dagger_\cC: \Stab_\cC^\dagger(\cD) \ra \Upsilon_{\reg}$ is a covering map, with group of deck transformations given by $\widehat{\PBr}^{\ST}$.
\end{proposition}
\begin{proof}
By \cref{lem:STacitononStabreducestoW}, the covering map intertwines the $\Br^{\ST}$-action on $\Stab^\dagger_{\cC}(\cD)$ and the $\bbW$-action on $\Upsilon_{\reg}$.
(Note $\Upsilon_{\reg} \subseteq (\bbR\Lambda)^*_\bbC$ is closed under the $\bbW$-action and we have identified the image of $\cZ^\dagger_\cC$ as $\Upsilon_{\reg}$ in \cref{prop:imageofZ}).
Since $\bbW$ acts freely on $\Upsilon_{\reg}$ (and $[2]$ acts on $\Upsilon_{\reg}$ trivially), it follows from \cref{prop:quotientisWquotient} that the smaller quotient $\Stab_\cC^\dagger(\cD)/\widehat{\PBr}^{\ST}$ is exactly $\Upsilon_{\reg}$.
\end{proof}

With this we have completed the proof of \cref{thm:maintheorem}.

\section{Embedding of hyperplane complements and covering spaces} \label{sec:unfolding}
This section involves a pair of related Coxeter systems, so we reintroduce the diagrams into the notation: $\bbW(\Gamma)$, $\B(\Gamma)$, $\Upsilon_{\reg}(\Gamma)$, $\cC(\Gamma)$ etc. 
We will also use $\cD(\Gamma):= \Kom^b(\zig(\Gamma)\lprmod)$.

In the previous section, \cref{thm:maintheorem} says that the submanifold $\Stab_{\cC(\Gamma)}^\dagger(\cD(\Gamma)) \subseteq \Stab^\dagger(\cD(\Gamma))$ is a covering space of the hyperplane complement $\Upsilon_{\reg}(\Gamma)$, with covering map given by the restriction of the local homeomorphism $\cZ(\Gamma): \Stab(\cD(\Gamma)) \ra \Hom_\bbZ(\Lambda(\Gamma), \bbC)$ to $\Stab_{\cC(\Gamma)}^\dagger(\cD(\Gamma))$.
Here we prove that the map $\cZ(\Gamma)$, restricted to the larger manifold $\Stab^\dagger(\cD(\Gamma))$ (see \cref{defn:distinguishedStab}), is also a covering map onto its image, which can be described as the hyperplane complement $\Upsilon_{\reg}(\check{\Gamma})$ of an ``unfolded'' Coxeter graph $\check{\Gamma}$.

\begin{definition} \label{defn:unfolding}
Let $\Gamma$ be a Coxeter graph.
We define its \emph{unfolded} Coxeter graph $\check{\Gamma}$ (with respect to the data \eqref{eqn:coxquiverlabel}) as follows:
\begin{itemize}
\item its set of vertices $\check{\Gamma}_0$ is given by $\Gamma_0 \times \Irr(\cC(\Gamma))$;
\item we have an edge (possibly labelled by $\infty$) between $(s, E)$ and $(t,E')$ if and only if $e:= (s,t)$ is an edge in $\Gamma$ and:
	\begin{enumerate}[(i)]
	\item if $m_{s,t} < \infty$, $E$ appears as a summand in $\Pi(e) \otimes E'$; otherwise
	\item if $m_{s,t} = \infty$, then $E=E'$; in this case the edge between $(s, A)$ and $(t,B)$ is also labelled by $\infty$.
	\end{enumerate}
\end{itemize}
\end{definition}
This unfolding construction agrees with the one in defined in \cite[\S 3.2]{heng2023coxeter}; in particular $\check{\Gamma}$ is finite-type if and only if $\Gamma$ is finite-type (see also \cref{rem:finiteiffunfoldfinite}).
Note that $\check{\Gamma}$ is symmetric Kac--Moody type by construction, and if $\Gamma$ is symmetric Kac--Moody type, then $\check{\Gamma} = \Gamma$.
%Being symmetric Kac--Moody, we could also identify $\zig(\check{\Gamma})$ with the ordinary zigzag algebra defined in \cite{HueKho, khovanov_seidel_2001} by (canonically) identifying $\cC(\check{\Gamma}) = \TLJ^{even}_3$ with $\vec_\bbC$ -- if we wished.

Refer to \cref{fig:unfoldingcoxeter} for some examples of Coxeter graphs and their corresponding unfolded Coxeter graphs; for the unfolding of the finite-type Coxeter graphs, see \cite[Figure 2]{heng2023coxeter}.

\begin{figure}[ht]
\begin{center}
\begin{tabular}{||c | c ||} 
 \hline & \\ [-2ex]
 $\Gamma$ & $\check{\Gamma}$ \\ [1ex] 
 \hline\hline
 $\widetilde{G}_2 =
 \adjustbox{scale=0.8}{
 % https://q.uiver.app/#q=WzAsMyxbMCwwLCJcXGJ1bGxldCJdLFsxLDAsIlxcYnVsbGV0Il0sWzIsMCwiXFxidWxsZXQiXSxbMCwxLCI2IiwwLHsic3R5bGUiOnsiaGVhZCI6eyJuYW1lIjoibm9uZSJ9fX1dLFsxLDIsIiIsMCx7InN0eWxlIjp7ImhlYWQiOnsibmFtZSI6Im5vbmUifX19XV0=
\begin{tikzcd}[every arrow/.append style = {shorten <= -.5em, shorten >= -.5em}]
	\bullet & \bullet & \bullet
	\arrow["6", no head, from=1-1, to=1-2]
	\arrow[no head, from=1-2, to=1-3]
\end{tikzcd}
 }
 $ & 
 $\widetilde{E}_7 \sqcup \widetilde{D}_{6} =  
 \adjustbox{scale=0.8}{
% https://q.uiver.app/#q=WzAsMTUsWzAsNCwiXFxidWxsZXQiXSxbMSwzLCJcXGJ1bGxldCJdLFswLDIsIlxcYnVsbGV0Il0sWzEsMSwiXFxidWxsZXQiXSxbMCwwLCJcXGJ1bGxldCJdLFsxLDAsIlxcYnVsbGV0Il0sWzAsMSwiXFxidWxsZXQiXSxbMSwyLCJcXGJ1bGxldCJdLFswLDMsIlxcYnVsbGV0Il0sWzEsNCwiXFxidWxsZXQiXSxbMiw0LCJcXGJ1bGxldCJdLFsyLDMsIlxcYnVsbGV0Il0sWzIsMiwiXFxidWxsZXQiXSxbMiwxLCJcXGJ1bGxldCJdLFsyLDAsIlxcYnVsbGV0Il0sWzAsMSwiIiwwLHsic3R5bGUiOnsiaGVhZCI6eyJuYW1lIjoibm9uZSJ9fX1dLFsxLDIsIiIsMCx7InN0eWxlIjp7ImhlYWQiOnsibmFtZSI6Im5vbmUifX19XSxbMiwzLCIiLDAseyJzdHlsZSI6eyJoZWFkIjp7Im5hbWUiOiJub25lIn19fV0sWzMsNCwiIiwwLHsic3R5bGUiOnsiaGVhZCI6eyJuYW1lIjoibm9uZSJ9fX1dLFs1LDYsIiIsMCx7InN0eWxlIjp7ImhlYWQiOnsibmFtZSI6Im5vbmUifX19XSxbNiw3LCIiLDAseyJzdHlsZSI6eyJoZWFkIjp7Im5hbWUiOiJub25lIn19fV0sWzcsOCwiIiwwLHsic3R5bGUiOnsiaGVhZCI6eyJuYW1lIjoibm9uZSJ9fX1dLFs4LDksIiIsMCx7InN0eWxlIjp7ImhlYWQiOnsibmFtZSI6Im5vbmUifX19XSxbOSwxMCwiIiwwLHsic3R5bGUiOnsiaGVhZCI6eyJuYW1lIjoibm9uZSJ9fX1dLFsxLDExLCIiLDAseyJzdHlsZSI6eyJoZWFkIjp7Im5hbWUiOiJub25lIn19fV0sWzcsMTIsIiIsMCx7InN0eWxlIjp7ImhlYWQiOnsibmFtZSI6Im5vbmUifX19XSxbMywxMywiIiwwLHsic3R5bGUiOnsiaGVhZCI6eyJuYW1lIjoibm9uZSJ9fX1dLFs1LDE0LCIiLDIseyJzdHlsZSI6eyJoZWFkIjp7Im5hbWUiOiJub25lIn19fV1d
\begin{tikzcd}[every arrow/.append style = {shorten <= -.5em, shorten >= -.5em}]
	\bullet & \bullet & \bullet \\
	\bullet & \bullet & \bullet \\
	\bullet & \bullet & \bullet \\
	\bullet & \bullet & \bullet \\
	\bullet & \bullet & \bullet
	\arrow[no head, from=5-1, to=4-2, crossing over]
	\arrow[no head, from=4-2, to=3-1, crossing over]
	\arrow[no head, from=3-1, to=2-2, crossing over]
	\arrow[no head, from=2-2, to=1-1, crossing over]
	\arrow[no head, from=1-2, to=2-1, crossing over]
	\arrow[no head, from=2-1, to=3-2, crossing over]
	\arrow[no head, from=3-2, to=4-1, crossing over]
	\arrow[no head, from=4-1, to=5-2, crossing over]
	\arrow[no head, from=5-2, to=5-3, crossing over]
	\arrow[no head, from=4-2, to=4-3, crossing over]
	\arrow[no head, from=3-2, to=3-3, crossing over]
	\arrow[no head, from=2-2, to=2-3, crossing over]
	\arrow[no head, from=1-2, to=1-3, crossing over]
\end{tikzcd}
 }$ \\ 
 \hline
$
\adjustbox{scale=0.8}{ 
% https://q.uiver.app/#q=WzAsNCxbMCwwLCJcXGJ1bGxldCJdLFsxLDAsIlxcYnVsbGV0Il0sWzIsMCwiXFxidWxsZXQiXSxbMywwLCJcXGJ1bGxldCJdLFswLDEsIjUiLDAseyJzdHlsZSI6eyJoZWFkIjp7Im5hbWUiOiJub25lIn19fV0sWzEsMiwiNSIsMCx7InN0eWxlIjp7ImhlYWQiOnsibmFtZSI6Im5vbmUifX19XSxbMiwzLCI1IiwwLHsic3R5bGUiOnsiaGVhZCI6eyJuYW1lIjoibm9uZSJ9fX1dXQ==
\begin{tikzcd}[every arrow/.append style = {shorten <= -.5em, shorten >= -.5em}]
	\bullet & \bullet & \bullet & \bullet
	\arrow["5", no head, from=1-1, to=1-2]
	\arrow["5", no head, from=1-2, to=1-3]
	\arrow["\infty", no head, from=1-3, to=1-4]
\end{tikzcd}
}
$ & 
$
 \adjustbox{scale=0.8}{
% https://q.uiver.app/#q=WzAsOCxbMCwzLCJcXGJ1bGxldCJdLFsxLDIsIlxcYnVsbGV0Il0sWzAsMSwiXFxidWxsZXQiXSxbMSwwLCJcXGJ1bGxldCJdLFsyLDEsIlxcYnVsbGV0Il0sWzIsMywiXFxidWxsZXQiXSxbMywxLCJcXGJ1bGxldCJdLFszLDMsIlxcYnVsbGV0Il0sWzAsMSwiIiwwLHsic3R5bGUiOnsiaGVhZCI6eyJuYW1lIjoibm9uZSJ9fX1dLFsxLDIsIiIsMCx7InN0eWxlIjp7ImhlYWQiOnsibmFtZSI6Im5vbmUifX19XSxbMiwzLCIiLDAseyJzdHlsZSI6eyJoZWFkIjp7Im5hbWUiOiJub25lIn19fV0sWzMsNCwiIiwwLHsic3R5bGUiOnsiaGVhZCI6eyJuYW1lIjoibm9uZSJ9fX1dLFs0LDEsIiIsMCx7InN0eWxlIjp7ImhlYWQiOnsibmFtZSI6Im5vbmUifX19XSxbMSw1LCIiLDAseyJzdHlsZSI6eyJoZWFkIjp7Im5hbWUiOiJub25lIn19fV0sWzQsNiwiXFxpbmZ0eSIsMCx7InN0eWxlIjp7ImhlYWQiOnsibmFtZSI6Im5vbmUifX19XSxbNSw3LCJcXGluZnR5IiwwLHsic3R5bGUiOnsiaGVhZCI6eyJuYW1lIjoibm9uZSJ9fX1dXQ==
\begin{tikzcd}[every arrow/.append style = {shorten <= -.5em, shorten >= -.5em}]
	& \bullet \\
	\bullet && \bullet & \bullet \\
	& \bullet \\
	\bullet && \bullet & \bullet
	\arrow[no head, from=4-1, to=3-2]
	\arrow[no head, from=3-2, to=2-1]
	\arrow[no head, from=2-1, to=1-2]
	\arrow[no head, from=1-2, to=2-3]
	\arrow[no head, from=2-3, to=3-2]
	\arrow[no head, from=3-2, to=4-3]
	\arrow["\infty", no head, from=2-3, to=2-4]
	\arrow["\infty", no head, from=4-3, to=4-4]
\end{tikzcd}
 }$ \\
 \hline
$
\adjustbox{scale=0.8}{
% https://q.uiver.app/#q=WzAsMyxbMCwwLCJcXGJ1bGxldCJdLFsxLDAsIlxcYnVsbGV0Il0sWzIsMCwiXFxidWxsZXQiXSxbMCwxLCI0IiwwLHsic3R5bGUiOnsiaGVhZCI6eyJuYW1lIjoibm9uZSJ9fX1dLFsxLDIsIjUiLDAseyJzdHlsZSI6eyJoZWFkIjp7Im5hbWUiOiJub25lIn19fV1d
\begin{tikzcd}[every arrow/.append style = {shorten <= -.5em, shorten >= -.5em}]
	\bullet & \bullet & \bullet
	\arrow["4", no head, from=1-1, to=1-2]
	\arrow["5", no head, from=1-2, to=1-3]
\end{tikzcd}
}
$ &
$
\adjustbox{scale=0.8}{
% https://q.uiver.app/#q=WzAsMTgsWzEsMSwiXFxidWxsZXQiXSxbMiwxLCJcXGJ1bGxldCJdLFswLDAsIlxcYnVsbGV0Il0sWzAsMiwiXFxidWxsZXQiXSxbMywxLCJcXGJ1bGxldCJdLFs0LDAsIlxcYnVsbGV0Il0sWzQsMSwiXFxidWxsZXQiXSxbNCwyLCJcXGJ1bGxldCJdLFs2LDIsIlxcYnVsbGV0Il0sWzYsMSwiXFxidWxsZXQiXSxbNiwwLCJcXGJ1bGxldCJdLFs3LDAsIlxcYnVsbGV0Il0sWzgsMCwiXFxidWxsZXQiXSxbOCwxLCJcXGJ1bGxldCJdLFs4LDIsIlxcYnVsbGV0Il0sWzcsMiwiXFxidWxsZXQiXSxbOSwyLCJcXGJ1bGxldCJdLFs5LDAsIlxcYnVsbGV0Il0sWzAsMSwiIiwwLHsic3R5bGUiOnsiaGVhZCI6eyJuYW1lIjoibm9uZSJ9fX1dLFswLDIsIiIsMix7InN0eWxlIjp7ImhlYWQiOnsibmFtZSI6Im5vbmUifX19XSxbMCwzLCIiLDIseyJzdHlsZSI6eyJoZWFkIjp7Im5hbWUiOiJub25lIn19fV0sWzEsNCwiIiwwLHsic3R5bGUiOnsiaGVhZCI6eyJuYW1lIjoibm9uZSJ9fX1dLFs0LDUsIiIsMCx7InN0eWxlIjp7ImhlYWQiOnsibmFtZSI6Im5vbmUifX19XSxbNCw2LCIiLDAseyJzdHlsZSI6eyJoZWFkIjp7Im5hbWUiOiJub25lIn19fV0sWzQsNywiIiwwLHsic3R5bGUiOnsiaGVhZCI6eyJuYW1lIjoibm9uZSJ9fX1dLFs4LDksIiIsMCx7InN0eWxlIjp7ImhlYWQiOnsibmFtZSI6Im5vbmUifX19XSxbOSwxMCwiIiwwLHsic3R5bGUiOnsiaGVhZCI6eyJuYW1lIjoibm9uZSJ9fX1dLFsxMCwxMSwiIiwwLHsic3R5bGUiOnsiaGVhZCI6eyJuYW1lIjoibm9uZSJ9fX1dLFsxMSwxMiwiIiwwLHsic3R5bGUiOnsiaGVhZCI6eyJuYW1lIjoibm9uZSJ9fX1dLFsxMiwxMywiIiwwLHsic3R5bGUiOnsiaGVhZCI6eyJuYW1lIjoibm9uZSJ9fX1dLFsxMywxNCwiIiwwLHsic3R5bGUiOnsiaGVhZCI6eyJuYW1lIjoibm9uZSJ9fX1dLFs4LDE1LCIiLDIseyJzdHlsZSI6eyJoZWFkIjp7Im5hbWUiOiJub25lIn19fV0sWzE1LDE0LCIiLDIseyJzdHlsZSI6eyJoZWFkIjp7Im5hbWUiOiJub25lIn19fV0sWzE0LDE2LCIiLDIseyJzdHlsZSI6eyJoZWFkIjp7Im5hbWUiOiJub25lIn19fV0sWzEyLDE3LCIiLDIseyJzdHlsZSI6eyJoZWFkIjp7Im5hbWUiOiJub25lIn19fV1d
\begin{tikzcd}[every arrow/.append style = {shorten <= -.5em, shorten >= -.5em}]
	\bullet &&&& \bullet && \bullet & \bullet & \bullet & \bullet \\
	& \bullet & \bullet & \bullet & \bullet && \bullet && \bullet \\
	\bullet &&&& \bullet && \bullet & \bullet & \bullet & \bullet
	\arrow[no head, from=2-2, to=2-3]
	\arrow[no head, from=2-2, to=1-1]
	\arrow[no head, from=2-2, to=3-1]
	\arrow[no head, from=2-3, to=2-4]
	\arrow[no head, from=2-4, to=1-5]
	\arrow[no head, from=2-4, to=2-5]
	\arrow[no head, from=2-4, to=3-5]
	\arrow[no head, from=3-7, to=2-7]
	\arrow[no head, from=2-7, to=1-7]
	\arrow[no head, from=1-7, to=1-8]
	\arrow[no head, from=1-8, to=1-9]
	\arrow[no head, from=1-9, to=2-9]
	\arrow[no head, from=2-9, to=3-9]
	\arrow[no head, from=3-7, to=3-8]
	\arrow[no head, from=3-8, to=3-9]
	\arrow[no head, from=3-9, to=3-10]
	\arrow[no head, from=1-9, to=1-10]
\end{tikzcd}
}
$ \\
\hline
\end{tabular}
\end{center}
\caption{Coxeter graphs $\Gamma$ and their corresponding unfolded Coxeter graphs $\check{\Gamma}$.}
\label{fig:unfoldingcoxeter}
\end{figure}
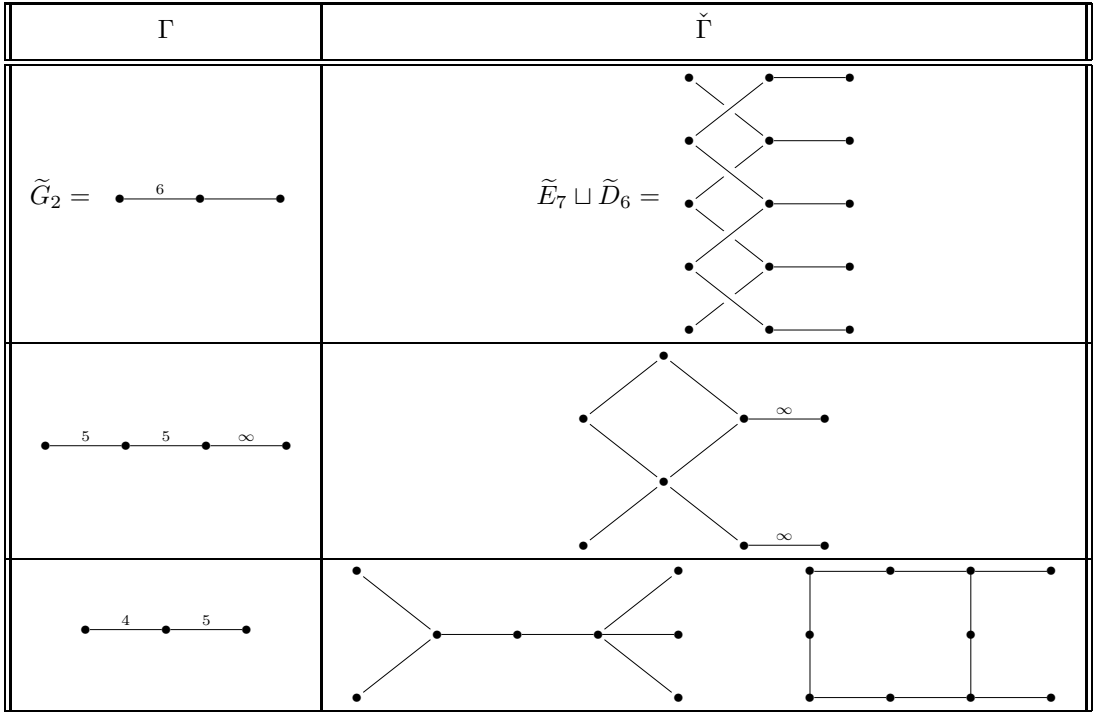

\begin{remark} \label{rem:generalunfolding}
If one associates a different fusion category (and fusion quiver) to the Coxeter diagram following the conditions set out in \cref{rem:fusioncatcoxeter}, then the unfolding needed is the one defined in \cite[Definition 3.1 and 6.2]{EH_fusionquiver}.
\end{remark}

\begin{remark}\label{rem:simplylacedunfolding}
For the edges $e = (s,t) \in \Gamma_1$ with $m_{s,t}=\infty$, one can choose an appropriate fusion category $\cC(\Gamma)$ and an object $\Pi(e) \in \cC(\Gamma)$ so that $\check{\Gamma}$ is simply-laced, with no $\infty$ labels.  The important requirements are that $\FPdim(\Pi(e)) \geq 2$, and that there are no summands of multiplicity greater than one appearing in $E \otimes \Pi(e)$ for any simple $E$ in $\cC(\Gamma)$.
%These are called multiplicity free fusion categories
One such example is given by $\rep(S_3)$ (chosen as the Deligne tensor of $\cC(\Gamma)$ associated to the label $\infty$), where $\Pi(e)$ is taken to be the standard two dimensional irreducible representation.
\end{remark}

We claim that the two categories $\zig(\Gamma)\lprmod$ and $\zig(\check{\Gamma})\lprmod$ are equivalent as additive categories, and hence their bounded homotopy categories are equivalent as triangulated categories.
We emphasise, however, they are different as \emph{module categories}: $\zig(\Gamma)\lprmod$ is a module category over $\cC(\Gamma)$, whereas $\zig(\check{\Gamma})\lprmod$ is a module category over $\cC(\check{\Gamma}) \cong \vec_\bbC$.
\begin{proposition}\label{prop:unfoldzigzagmodule}
We have an equivalence of additive categories
\[
\cU: \zig(\Gamma)\lprmod \xra{\cong} \zig(\check{\Gamma})\lprmod
\]
which is defined on indecomposable objects by $P_s \otimes E\<k\> \mapsto P_{(s,E)}\<k\>$.
In turn, this induces an equivalence of triangulated categories between the bounded homotopy categories
\[
\cU: \cD(\Gamma) \xra{\cong} \cD(\check{\Gamma}).
\]
\end{proposition}
\begin{proof}
Comparing the definition of the unfolded Coxeter graph $\check{\Gamma}$ with the formulae for the hom spaces in \cref{graded hom space}, it follows that the dimensions of the hom spaces between indecomposable objects agree.  Also, by \cref{prop: indecomposable Krull-Sch}, these are all of the indecomposables.
The fact that this agreement of hom spaces can be refined to a functor is a straightforward computation which we leave to the reader; the important point is that the adjunction between $P_s \otimes -$ and ${}_sP\otimes_{\zig(\Gamma)} -$ tells us that a morphism out of $P_s$ is completely determined by its restriction to the summand $e_s$ of $P_s$.
\end{proof}

%\begin{remark}
%A perhaps more highbrow method of the whole unfolding story is as follows.
%Ignoring the grading on $\zig(\Gamma)$ for now, the category of (non-graded) modules $\cM:= \zig(\Gamma)\lmod$ in $\cC(\Gamma)$ is a finite abelian category, and is therefore equivalent to $A_P\lmod$ with $A_P:= \Hom_{\cM}(P,P)$ the ordinary endomorphism algebra of a projective  generator $P$ of $\cM$.
%A non-graded version of proposition \ref{graded hom space} and \ref{prop: indecomposable Krull-Sch} would tell us that we can choose 
%\[
%P := \bigoplus_{s \in \Gamma_0, E \in \Irr(\cC(\Gamma))} P_s \otimes E
%\]
%to be a projective generator.
%With this choice, $A_P$ is isomorphic to the usual zigzag algebras associated to $\check{\Gamma}$.
%From here on the re-consideration of gradings will not be difficult.
%\end{remark}

The equivalence $\cU$ induces a $\bbZ$-linear isomorphism between $K_0(\cD(\Gamma))$ and $K_0(\cD(\check{\Gamma}))$, and also a $\bbZ$-linear isomorphism of their respective lattice quotients: 
\begin{align*}
\cU: \Lambda(\Gamma) &\xra{\cong} \Lambda(\check{\Gamma}) \\
[E]\cdot \alpha_s &\mapsto \alpha_{(s,E)}, \qquad \text{ for each } s \in \Gamma_0, \ E \in \Irr(\cC(\Gamma)).
\end{align*}
We will also abuse notation and use $\cU$ to denote the induced $\bbC$-linear isomorphism on the dual spaces $\cU: \Hom_\bbZ(\Lambda(\Gamma), \bbC) \xra{\cong} \Hom_\bbZ(\Lambda(\check{\Gamma}), \bbC)$.  It should hopefully be clear from the context which isomorphism $\cU$ is referring to. 
Now let
\begin{align*}
\cZ(\Gamma):& \Stab(\cD(\Gamma)) \ra \Hom_\bbZ(\Lambda(\Gamma), \bbC) \\
\cZ(\check{\Gamma}):& \Stab(\cD(\check{\Gamma})) \ra \Hom_\bbZ(\Lambda(\check{\Gamma}), \bbC)
\end{align*} 
denote the local homeomorphisms of the respective Bridgeland stability manifolds.
Let $\Stab^\dagger(\cD(\Gamma))$ and $\Stab^\dagger(\cD(\check{\Gamma}))$ denote the connected components that contain the stability conditions whose standard heart is the linear heart.
The equivalence $\cU: \cD(\Gamma) \xra{\cong} \cD(\check{\Gamma})$ induces a homeomorphism between the components $\cU: \Stab^\dagger(\cD(\Gamma)) \xra{\cong} \Stab^\dagger(\cD(\check{\Gamma}))$.
We denote by $\cZ^\dagger(\Gamma)$ and $\cZ^\dagger(\check{\Gamma})$ the restrictions of $\cZ(\Gamma)$ and $\cZ(\check{\Gamma})$ to $\Stab^\dagger(\cD(\Gamma))$ and $\Stab^\dagger(\cD(\check{\Gamma}))$, respectively.
%\begin{align*}
%\cZ^\dagger(\Gamma):& \Stab^\dagger(\cD(\Gamma)) \ra \Hom_\bbZ(\Lambda(\Gamma), \bbC) \\
%\cZ^\dagger(\check{\Gamma}):& \Stab^\dagger(\cD(\check{\Gamma})) \ra \Hom_\bbZ(\Lambda(\check{\Gamma}), \bbC)
%\end{align*} 
Then there is a commutative diagram
\[
\begin{tikzcd}
\Stab^\dagger(\cD(\Gamma)) \ar[r, "\cU", "\cong"'] \ar[d, "\cZ^\dagger(\Gamma)"]
	& \Stab^\dagger(\cD(\check{\Gamma})) \ar[d, "\cZ^\dagger(\check{\Gamma})"] \\
\Hom_\bbZ(\Lambda(\Gamma), \bbC) \ar[r, "\cU", "\cong"']
	& \Hom_\bbZ(\Lambda(\check{\Gamma}), \bbC).
\end{tikzcd}
\]
Now denote the further restriction of $\cZ(\Gamma)$ to $\Stab_{\cC(\Gamma)}^\dagger(\cD(\Gamma))$ by $\cZ^\dagger_{\cC(\Gamma)}(\Gamma)$.

As before, using the isomorphisms from \cref{prop:identifywithdualspace}:
%\Hom_{K_0(\cC(\Gamma))}(\Lambda(\Gamma), \bbC) &\cong \Hom_\bbR(\bbR\Lambda(\Gamma), \bbC) \\
\[
\Hom_{K_0(\cC(\Gamma))}(\Lambda(\Gamma), \bbC) \cong \Hom_\bbR(\bbR\Lambda(\Gamma), \bbC), \quad 
\Hom_\bbZ(\Lambda(\check{\Gamma}), \bbC) \cong \Hom_\bbR(\bbR\Lambda(\check{\Gamma}), \bbC)
\]
we regard the hyperplane complements as sitting inside the dual spaces of $\Lambda(\Gamma)$ and $\Lambda(\check{\Gamma})$:
\[
\Upsilon_{\reg}(\Gamma) \subseteq \Hom_{K_0(\cC(\Gamma))}(\Lambda(\Gamma), \bbC), \quad
\Upsilon_{\reg}(\check{\Gamma}) \subseteq \Hom_\bbZ(\Lambda(\check{\Gamma}), \bbC).
\]
We obtain the following result.
\begin{theorem}\label{thm:embeddings}
Let $\Gamma$ be a Coxeter graph and $\check{\Gamma}$ be its unfolded Coxeter graph.
We have a commutative diagram where all the vertical maps are covering maps and all horizontal maps are closed embeddings:
\[
\begin{tikzcd}
\Stab_{\cC(\Gamma)}^\dagger(\cD(\Gamma)) \ar[r, hookrightarrow, "\subseteq"] \ar[d, twoheadrightarrow, "\cZ^\dagger_{\cC(\Gamma)}(\Gamma)"]
	& \Stab^\dagger(\cD(\Gamma)) \ar[r, "\cU", "\cong"'] \ar[d, twoheadrightarrow, "\cZ^\dagger(\Gamma)"]	
	& \Stab^\dagger(\cD(\check{\Gamma})) \ar[d, twoheadrightarrow, "\cZ^\dagger(\check{\Gamma})"]	
	\\
\Upsilon_{\reg}(\Gamma) \ar[r, hookrightarrow, "\subseteq"]
	& \img(\cZ^\dagger(\Gamma)) \ar[r, "\cU", "\cong"']
	& \Upsilon_{\reg}(\check{\Gamma}).
\end{tikzcd}
\]
\end{theorem}
\begin{proof}
The commutativity of the diagram follows immediately from the definitions.
The fact that $\cZ^\dagger_{\cC(\Gamma)}(\Gamma)$ and $\cZ^\dagger(\check{\Gamma})$ are covering maps onto the respective hyperplane complements follows from \cref{thm:maintheorem}.
The map $\cZ^\dagger(\Gamma)$ is therefore also a covering map with image homeomorphic to $\Upsilon_{\reg}(\check{\Gamma})$.
The embeddings of hyperplane complements is a closed embedding, and the fact that the embedding $\Stab_{\cC(\Gamma)}^\dagger(\cD(\Gamma)) \xhookrightarrow{\subseteq} \Stab^\dagger(\cD(\Gamma))$ is closed is a general consequence of \cref{thm:closedsubmfld}.
\end{proof}

To describe the embedding $\Upsilon_{\reg}(\Gamma) \hookrightarrow \Upsilon_{\reg}(\check{\Gamma})$ in more detail,
consider the following $\bbC$-linear subspace of $\Hom_\bbZ(\Lambda(\check{\Gamma}), \bbC)$:
\[
\Hom^{\cC(\Gamma)}_\bbZ(\Lambda(\check{\Gamma}),\bbC) := \{ Z \in \Hom_\bbZ(\Lambda(\check{\Gamma}), \bbC) \mid \forall (s,E) \in \check{\Gamma}_0,  Z(\alpha_{(s,E)}) = \FPdim(E)Z(\alpha_{(s,\1)}) \}.
\]
This subspace is the image of $\Hom_{K_0(\cC(\Gamma))}(\Lambda(\Gamma), \bbC) \subseteq \Hom_\bbZ(\Lambda(\Gamma), \bbC)$ under the isomorphism $\cU: \Hom_\bbZ(\Lambda(\Gamma), \bbC) \xra{\cong} \Hom_\bbZ(\Lambda(\check{\Gamma}), \bbC)$, and the image of the embedding $\Upsilon_{\reg}(\Gamma) \hookrightarrow \Upsilon_{\reg}(\check{\Gamma})$ lives in the intersection:
\begin{equation}\label{eqn:intersectionimage}
\Upsilon_{\reg}(\check{\Gamma}) \bigcap \Hom^{\cC(\Gamma)}_\bbZ(\Lambda(\check{\Gamma}),\bbC).
\end{equation}
In general this intersection is not connected; the image of the embedding is the distinguished connected component whose intersection with the complexified chamber $C(\check{\Gamma})$ is non-empty.
In other words, the unique connected component of \eqref{eqn:intersectionimage} that contains
\[
\widetilde{C}(\Gamma) := \{ Z \in \Hom_\bbZ(\Lambda(\check{\Gamma}), \bbC) \mid \forall (s,E) \in \check{\Gamma}_0,  Z(\alpha_{(s,E)}) = \FPdim(E)Z(\alpha_{(s,\1)}) \in \bbH \cup \bbR_{<0} \}
\]
is exactly the image of $\Upsilon_{\reg}(\Gamma)$.
When $\Gamma$ is finite-type, the intersection \eqref{eqn:intersectionimage} is always connected.
\begin{example}\label{eg:pentagonunfold}
Let $\Gamma = I_2(5) = 
	\begin{tikzcd}[every arrow/.append style = {shorten <= -.2em, shorten >= -.2em}]
	s & t
	\arrow["5", no head, from=1-1, to=1-2]
	\end{tikzcd}
$, so that 
\[
\check{\Gamma} = A_4 = 
	\begin{tikzcd}[every arrow/.append style = {shorten <= -.2em, shorten >= -.2em}]
	(s, \Pi_2) & (t, \Pi_2) \\
	(s, \Pi_0) & (t, \Pi_0)
	\arrow[no head, from=1-1, to=1-2]
	\arrow[no head, from=1-1, to=2-2, crossing over]
	\arrow[no head, from=2-1, to=1-2, crossing over]
	\end{tikzcd}
	=:
	\begin{tikzcd}[every arrow/.append style = {shorten <= -.2em, shorten >= -.2em}]
	3 & 2 \\
	1 & 4
	\arrow[no head, from=1-1, to=1-2]
	\arrow[no head, from=1-1, to=2-2, crossing over]
	\arrow[no head, from=2-1, to=1-2, crossing over]
	\end{tikzcd}.
\]
In this case, under the embedding $\Upsilon_{\reg}(I_2(5)) \hookrightarrow \Upsilon_{\reg}(A_4)$, the 10 hyperplanes in $\Upsilon_{\reg}(A_4)$ intersect pair-wise to give 5 hyperplanes in the subspace $\Upsilon_{\reg}(I_2(5))$, which correspond to the 5 reflections of the pentagon. 
See \cref{fig:pentagonunfold} for details.
\end{example}

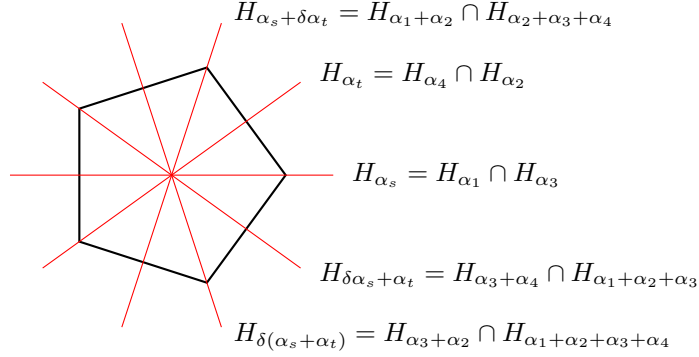
\begin{figure}
\begin{tikzpicture}
% Define the radius of the circumcircle
\def\R{1.5}

% Define scaling of outer circle for reflection lines
\def\L{1.5}

% Draw the regular pentagon using polar coordinates
\draw[thick]
    (0:\R) -- (72:\R) -- (144:\R) -- (216:\R) -- (288:\R) -- cycle;

% Define outer nodes for reflection lines
\foreach [count=\x] \angle in {0, 72, 144, 216, 288} {
	\node (p\x) at (\angle:{\R*\L}) {};
	\node (p'\x) at (180+\angle:{\R*\L}) {};
}

% Draw the 5 reflection lines
\foreach \x in {1,...,5} {
	\draw[red] (p\x) -- (p'\x);
}

% Nodes for hyperplanes
\node[right] at (p1) {$H_{\alpha_s} = H_{\alpha_{1}} \cap H_{\alpha_{3}}$};
\node[right] at (p'4) {$H_{\alpha_t} = H_{\alpha_{4}} \cap H_{\alpha_{2}}$};
\node[right] at (p2) {$H_{\alpha_s + \delta\alpha_t} = H_{\alpha_{1} + \alpha_{2}} \cap H_{\alpha_{2} + \alpha_{3} + \alpha_{4}}$};
\node[right] at (p5) {$H_{\delta(\alpha_s + \alpha_t)} = H_{\alpha_{3} + \alpha_{2}} \cap H_{\alpha_{1} + \alpha_{2} + \alpha_{3} + \alpha_{4}}$};
\node[right] at (p'3) {$H_{\delta\alpha_s + \alpha_t} = H_{\alpha_{3} + \alpha_{4}} \cap H_{\alpha_{1} + \alpha_{2} + \alpha_{3}}$};

\end{tikzpicture}
\caption{The real slice of $\Upsilon_{\reg}(\Gamma) := \Upsilon_{\reg}(I_2(5))$ viewed as a subspace of $\Upsilon_{\reg}(\check{\Gamma}) := \Upsilon_{\reg}(A_4)$. In this subspace, the (removed) 10 reflection hyperplanes of type $A_4$ coincides with the 5 reflection lines of type $I_2(5)$, as mentioned in \cref{eg:pentagonunfold}.}
\label{fig:pentagonunfold}
\end{figure}

Recall that $\bbW(\Gamma)$ acts on $\Lambda(\Gamma)$ by $K_0(\cC(\Gamma))$-linear automorphisms (see \cref{lem:fusionTitsrep}), and hence by  $\bbZ$-linear automorphisms.
Via the $\bbZ$-linear isomorphism $\cU: \Lambda(\Gamma) \xra{\cong} \Lambda(\check{\Gamma})$, we have an action of $\bbW(\Gamma)$ on $\Lambda(\check{\Gamma})$ by $\bbZ$-linear automorphisms.
This action can be described as follows.
\begin{proposition}
The group homomorphism $\psi: \bbW(\Gamma) \ra \bbW(\check{\Gamma})$ defined on the generators by $s \mapsto \prod_{E \in \Irr(\cC(\Gamma))} (s,E)$ is well-defined and injective.
Moreover, $\psi$ together with the $\bbZ$-linear isomorphism $\cU: \Lambda(\Gamma) \xra{\cong} \Lambda(\check{\Gamma})$ intertwine the $\bbW(\Gamma)$-actions on $\Lambda(\Gamma)$ and
$\Hom_\bbZ(\Lambda(\Gamma), \bbC)$ with the $\bbW(\check{\Gamma})$-actions on $\Lambda(\check{\Gamma})$ and $ \Hom_\bbZ(\Lambda(\check{\Gamma}), \bbC)$, respectively.
%Dually, $\psi$ together with the $\bbC$-linear isomorphism $\cU: \Hom_\bbZ(\Lambda(\Gamma), \bbC) \xra{\cong} \Hom_\bbZ(\Lambda(\check{\Gamma}), \bbC)$ also intertwine the respective dual actions.
\end{proposition}
\begin{proof}
Firstly, note that the product $\prod_{E \in \Irr(\cC(\Gamma))} (s,E) \in \bbW(\check{\Gamma})$ is well-defined since the elements $(s,E)$ pair-wise commute for a fixed $s$.
Moreover, the action of $\prod_{E \in \Irr(\cC(\Gamma))} (s,E)$ on $\Lambda(\check{\Gamma})$ on each basis element $\alpha_{(t,F)}$ reduces to
\begin{equation} \label{eqn:unfoldedelement}
\prod_{E \in \Irr(\cC(\Gamma))} (s,E) \cdot \alpha_{(t,F)} = \alpha_{(t,F)} - \sum_{E \in \Irr(\cC(\Gamma))} B(\alpha_{(s,E)}, \alpha_{(t,F)})\alpha_{(s,E)},
\end{equation}
where $B$ is the standard bilinear form on $\Lambda(\check{\Gamma})$ associated to the Coxeter graph $\check{\Gamma}$ (since $\Gamma$ is symmetric Kac--Moody type, $B$ only returns integer output).
On the other hand, the action of $s$ on $\Lambda(\check{\Gamma})$ via conjugating by $\cU: \Lambda(\Gamma) \xra{\cong} \Lambda(\check{\Gamma})$ gives
\begin{equation}  \label{eqn:foldedelement}
s \cdot \alpha_{(t,F)} = 
	\begin{cases}
	-\alpha_{(t,F)}, & \text{if } s = t; \\
	\alpha_{(t,F)} - \cU([\Pi(e) \otimes F] \cdot \alpha_s), & \text{if } e:= (s,t) \in \Gamma_1; \\
	0, & \text{otherwise}.
	\end{cases}
\end{equation}
The fact that \eqref{eqn:unfoldedelement} and \eqref{eqn:foldedelement} agree follows from the definition of the unfolded Coxeter graph $\check{\Gamma}$.

Since both the actions of $\bbW(\Gamma)$ and $\bbW(\check{\Gamma})$ are faithful (by \cref{cor:Wcontragradientfaithful} for $\bbW(\Gamma)$), the homomorphism $\psi:\bbW(\Gamma) \ra \bbW(\check{\Gamma})$ defined by sending $s \mapsto \prod_{E \in \Irr(\cC(\Gamma))} (s,E)$ is indeed well-defined and injective.
\end{proof}
\begin{remark} \label{rem:finiteiffunfoldfinite}
This proposition also implies that $\bbW(\Gamma)$ is finite if $\bbW(\check{\Gamma})$ is finite; the other direction is immediate from \cite[Figure 2]{heng2023coxeter}. %Indeed, for our fusion category $\cC(\Gamma)$, the forward implication is an easy case-by-case analysis; see \cite[Figure 2]{heng2023coxeter}.
%For general fusion categories, we will need the main theorem in \cite{EH_fusionquiver}.
\end{remark}

The proposition above shows that $\cU: \Hom_\bbZ(\Lambda(\Gamma), \bbC) \xra{\cong} \Hom_\bbZ(\Lambda(\check{\Gamma}), \bbC)$ restricted to the corresponding hyperplane complements factors through their respective quotient:
%It also follows that
\[
\underline{\cU}: \Upsilon_{\reg}(\Gamma)/\bbW(\Gamma) \ra \Upsilon_{\reg}(\check{\Gamma})/\bbW(\check{\Gamma})
\]
Moreover, $\underline{\cU}$ is a closed embedding, too, since $\psi$ is injective and both group actions are free.
In turn, we obtain the following relation on the fundamental groups.
\begin{proposition}
Under the isomorphism in \eqref{eqn:fundgrpArtin} of \cref{cor:homotopyequivalence}, the induced homomorphism on fundamental groups 
\[
\underline{\cU}_*: \pi_1(\Upsilon_{\reg}(\Gamma)/\bbW(\Gamma)) \ra \pi_1(\Upsilon_{\reg}(\check{\Gamma})/\bbW(\check{\Gamma}))
\]
is defined on the Artin--Tits generators by $\sigma_s \mapsto \prod_{E \in \Irr(\cC(\Gamma))} \sigma_{(s,E)}$, and sends a generator of the $\bbZ$ factor in the domain to a generator of the $\bbZ$ factor in the codomain. (When $\Gamma$ and $\check{\Gamma}$ are finite-type, there are no such $\bbZ$ factors to consider).
\end{proposition}
\begin{proof}
The statement on the generator of $\bbZ$ is immediate; the main thing to check is that this homomorphism sends the Artin--Tits generator $\sigma_s$ to the product $\prod_{E \in \Irr(\cC(\Gamma))} \sigma_{(s,E)}$. This is probably best checked independently by the reader, who should convince themselves that the image of a loop in the central charge space which winds around the hyperplane corresponding to $s$ gets sent to a loop which winds around each of the (pairwise orthogonal) hyperplanes corresponding to $\{(s,E)\}_{E\in \Irr(\cC(\Gamma))}$ exactly once -- the order of which does not matter since all of them produce homotopy equivalent loops.
\end{proof}

%As such, the map $\widehat{\Psi}: \widehat{\Br}^{\ST}(\Gamma) \ra \widehat{\Br}^{\ST}(\check{\Gamma})$ which sends $\sigma_{P_s} \mapsto \prod_{E \in \Irr(\cC(\Gamma))} \sigma_{P_{(s,E)}}$ and $[2] \mapsto [2]$ is a well-defined group homomorphism, being defined by the commutative diagram:
%\begin{equation} \label{eqn:embeddingswithgroups}
%\begin{tikzcd}
%\widehat{\Br}^{\ST}(\Gamma) \ar[r, dashed, "\widehat{\Psi}"]
%	& \widehat{\Br}^{\ST}(\check{\Gamma}) \\
%\pi_1(\Upsilon_{\reg}(\Gamma)/\bbW(\Gamma)) \ar[u, two heads] \ar[r, "\underline{\cU}_*"]
%	& \pi_1(\Upsilon_{\reg}(\check{\Gamma})/\bbW(\check{\Gamma})) \ar[u, two heads].
%\end{tikzcd}
%\end{equation}

\begin{corollary}\label{cor:unfolding}
We have a commutative diagram, where the vertical maps are covering maps and horizontal maps are closed embeddings:
\[
\begin{tikzcd}
\Stab_{\cC(\Gamma)}^\dagger(\cD(\Gamma)) \ar[r, hook, "\subseteq"] \ar[d, two heads]
	& \Stab^\dagger(\cD(\Gamma)) \ar[d, two heads] \ar[r, phantom, "\cong"]
	&[-25pt] \Stab^\dagger(\cD(\check{\Gamma})) 
		% note option to reduce space between particular columns
	\\
\Upsilon_{\reg}(\Gamma)/\bbW(\Gamma) \ar[r, hook, "\underline{\cU}"]
	& \Upsilon_{\reg}(\check{\Gamma})/\bbW(\check{\Gamma}).
\end{tikzcd}
\]
Moreover, the fundamental groups act by the corresponding fusion-spherical twists groups (together with the triangulated shift $[2]$), and the embedding of covering spaces intertwines the respective actions.
\end{corollary}
%\begin{remark}
%The fact that $\widehat{\Psi}$ in \eqref{eqn:embeddingswithgroups} is a well-defined group homomorphism can also be shown purely algebraically, where one shows that $\sigma_{P_s}$ and $\cU^{-1} \circ \left(\prod_{E \in \Irr(\cC(\Gamma))} \sigma_{P_{(s,E)}} \right) \circ \cU$ are isomorphic as autoequivalences on $\cD(\Gamma)$; we leave this to the interested reader.
%\end{remark}

\subsection{LCM-homomorphisms}\label{sec:LCM}
The map $f: \check{\Gamma} \ra \Gamma$ defined on the vertices by $(s, E) \xmapsto{f} s$ is an (LCM-)folding in the sense of \cite[Definition 4.1]{Crisp99}.
More precisely, for pair of distinct $s, t \in \Gamma$ with $m_{s,t} < \infty$, the full subgraph containing $f^{-1}(s)$ and $f^{-1}(t)$ is a disjoint union of type $A_{m-1}$ Coxeter graphs with $m := m_{s,t}$.
The LCM-homomorphism $\varphi_f: \B(\Gamma) \ra \B(\check{\Gamma})$ associated to $f$ is the homomorphism defined by sending $\sigma_s \mapsto \prod_{s' \in f^{-1}(s)} \sigma_{s'}$; see \cite[\S 1]{Crisp99}.  These were studied in \cite{Crisp99} and have played important roles in the proof of the Tits conjecture \cite{CrispParis_2001} and the proof that the Artin--Tits monoids inject into their groups \cite{Paris_monoid_inject}.

\begin{remark}
If one makes different choices for the original fusion category $\cC(\Gamma)$ (as in \cref{rem:fusioncatcoxeter}, see also \cref{rem:generalunfolding}), the map $f: \check{\Gamma}\ra \Gamma$ above is still a folding.  The full subgraph containing $f^{-1}(s)$ and $f^{-1}(t)$ will now potentially be a disjoint union of different (not necessarily type $A$) Coxeter graphs, but these components will still have the same Coxeter number $m = m_{s,t}$ \cite[Corollary 5.31]{EH_fusionquiver}.  Additionally, as in \cref{rem:simplylacedunfolding}, we can choose $\cC(\Gamma)$ so that $\check{\Gamma}$ is simply--laced.
\end{remark}

Corollary \ref{cor:unfolding} can be viewed as a strengthening of \cite[Theorem 3.4]{Crisp99} for particular foldings.
Namely, it not only gives a geometric realisation of the LCM homomorphism $\varphi_f = \underline{\cU}_*$, but also shows that faithfulness of the $\B(\Gamma)$-action on $\cD(\Gamma)$ implies the injectivity of $\varphi_f$.  This injectivity is open in general, see \cite[Remark pg.\ 16]{CrispParis_2001} and \cite[Question pg.\ 4]{Crisp99}.
%View both $\pi_1$'s as acting on the same triangulated category; say $\cD(\Gamma)$.
In particular, if one shows that the Artin--Tits group $\B(\Gamma)$ acts faithfully on its associated category $\cD(\Gamma)$, it follows that $\B(\Gamma)$ can be embedded into an Artin--Tits group of simply-laced type.

\printbibliography[]

\end{document}